\documentclass[11pt, a4paper]{article}

\usepackage{inputenc} 
\usepackage[english]{babel}
\usepackage{dsfont}
\usepackage{authblk}
\usepackage{subcaption}
\usepackage{placeins}

\usepackage[totalheight=24 true cm, totalwidth=15 true cm]{geometry}
\usepackage{enumitem}

\setlist[itemize]{noitemsep}
\setlist[enumerate]{noitemsep}
\usepackage{bbding}
\usepackage{amsmath,amsthm,amssymb}
\usepackage{mathrsfs}
\usepackage{mathtools}
\usepackage{stmaryrd}
\usepackage{url}
\usepackage{wrapfig}
\usepackage{starfont}
\usepackage{pifont}
\usepackage{eurosym}
\usepackage{subcaption}
\usepackage{setspace}
\usepackage{bbm}
\usepackage{dsfont}
\usepackage{url}
\usepackage{amsmath,blkarray,booktabs, bigstrut}
\usepackage{multirow,multicol}
\usepackage{xcolor}
\usepackage{color}
\usepackage{imakeidx}
\makeindex[]

\usepackage{algorithm}
\usepackage{algpseudocode}
\algrenewcommand\algorithmicrequire{\textbf{Input:}}
\algrenewcommand\algorithmicensure{\textbf{Output:}}
\usepackage{relsize}
\usepackage{float}
\usepackage{tikz,pgf}
\usepackage{tikz-cd}
\usepackage{wasysym}
\usepackage{graphicx}
\usepackage{fancyhdr}
\usepackage{verbatim}

\usepackage{pgfplots}
\pgfplotsset{compat=1.15}
\usetikzlibrary{arrows}
\usepackage{xspace}

\usepackage[all,dvips,knot,web,arc,curve,color,frame]{xy}
\usepackage[all,knot,arc,color,web]{xy}
\xyoption{arc}
\xyoption{web}
\xyoption{curve}

\numberwithin{equation}{section}

\theoremstyle{plain}
\newtheorem{theorem}{Theorem}[section]
\newtheorem{proposition}[theorem]{Proposition}

\newtheorem{lemma}[theorem]{Lemma}

\newtheorem{claim}{Claim}

\theoremstyle{definition}
\newtheorem{definition}[theorem]{Definition}

\newtheorem{remark}[theorem]{Remark}

\newtheorem{example}[theorem]{Example}
\newtheorem{question}{Question}
\newtheorem*{questione}{Question}
\newtheorem{subquestion}{Question}

\newtheorem{conjecture}[theorem]{Conjecture}

\usepackage[bookmarksnumbered=true]{hyperref} 
\hypersetup{
     colorlinks = true,
     linkcolor = blue,
     anchorcolor = blue,
     citecolor = teal,
     filecolor = blue,
     urlcolor = blue
     }
\newcommand\restr[2]{{
  \left.\kern-\nulldelimiterspace 
  #1 
  \vphantom{\big|} 
  \right|_{#2} 
  }}

\everymath{\displaystyle}
\allowdisplaybreaks

\makeatletter
\def\mathcenterto#1#2{\mathclap{\phantom{#1}\mathclap{#2}}\phantom{#1}}
\let\old@widetilde\widetilde
\def\widetildeto#1#2{\mathcenterto{#2}{\old@widetilde{\mathcenterto{#1}{#2\,}}}}
\let\old@widehat\widehat
\def\widehatto#1#2{\mathcenterto{#2}{\old@widehat{\mathcenterto{#1}{#2\,}}}}
\makeatother
\usepackage{graphicx}

\newcommand{\size}[1]{\left| #1 \right|} 

\newcommand{\geqhyp}{\preceq}

\newcommand{\minabove}[1]{\textup{max}_{
\scalebox{.7}{\vspace*{.7em}$\,<$}}\!\left(#1\right)}
\newcommand{\minabovehyp}[1]{\textup{max}_{
\scalebox{.7}{\vspace*{.75em}$\,\preceq$}}\!\left(#1\right)}
\newcommand{\minposet}[1]{\textup{max}\left\{#1\right\}}

\newcommand*\closure[1]{\overline{#1}}

\newcommand{\Ccal}{\mathcal{C}}
\newcommand{\Pcal}{\mathcal{P}}

\newcommand{\ABCD}[2][M]{\mathcal{S}_{#2}(#1)}

\newcommand{\vs}{\vspace*{.5em}}

\DeclareMathOperator{\rank}{rk}

\DeclareMathOperator{\CC}{\mathbb{C}}

\DeclareMathOperator{\Dual}{M_{\textup{Fano}}^{\ast}}
\DeclareMathOperator{\K}{M_{3,3}}

\hypersetup{
    colorlinks=true,
    linkcolor=blue,
    filecolor=magenta,      
    urlcolor=cyan,
    pdftitle={Efficient Algorithms for Maximal Matroid Degenerations and Irreducible Decompositions of Circuit Varieties},
    pdfpagemode=FullScreen,
    }
    
\title{\vspace*{-2em} Efficient Algorithms for Maximal Matroid Degenerations and Irreducible Decompositions of Circuit Varieties}
\author[1]{Emiliano Liwski}
\author[2]{Fatemeh Mohammadi}
\author[3]{Rémi Prébet}

\affil[1,2]{\small Department of Mathematics, KU Leuven, Belgium}
\affil[2]{\small Department of Computer Sciences, KU Leuven, Belgium}
\affil[3]{\small Inria, CNRS, ENS de Lyon, Université Claude Bernard Lyon 1, LIP, UMR 5668, 69342, Lyon cedex 07, France}

\date{\vspace*{-.5em}\today\vspace*{-1.5em}}

\begin{document}
\maketitle

\begin{abstract}
\noindent Matroid theory provides a unifying framework for studying dependence across combinatorics, geometry, and applications ranging from rigidity to statistics. In this work, we study circuit varieties of matroids, defined by their minimal dependencies, which play a central role in modeling determinantal varieties, rigidity problems, and conditional independence relations. We introduce an efficient computational strategy for decomposing the circuit variety of a given matroid $M$, based on an algorithm that identifies its maximal degenerations. These degenerations correspond to the largest matroids lying below $M$ in the weak order. Our framework yields explicit and computable decompositions of circuit varieties that were previously out of reach for symbolic or numerical algebra systems. We apply our strategy to several classical configurations, including the Vámos matroid, the unique Steiner quadruple system $S(3,4,8)$, projective and affine planes, the dual of the Fano matroid, and the dual of the graphic matroid of $K_{3,3}$. In each case, we successfully compute the minimal irreducible decomposition of their circuit varieties.
\end{abstract}

\section{Introduction}

\subsection{Motivation} 
A matroid provides a combinatorial framework for capturing linear dependence in vector spaces~\cite{whitney1992abstract, Oxley, piff1970vector}. 
Given a finite collection of vectors, the linearly dependent subsets determine a matroid. 
When this process can be reversed, meaning that a given matroid $M$ corresponds to a collection of vectors, we refer to such a collection as a realization of~$M$. 

Let $M$ be a matroid of rank $n$ on $[d]$. 
A realization of $M$ is a tuple of vectors $\gamma=(\gamma_{1},\ldots,\gamma_{d}) \in \CC^{n}$ such that
    $\{i_{1},\ldots,i_{p}\}$ is dependent in $M$ 
    if and only if 
    $\{\gamma_{i_{1}},\ldots,\gamma_{i_{p}}\}$ is linearly dependent.
Equivalently, $\gamma$ may be regarded as an $n \times d$ matrix with columns $\gamma_{1},\ldots,\gamma_{d}$.  
The \emph{realization space} of $M$ is
\begin{eqnarray}\label{def: real}
\Gamma_{M} \;=\; \{ \gamma \in \CC^{nd} : \text{the column matroid of $\gamma$ realizes $M$} \},
\end{eqnarray}
where $\CC^{nd}$ has coordinate ring $S = \CC[x_{ij} \mid 1 \leq i \leq n,\ 1 \leq j \leq d]$.  
Thus a point of $\Gamma_M$ may be viewed either as a matrix in $\CC^{n \times d}$ or as a tuple of $d$ vectors in $\CC^{n}$.  
A matroid is called \emph{realizable} if $\Gamma_{M}$ is nonempty.  
The \emph{matroid variety} $V_{M}$ is the Zariski closure of $\Gamma_{M}$ in $\CC^{nd}$.

This notion of matroid variety differs from the more common one where varieties are embedded in the Grassmannian $\mathrm{Gr}(d,n)$.
Introduced in \cite{gelfand1987combinatorial}, matroid varieties in the affine setting have been extensively studied~\cite{clarke2021matroid, Vakil, sidman2021geometric, liwski2025solvable, sturmfels1989matroid, feher2012equivariant, knutson2013positroid}.  

In this work, we focus on the \emph{circuit variety} $V_{\Ccal(M)}$ associated with a matroid $M$, where $\Ccal(M)$ denotes its set of circuits, i.e., the minimal dependent sets.

\begin{definition}\normalfont\label{cir} 
Let $M$ be a matroid of rank $n$ on $[d]$. 
Consider the $n \times d$ matrix $X = (x_{i,j})$ of indeterminates. 
The \emph{circuit ideal} of $M$ is
\[
 I_{\Ccal(M)} \;=\; \{ [A|B]_X : B \in \Ccal(M),\ A \subseteq [n],\ |A| = |B| \}\subset S,
\]
where $[A|B]_X$ denotes the minor of $X$ with row set $A$ and column set $B$.  

A tuple of vectors $\gamma = (\gamma_{1}, \ldots, \gamma_{d}) \in \CC^{n}$ is said to \emph{include the dependencies} of $M$ if
\[
\{i_{1}, \ldots, i_{k}\}\ \text{is dependent in $M$} 
\;\;\Longrightarrow\;\; \{\gamma_{i_{1}}, \ldots, \gamma_{i_{k}}\}\ \text{is linearly dependent}.
\]  

The \emph{circuit variety} of $M$ is
\[
V_{\Ccal(M)} \;=\; V(I_{\Ccal(M)}) 
= \{ \gamma \in \CC^{nd} : \text{$\gamma$ includes the dependencies of $M$}\}.
\]
\end{definition}



Note that the circuit variety contains the matroid variety, and then the realization space. 
Our guiding question is the following:

\begin{question}\label{mainquest}
 How to decompose circuit varieties into irreducible, non-redundant varieties?
\end{question}

While matroid varieties are mainly studied in algebraic geometry for their rich structure, circuit varieties and their decompositions arise more naturally in applications such as determinantal varieties \cite{bruns2003determinantal, clarke2020conditional, clarke2021matroid, herzog2010binomial, pfister2019primary, ene2013determinantal}, rigidity theory \cite{jackson2024maximal, whiteley1996some, graver1993combinatorial, maximum}, and conditional independence models \cite{Studeny05:Probabilistic_CI_structures, DrtonSturmfelsSullivant09:Algebraic_Statistics, hocsten2004ideals, clarke2020conditional, clarke2022conditional, caines2022lattice, mohammadi2018prime}. Circuit varieties may strictly contain matroid varieties, and when the latter are irreducible, they appear as components in the decomposition of the former. This makes circuit varieties especially relevant in contexts centered on minimal matroid dependencies.

This problem is highly challenging. For example, \cite{pfister2019primary} presents an algorithm for decomposing the circuit variety of the $3 \times 4$ grid configuration, which has 16 circuits of size 3. Using {\tt Singular}, they showed that this variety has two components. However, the computations push the limits of current algebra systems, and the resulting components lack a combinatorial interpretation. In contrast, the methods developed here apply to broader classes of matroids and yield a combinatorial and geometric description of the decomposition.

\medskip
Note that, to answer Question~\ref{mainquest}, one could determine the primary decomposition of the associated circuit ideal. However, this would include identifying defining equations of matroid varieties, a very difficult problem resolved only in a few cases \cite{liwski2025solvable,pfister2019primary}. Because of these difficulties, our focus here is instead on decomposing the circuit variety itself. 


\subsection{Outline and our results}

In this work, we introduce an efficient method for computing the irreducible decomposition of circuit varieties, utilizing a new algorithm that identifies its maximal matroid degenerations. We now explain this concept. Consider matroids defined on a common ground set, ordered by the dependency relation $M \geq N$, where every dependent set of $M$ is also dependent in $N$. This ordering is known as the {\em weak order}~\cite{Oxley}. Our primary focus is on identifying the maximal matroid degenerations of a given matroid $M$, which are the largest matroids that are less than $M$ in the weak order. 

\medskip\noindent{\bf Strategy outline.} Our strategy to decomposing the circuit variety of $M$ begins with a reduction to smaller circuit varieties (Proposition~\ref{deco circ}), each associated with a maximal matroid degeneration of $M$. For each such variety, we determine whether it is irreducible; if not, we recursively apply the same decomposition process. 
A key ingredient in this strategy is thus an algorithm for identifying the maximal matroid degenerations of $M$; see Section~\ref{general rank} and Algorithm~\ref{211}. 
The algorithm relies on the submodularity of the rank function to derive rank constraints on various subsets, thereby identifying forced dependencies and providing a natural termination condition.




\medskip\noindent{\bf Limitations.} We now comment on why decomposing the circuit variety $V_{\mathcal{C}(M)}$ is difficult.

\smallskip
\noindent $\blacktriangleright$ The \textbf{initial step} involves determining the maximal matroid degenerations of $M$, which requires constructing a poset of potential candidates and developing methods to prune the search set effectively. From the standpoint of enumerative combinatorics, this is a challenging problem. A key difficulty lies in the fact that introducing a single dependency, and subsequently all those enforced by it, can result in the same matroid arising from many different initial choices, complicating the enumeration.
This relates to a classical question in rigidity theory to determine when a given family of matroids has a unique maximal element. This is particularly relevant in the study of maximal abstract rigidity matroids \cite{whiteley1996some,graver91rigidity,graver1993combinatorial}, maximal $H$-matroids \cite{maximum}, and $\mathcal{X}$-matroids~\cite{jackson2024maximal}. 

\smallskip

\noindent $\blacktriangleright$ The \textbf{second step} involves addressing the geometric aspects required to decompose $V_{\mathcal{C}(M)}$. Once the maximal matroid degenerations of $M$ have been identified, one must determine which ones give rise to irreducible circuit varieties. This is a subtle and generally difficult problem~\cite{clarke2021matroid,Vakil,sidman2021geometric,sturmfels1989matroid}. For those matroids whose associated circuit variety is reducible, the decomposition process must be recursively applied, further increasing the complexity of the problem. 
Iterating through the algorithm multiple times leads to considerable redundancy, as the same matroid may appear repeatedly, arising from distinct sequences of added dependencies. This creates a significant combinatorial and enumeration challenge. One must determine whether newly obtained matroids are isomorphic to any of those which have already been seen, in which case further iterations from that point can be avoided. This problem extends the classical graph isomorphism problem~\cite{babai2016graph}.

\medskip\noindent{\bf Effectiveness.} In practice, when applying our algorithm to large families of examples, we have observed that the maximal matroid degenerations that arise are very structured, such as being nilpotent or inductively connected (see \cite{Fatemeh5} and Definitions~\ref{nilpotent} and~\ref{induc}). For these families, we can apply Theorem~\ref{teo ir} to decompose their associated circuit varieties, hence not needing to run the algorithm repeatedly. In particular, when the original matroid has a large automorphism group, indicating a high degree of symmetry, the resulting maximal degenerations fall into fewer classes, which are also classified up to symmetry. For such matroids, the algorithm performs especially efficiently, as shown in Section~\ref{examples}. For example, matroids from Steiner systems have such symmetries. Moreover, for the subfamily arising from affine and projective planes, we conjecture in Section~\ref{furt}, based on our computations, that exactly four types of maximal matroid degenerations occur.

\medskip

\noindent{\bf Termination.}
A general termination criterion for our strategy is extremely difficult to obtain.
First, the poset of matroid degenerations below a fixed matroid can grow
super-exponentially, even in rank four.  Distinct initial choices of additional
dependencies often lead to the \emph{same} degeneration after closure under
submodularity, making it hard to determine whether the degeneration poset is
finite.  In fact, this already subsumes long-standing open problems in rigidity
theory, such as the uniqueness of maximal abstract rigidity matroids and $\mathcal{X}$-matroids.
Second, completeness of the decomposition requires detecting when two
degenerations are isomorphic.  This task is at least as hard as the graph
isomorphism problem, and no effective structural classification is known.  For
this reason, a full characterization of the matroids for which the algorithm is
guaranteed to terminate appears to be beyond the reach of current theory.

\medskip
Despite these intrinsic limitations, the algorithm performs very well in
practice.  
To further illustrate the effectiveness of our approach, we apply it to decompose the circuit varieties of several rank-four matroids, including the Vámos matroid, the unique Steiner system $S(3,4,8)$, the Fano dual, and the dual of the graphic matroid $M(K_{3,3})$. 

\begin{example}
Consider the graphic matroid $M(K_{3,3})$ associated with the bipartite graph $K_{3,3}$, and let $\K$ denote its dual. In Subsection~\ref{k33}, we show the following.
\begin{itemize}
\item The matroid $\K$ has exactly $34$ maximal matroid degenerations.
\item The circuit variety of $\K$ has precisely two minimal components: the matroid variety of $\K$ itself, and that of its truncation, known as the $3\times 3$ grid. 
\end{itemize}
\end{example}

\noindent These decompositions were previously unknown, and current symbolic or numerical algebra systems cannot perform the necessary computations. Our decomposition strategy also connects to conditional independence models in statistics~\cite{clarke2021matroid,ollie_fatemeh_harshit}. An open-source, Python-optimized implementation of our algorithms is available at:      
\url{https://github.com/rprebet/maximal_matroids}.

\medskip
\noindent
{\bf Outline.} Section~\ref{sec 2} provides an overview of key concepts, including matroids and their realization spaces. In Section~\ref{sec 3}, we introduce the notion of maximal matroid degenerations and provide a decomposition strategy for computing the irreducible components of circuit varieties, which relies on an algorithm for identifying maximal degenerations, detailed in Section~\ref{general rank}. Section~\ref{algorithm} presents an optimized version of this algorithm for rank-four matroids and its implementation. In Section~\ref{examples}, we apply this strategy to compute the irreducible decompositions of circuit varieties for several classical rank-four matroids. In Section~\ref{furt}, we formulate a conjecture on maximal matroid degenerations of affine and projective planes of arbitrary order. Finally, Section~\ref{appen} discusses techniques for identifying redundant matroid varieties and provides proofs for the technical lemmas in Section~\ref{examples}.

\medskip
\noindent{\bf Acknowledgement.} 
F.M. would like to thank Hugues Verdure for helpful discussions. She also gratefully acknowledges the hospitality of the Mathematics Department at Stockholm University during her research visit to Samuel Lundqvist, where part of this work was carried out. The authors were partially supported by the FWO grants G0F5921N (Odysseus) and G023721N, and the grant iBOF/23/064 from KU Leuven. E.L. was supported by PhD fellowship 1126125N.

\vspace{-3mm}

\section{Preliminaries}\label{sec 2}

We begin by recalling key properties of matroids; see \cite{Oxley, gelfand1987combinatorial, liwski2025solvable, Fatemeh5} for details. For positive integers $d,n$, we write $[d] = \{1,\ldots,d\}$ and $\textstyle \binom{[d]}{n}$ for the set of $n$-element subsets of $[d]$.

\vspace{-4mm}

\subsection{Matroids}

We first present some preliminary results about matroids; see \cite{Oxley, piff1970vector} for more details. 
Recall first the definition of a matroid in terms of its independent sets.

\begin{definition}\label{def:matroidcircuit}\normalfont 
A matroid $M$ consists of a ground set $[d]$ together with a collection $\mathcal{I}(M)$ of subsets of $[d]$, called \emph{independent sets}, satisfying: 
\begin{enumerate}[label=(\roman*)]
\item $\emptyset \in \mathcal{I}$;\vs
\item if $I \in \mathcal{I}$ and $I' \subseteq I$, then $I' \in \mathcal{I}$;\vs
\item if $I_1, I_2 \in \mathcal{I}$ with $|I_1| < |I_2|$, then there exists $e \in I_2 \setminus I_1$ such that $I_1 \cup \{e\} \in \mathcal{I}$.
\end{enumerate}
\end{definition}

There are multiple equivalent ways to define a matroid, including descriptions in terms of independent sets, the rank function, or bases. We now introduce these concepts and refer to \cite{Oxley} for a comprehensive discussion on these equivalent definitions.

\begin{definition}\label{def:dependant}\normalfont
Let $M$ be a matroid on the ground set $[d]$. 
\vspace*{-.5em}

\hspace*{-1cm}\parbox{\linewidth}{
\begin{itemize}[label=$\blacktriangleright$]

\item A subset of $[d]$ that is not independent is called \emph{dependent}. The set of all dependent sets of $M$ is denoted by $\mathcal{D}(M)$.\vs

\item A \emph{circuit} is a minimally dependent subset of $[d]$. The set of all circuits is denoted by $\mathcal{C}(M)$.
We denote by $\Ccal_{i}(M)$ the collection of circuits of size $i$ of $M$.\vs
\item The \emph{rank} of a subset $F \subseteq [d]$, denoted $\rank(F)$, is the size of the largest independent set contained in $F$. The rank of $M$, denoted $\rank(M)$, is defined to be the rank of $[d]$.
\item Let $F\subset [d]$. The {\em closure} $\closure{F}$ of $F$, is the set of all $x\in [d]$ such that $\rank(F\cup \{x\})=\rank(F)$.\vs
\item $F$ is called a {\em flat} if $F=\closure{F}$, and is a \emph{cyclic flat} if it is also a union of circuits.\vs

\item Let $x\in [d]$, if $\rank(\{x\})=0$ then $x$ is called a {\em loop}. Conversely, if $x$ is a {\em coloop} if it does not belong to any circuit of $M$.
A subset $\{x,y\}\subset [d]$ is called a {\em double point} if $\rank(\{x,y\})=1$. Finally, a matroid without loops or double points is called {\em simple}.
\end{itemize}
}
\end{definition}

\begin{proposition}[\textup{\cite[Lemma~1.3.1]{Oxley}}]\label{prop:submodularity}
The rank function of a matroid is submodular, meaning that for any subsets $A$ and $B$ of the ground set, the following inequality holds:
\[\rank(A)+\rank(B)\geq \rank(A\cup B)+\rank(A\cap B).\]
\end{proposition}

We now review the concepts of {\em restriction} and {\em deletion}: 

\begin{definition}
\normalfont \label{subm}
Let $M$ be a matroid of rank $n$ on the ground set $[d]$ and $S\subseteq [d]$. 

{
\begin{itemize}[label=$\blacktriangleright$]
\item The {\em restriction of $M$ to $S$} is the matroid 
 on $S$ whose rank function is given by

 \[
  \rank(A)=\rank_{M}(A) \text{\quad for any $A\subset S$},
\] 
where $\rank_{M}$ is the rank function on $M$. This matroid is called a {\em submatroid} of $M$ and is denoted by $M|S$, or simply $S$ when the context is clear.\vs
\item The {\em deletion} of $S$, denoted $M\setminus S$, corresponds to the restriction
$M|([d]\setminus S)$.\vs
\end{itemize}
}
\end{definition}

\begin{definition}\normalfont\label{uniform 3}
The uniform matroid $U_{n, d}$ on the ground set $[d]$ of rank $n$ is  
the one whose independent sets are the subsets of size at most $n$. See Figure~\ref{uniform}.
\end{definition}

\begin{figure}[H]
    \centering
    \begin{subfigure}[b]{0.28\textwidth}
    \centering
    \includegraphics[width=\textwidth, trim=0pt 0pt 180pt 0pt, clip]{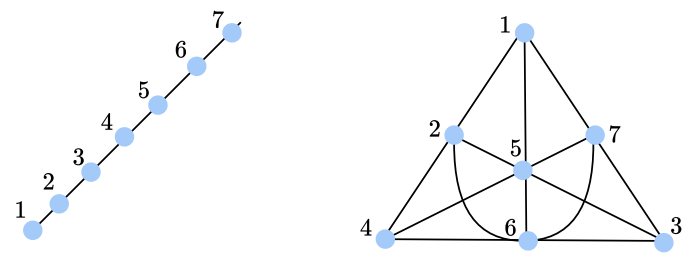}
    \caption{Uniform matroid $U_{2,7}$}
    \label{uniform}
    \end{subfigure}
    \hspace*{1cm}
    \begin{subfigure}[b]{0.28\textwidth}
        \centering
        \includegraphics[width=\textwidth, trim=180pt 0pt 0pt 0pt, clip]{Figure1.png}
        \caption{Fano plane}
        \label{fig:right}
        \label{fano}
    \end{subfigure}
    \caption{Uniform matroid $U_{2,7}$ and Fano plane.}
    \label{figure uniform and fano}
\end{figure}

\begin{definition}\normalfont\label{general}
Let $M$ be a matroid of rank $n$ on $[d]$, with elements, referred to as points. We define an equivalence relation on the circuits of $M$ of size less than $n+1$: 
\begin{equation}\label{equiv}C_{1}\sim C_{2} \Longleftrightarrow \closure{C_{1}}=\closure{C_{2}}.\end{equation}
We adopt the following terminology and notation.\vspace*{-.5em}

\hspace*{-1cm}\parbox{\linewidth}{
\begin{itemize}[label=$\blacktriangleright$] 
\item A {\em subspace} of $M$ is an equivalence class $l$. We say that $\rank(l)=k$ if $\rank(C)=k$ for any circuit $C\in l$. We denote by $\mathcal{L}_{M}$ the set of all subspaces of $M$.\vs
\item A point $p\in [d]$ is said to belong to the subspace $l$, if $p\in C$ for some circuit $C\in l$. For each $p\in [d]$, let $\mathcal{L}_{p}$ denote the set of all the subspaces of $M$ containing $p$. 
The {\em degree} of $p$, is defined as $\deg(p) = \size{\mathcal{L}_{p}}$. 
\end{itemize} 
}
\end{definition}

\begin{example}
Consider the quadrilateral set configuration $\text{QS}$ shown in Figure~\ref{Quadrilateral set}. This represents a rank-$3$ matroid on $[6]$, with the following circuits of size at most three: 
\[\Ccal = \{\{1,2,3\}, \{1,5,6\}, \{3,4,5\}, \{2,4,6\}\}.\]
The subspaces of $\text{QS}$ coincide with $\Ccal$, and each point has degree two.
\end{example}

\begin{figure}[H]
    \centering
    \begin{subfigure}[b]{0.25\textwidth}
        \centering
        \begin{tikzpicture}[x=0.75pt,y=0.75pt,yscale=-1,xscale=1]

\draw [line width=0.75]    (247.65,123.92+60) -- (201.96,207.53+60) ;
\draw [line width=0.75]    (247.65,123.92+60) -- (291.68,207.53+60) ;
\draw [line width=0.75]    (219.23,174.98+60) -- (291.68,207.53+60) ;
\draw [line width=0.75]    (274.41,174.98+60) -- (201.96,207.53+60) ;
\draw [fill={rgb, 255:red, 173; green, 216; blue, 230}, fill opacity=1] (244.87,123.92+60) .. controls (244.87,125.68+60) and (246.12,127.11+60) .. (247.65,127.11+60) .. controls (249.19,127.11+60) and (250.44,125.68+60) .. (250.44,123.92+60) .. controls (250.44,122.15+60) and (249.19,120.73+60) .. (247.65,120.73+60) .. controls (246.12,120.73+60) and (244.87,122.15+60) .. (244.87,123.92+60) -- cycle ;
\draw [fill={rgb, 255:red, 173; green, 216; blue, 230}, fill opacity=1] (244.31,187.1+60) .. controls (244.31,188.87+60) and (245.56,190.29+60) .. (247.1,190.29+60) .. controls (248.64,190.29+60) and (249.88,188.87+60) .. (249.88,187.1+60) .. controls (249.88,185.34+60) and (248.64,183.91+60) .. (247.1,183.91+60) .. controls (245.56,183.91+60) and (244.31,185.34+60) .. (244.31,187.1+60) -- cycle ;
\draw [fill={rgb, 255:red, 173; green, 216; blue, 230}, fill opacity=1] (216.45,174.98+60) .. controls (216.45,176.74+60) and (217.69,178.17+60) .. (219.23,178.17+60) .. controls (220.77,178.17+60) and (222.02,176.74+60) .. (222.02,174.98+60) .. controls (222.02,173.21+60) and (220.77,171.79+60) .. (219.23,171.79+60) .. controls (217.69,171.79+60) and (216.45,173.21+60) .. (216.45,174.98+60) -- cycle ;
\draw [fill={rgb, 255:red, 173; green, 216; blue, 230}, fill opacity=1] (271.62,174.98+60) .. controls (271.62,176.74+60) and (272.87,178.17+60) .. (274.41,178.17+60) .. controls (275.94,178.17+60) and (277.19,176.74+60) .. (277.19,174.98+60) .. controls (277.19,173.21+60) and (275.94,171.79+60) .. (274.41,171.79+60) .. controls (272.87,171.79+60) and (271.62,173.21+60) .. (271.62,174.98+60) -- cycle ;
\draw [fill={rgb, 255:red, 173; green, 216; blue, 230}, fill opacity=1] (288.9,207.53+60) .. controls (288.9,209.29+60) and (290.14,210.72+60) .. (291.68,210.72+60) .. controls (293.22,210.72+60) and (294.47,209.29+60) .. (294.47,207.53+60) .. controls (294.47,205.77+60) and (293.22,204.34+60) .. (291.68,204.34+60) .. controls (290.14,204.34+60) and (288.9,205.77+60) .. (288.9,207.53+60) -- cycle ;
\draw [fill={rgb, 255:red, 173; green, 216; blue, 230}, fill opacity=1] (199.17,207.53+60) .. controls (199.17,209.29+60) and (200.42,210.72+60) .. (201.96,210.72+60) .. controls (203.49,210.72+60) and (204.74,209.29+60) .. (204.74,207.53+60) .. controls (204.74,205.77+60) and (203.49,204.34+60) .. (201.96,204.34+60) .. controls (200.42,204.34+60) and (199.17,205.77+60) .. (199.17,207.53+60) -- cycle ;

\draw (243.29,107.07+60) node [anchor=north west][inner sep=0.75pt]   [align=left] {{\scriptsize $1$}};
\draw (200.77,165.67+60) node [anchor=north west][inner sep=0.75pt]   [align=left] {{\scriptsize $2$}};
\draw (184.94,205.77+60) node [anchor=north west][inner sep=0.75pt]   [align=left] {{\scriptsize $3$}};
\draw (241.88,169.35+60) node [anchor=north west][inner sep=0.75pt]   [align=left] {{\scriptsize $4$}};
\draw (281.98,164.77+60) node [anchor=north west][inner sep=0.75pt]   [align=left] {{\scriptsize $5$}};
\draw (299.84,206.21+60) node [anchor=north west][inner sep=0.75pt]   [align=left] {{\scriptsize $6$}};
\end{tikzpicture}
        \caption{Quadrilateral set}
      \label{Quadrilateral set}
    \end{subfigure}
    \hspace*{1.5cm}
    \begin{subfigure}[b]{0.35\textwidth}
        \centering
\tikzset{every picture/.style={line width=0.75pt}} 
        \begin{tikzpicture}[x=0.75pt,y=0.75pt,yscale=-1,xscale=1]

\tikzset{every picture/.style={line width=0.75pt}} 

\draw    (81.69,116.61) -- (191.16,174.3) ;
\draw    (77,131.88) -- (224,131.88) ;
\draw    (80.13,150.55) -- (191.16,79.27) ;
\draw [fill={rgb, 255:red, 173; green, 216; blue, 230}, fill opacity=1]
(107.34,131.2) .. controls (107.34,133.07) and (108.63,134.58) .. (110.23,134.58) .. controls (111.83,134.58) and (113.12,133.07) .. (113.12,131.2) .. controls (113.12,129.33) and (111.83,127.81) .. (110.23,127.81) .. controls (108.63,127.81) and (107.34,129.33) .. (107.34,131.2) -- cycle ;
\draw [fill={rgb, 255:red, 173; green, 216; blue, 230}, fill opacity=1]
(142.62,149.93) .. controls (142.62,151.8) and (143.91,153.32) .. (145.51,153.32) .. controls (147.11,153.32) and (148.4,151.8) .. (148.4,149.93) .. controls (148.4,148.06) and (147.11,146.55) .. (145.51,146.55) .. controls (143.91,146.55) and (142.62,148.06) .. (142.62,149.93) -- cycle ;
\draw [fill={rgb, 255:red, 173; green, 216; blue, 230}, fill opacity=1]
(173.7,166.79) .. controls (173.7,168.66) and (174.99,170.18) .. (176.59,170.18) .. controls (178.19,170.18) and (179.48,168.66) .. (179.48,166.79) .. controls (179.48,164.92) and (178.19,163.41) .. (176.59,163.41) .. controls (174.99,163.41) and (173.7,164.92) .. (173.7,166.79) -- cycle ;
\draw [fill={rgb, 255:red, 173; green, 216; blue, 230}, fill opacity=1] (201.42,131.2) .. controls (201.42,133.07) and (202.71,134.58) .. (204.31,134.58) .. controls (205.91,134.58) and (207.2,133.07) .. (207.2,131.2) .. controls (207.2,129.33) and (205.91,127.81) .. (204.31,127.81) .. controls (202.71,127.81) and (201.42,129.33) .. (201.42,131.2) -- cycle ;
\draw [fill={rgb, 255:red, 173; green, 216; blue, 230}, fill opacity=1] (167.82,131.2) .. controls (167.82,133.07) and (169.11,134.58) .. (170.71,134.58) .. controls (172.31,134.58) and (173.6,133.07) .. (173.6,131.2) .. controls (173.6,129.33) and (172.31,127.81) .. (170.71,127.81) .. controls (169.11,127.81) and (167.82,129.33) .. (167.82,131.2) -- cycle ;
\draw [fill={rgb, 255:red, 173; green, 216; blue, 230}, fill opacity=1] (177.06,87.18) .. controls (177.06,89.05) and (178.35,90.56) .. (179.95,90.56) .. controls (181.55,90.56) and (182.84,89.05) .. (182.84,87.18) .. controls (182.84,85.31) and (181.55,83.79) .. (179.95,83.79) .. controls (178.35,83.79) and (177.06,85.31) .. (177.06,87.18) -- cycle ;
\draw [fill={rgb, 255:red, 173; green, 216; blue, 230}, fill opacity=1] (145.14,105.91) .. controls (145.14,107.78) and (146.43,109.3) .. (148.03,109.3) .. controls (149.63,109.3) and (150.92,107.78) .. (150.92,105.91) .. controls (150.92,104.04) and (149.63,102.52) .. (148.03,102.52) .. controls (146.43,102.52) and (145.14,104.04) .. (145.14,105.91) -- cycle ;

\draw (166.25,71.57) node [anchor=north west][inner sep=0.75pt]   [align=left] {{\scriptsize $1$}};
\draw (206.28,115.51) node [anchor=north west][inner sep=0.75pt]   [align=left] {{\scriptsize $3$}};
\draw (138.91,153.35) node [anchor=north west][inner sep=0.75pt]   [align=left] {{\scriptsize $6$}};
\draw (170.08,173.62) node [anchor=north west][inner sep=0.75pt]   [align=left] {{\scriptsize $5$}};
\draw (173.44,116.45) node [anchor=north west][inner sep=0.75pt]   [align=left] {{\scriptsize $4$}};
\draw (106.8,111.36) node [anchor=north west][inner sep=0.75pt]   [align=left] {{\scriptsize $7$}};
\draw (136.66,90.59) node [anchor=north west][inner sep=0.75pt]   [align=left] {{\scriptsize $2$}};
\end{tikzpicture}
       \caption{Three concurrent lines}
       \label{three concurrent lines}
    \end{subfigure}
\caption{Examples of rank 3 matroids.}
    \label{fig:combined}
\end{figure}

\subsection{Paving, nilpotent and inductively connected matroids}

In this section, we present families of matroids, from \cite{Fatemeh5,Fatemeh6}, that will play an important role in this work. As we will see in the next section, these will form base cases for our strategy. In the following, consider $M$, a matroid of rank $n$ on $[d]$.

\begin{definition}\normalfont \label{pav}
A matroid $M$ of rank $n$ is called a {\em paving matroid} if every circuit of $M$ has a size either $n$ or $n+1$. In this case, we refer to $M$ as an $n$-paving matroid. We also introduce the following terminology.
\begin{itemize}
\item The set of subspaces $\mathcal{L}_{M}$, as defined in Definition~\ref{general}, corresponds to the collection of {\em dependent hyperplanes} of $M$. These are maximal subsets of points, of size at least $n$, in which every subset of $n$ points forms a circuit.
\item When $n=3$, these dependent hyperplanes are simply called {\em lines}, and $M$ is referred to as a {\em point-line configuration}. 
\end{itemize}
\end{definition}

\begin{example}
The matroid of rank $3$ depicted in Figure~\ref{three concurrent lines} is a point-line configuration with points $[7]$ and lines given by $\mathcal{L}=\{\{1,2,7\}, \{3,4,7\}, \{5,6,7\}\}$. 
\end{example}

\begin{example}
The following collection of subsets of $[7]$:
\[\{1,2,4\},\{1,3,7\},\{1,5,6\},\{2,3,5\},\{4,5,7\},\{2,6,7\},\{3,4,6\},\] 
defines a point-line configuration, where each subset corresponds to a line. The associated matroid is known as the \emph{Fano plane}, see Figure~\ref{fano}.
\end{example}

For the following definition, recall Definition~\ref{general}.

\begin{definition}
\label{nilpotent}
Let $S_M=\{p \in [d] \mid \text{deg}(p) > 1\}$.
The \textit{nilpotent chain} of $M$ is defined as the following sequence of submatroids of $M$: 
$$M_{0}=M,\quad M_{1}=M|S_{M},\quad \text{and }\ M_{j+1}=M|S_{M_{j}} \ \text{ for every $j\geq 1$.}$$
We say that $M$ is \textit{nilpotent} if $M_{j}=\emptyset$ for some $j$.
\end{definition}

A \emph{basis} of $M$ is a maximal independent subset of $[d]$ with respect to inclusion. All bases of $M$ have the same size, and the collection of all bases is denoted by $\mathcal{B}(M)$.\vs

\begin{definition}
\label{induc}
We say that $M$ is \emph{inductively connected} if there exists a permutation $w=(j_{1},\ldots,j_{d})$ of $[d]$ such that:
\begin{enumerate}[label=(\roman*)]
\item the first $n$ elements $j_{1},\ldots,j_{n}$ form a basis of $M$;\vs
\item for each $i \in \{n+1,\dotsc,d\}$, we have $\deg(j_{i}) \leq 2$ within 
$M|\{j_{1},\ldots,j_{i}\}$. 
\end{enumerate}
\end{definition}

\begin{example}
The configuration in Figure~\ref{Quadrilateral set} is not nilpotent, while the one in Figure~\ref{three concurrent lines} is. Both matroids are inductively connected.
\end{example}

\noindent Note that, from the definition, it follows that nilpotent matroids are inductively connected. 

\section{Decomposing using maximal matroid degenerations}\label{sec 3}

As outlined earlier, our main approach is to first decompose a circuit variety into smaller ones, then either apply known results or recursively repeat the process. In this section, we introduce an efficient method for decomposing the circuit variety of a given matroid $M$. This method is based on an algorithm for identifying the maximal matroid degenerations of $M$, which we will detail in the following subsections.

\subsection{Reduction to the maximal matroid degenerations}\label{sec deco}

We first introduce an order relation on matroids. 
This 
is the {\em weak order} commonly studied in the literature~\cite{Oxley}.
\begin{definition}\label{dependency}
Let $N_{1}$ and $N_{2}$ be matroids on $[d]$. We say that $N_{2}\leq N_{1}$ if $\mathcal{D}(N_{2}) \supseteq \mathcal{D}(N_{1})$. This partial order is referred to as the {\em 
weak order}
on matroids. 
\end{definition}

\noindent We can now present the main object of interest in this work.

\begin{definition}\label{def min}
The set of all maximal matroid degenerations of a matroid $M$ is:
\[
    \minabove{M}=\minposet{N:N<M}.
\] 
\end{definition}
We recall the following result from \textup{\cite[Proposition~4.1]{liwski2025minimal}}, which establishes a relationship between the circuit variety of $M$ and those of its maximal matroid degenerations.

\begin{proposition}
\label{deco circ}
Let $M$ be a matroid. Then $V_{\Ccal(M)}=\cup_{N\in \minabove{M}}V_{\Ccal(N)}\ \cup \ V_{M}$.
\end{proposition}

\noindent Building on the result above, the first step toward Question~\ref{mainquest} is to address the following.

\begin{question}\label{quest} 
How can we compute $\minabove{M}$ given a matroid $M$?
\end{question}

Sections~\ref{general rank} and \ref{algorithm} will be devoted to fully answering Question~\ref{quest}. Assuming we know how to perform such a task, the following key results from \cite{Fatemeh6,Fatemeh5,liwski2025solvable} give tools to conclude on some special cases that will constitute our base cases.

\begin{theorem}\label{teo ir}
If $M$ is a paving matroid without points of degree greater than two, then: 
\begin{enumerate}[label=\rm(\roman*)]
\item if $M$ is nilpotent, then $V_{\Ccal(M)}=V_{M}$;
\item if every proper submatroid of $M$ is nilpotent, then $V_{\Ccal(M)}=V_{M}\cup V_{U_{n-1,d}}$;
\item if $M$ is inductively connected and realizable, then $V_{M}$ is an irreducible variety.
\end{enumerate}
\end{theorem}

\begin{remark}
Theorem~\ref{teo ir} identifies two broad families of matroids for which our strategy is
guaranteed to terminate and to recover the full irreducible decomposition:
nilpotent matroids and inductively connected realizable matroids.  These serve
as the base cases of our recursive approach, and in all examples of Section~\ref{examples}
the degenerations produced by Algorithms~\ref{211} and \ref{alg:computeleaf} fall into these families after
one or two iterations.
\end{remark}

\begin{example}\label{ej quad}
Consider the point-line configuration $\text{QS}$ illustrated in Figure~\ref{Quadrilateral set}. Since every proper submatroid of $\text{QS}$ is nilpotent, it follows that $V_{\Ccal(\text{QS})}=V_{\text{QS}}\cup V_{U_{2,6}}$. Noting that both $\text{QS}$ and the uniform matroid $U_{2,6}$ are inductively connected, so that their matroid varieties are irreducible. It follows that this decomposition gives the irreducible components of the circuit variety of $\text{QS}$.
\end{example}

Finally, the last ingredient in our strategy involves identifying redundant irreducible components, which boils down to answering the following question.

\begin{question}\label{question}
Given two realizable matroids $M$ and $N$ on the same ground set, under what conditions does the inclusion $V_N \subseteq V_M$ hold between their associated varieties?
\end{question}

We will partially answer this question in Section~\ref{appen} but, for the time being, assume that we know how to answer it. We outline our strategy for determining the irreducible decomposition of circuit varieties in the following strategy. We emphasize that this strategy is not guaranteed to terminate nor to provide the full irreducible decomposition of $V_{\Ccal(M)}$ in all cases. In particular, we do not always reach one of the termination cases (1 or 2). Moreover, even when these cases are satisfied, the final two steps do not offer a complete strategy for determining the irreducibility and redundancy of all the obtained components. Therefore, alternative and adapted methods may be required for such matroid varieties; see Section~\ref{examples} for several examples.

\begin{algorithm}[H]
\caption*{Decomposition strategy}\label{alg:decompstrategy}
\begin{algorithmic}[1]
 \Require A matroid $M$\vs
 \Ensure A list of matroids $L_M=(M_1,\dotsc,M_k)$ such that\vspace*{-.3em}
 \[
    V_{\Ccal(M)} = V_{M_1} \cup \cdots \cup V_{M_k},\vspace*{-.3em}
 \]
 and this is a potential irreducible decomposition.\vs
\State {\bf Case~1:} if $M$ is a nilpotent paving matroid with no points of degree greater 
than two then \Return the list $(M)$. \hfill \scalebox{.9}{(Thm\,\ref{teo ir})}\vs
\State {\bf Case~2:} if $M$ is a paving matroid where all points have degree at most two, 
and all proper submatroids are nilpotent, then \Return the list $(M,U_{2,d})$. \hfill \scalebox{.9}{(Thm\,\ref{teo ir})}\vs 
\State Else, compute the set $\minabove{M}$ of maximal matroid degenerations of $M$. \hfill \scalebox{.9}{(Question\,\ref{quest})}\vs 
\State For each $N \in \minabove{M}$ compute a list $L_N$, such that $V_{\Ccal(N)} = \cup_{N' \in L_N} V_{N'}$

\noindent by applying recursively this strategy to $N$. 

\noindent Note $L_M$ the list containing $M$ and the concatenation of all $(L_N)_{N \in \minabove{M}}$.\vs

\State\label{step:redun} For each $N \in L_M$, attempt to determine the irreducibility of $V_N$ e.g. by 
identifying the associated inductively connected matroids. \hfill \scalebox{.9}{(Thm\,\ref{teo ir})}\vs
\State Remove the $N \in L_M$ corresponding to redundant irreducible components. \hfill
\scalebox{.9}{(Question\,\ref{question})}\vs
\State \textbf{Return} the list $L_M$.
\end{algorithmic}
\end{algorithm}

\begin{example}\label{example Fano}
Consider the Fano plane, which we denote $M_{\textup{Fano}}$ and is depicted in Figure~\ref{fano}. This is the point-line configuration on $[7]$ with the following set of lines:
\[\{1,2,4\},\{1,3,7\},\{1,5,6\},\{2,3,5\},\{4,5,7\},\{2,6,7\},\{3,4,6\}.\]
Using the decomposition strategy (see also~\cite{clarke2021matroid}, where this decomposition is provided), we obtain the irreducible decomposition of $V_{\Ccal(M_{\textup{Fano}})}$: 
\begin{equation}\label{dec fano 2}
V_{\Ccal(M_{\text{Fano}})}=V_{U_{2,7}}\bigcup_{i=1}^{7}V_{M_{\text{Fano}}(i)}\bigcup_{j=1}^{7}V_{A_{j}\rq}\bigcup_{k=1}^{7}V_{B_{k}\rq},\end{equation}
where the matroids $M_{\textup{Fano}}(i),A_{j}\rq$ and $B_{k}\rq$ consist of the following matroids (see 
Figure~\ref{figure 4}):
\begin{itemize}
\item The matroids $M_{\textup{Fano}}(i)$, for $i \in [7]$, obtained from $M_{\textup{Fano}}$ by declaring the point $i$ to be a loop.
\item The matroids $A_{j}'$, for $j\in [7]$, consisting of a line of $M_{\text{Fano}}$, with the remaining four points coinciding outside this line;
\item The matroids $B_{k}'$, for $k \in [7]$, consist of a line containing three double points together with a free point $k$ outside this line, where points lying on a common line of $M_{\textup{Fano}}$ with $k$ form a double point.
\end{itemize}

\begin{figure}[H]
    \centering

\begin{subfigure}[b]{0pt}
  \phantomsubcaption
  \label{fig:3a}
\end{subfigure}%
\begin{subfigure}[b]{0pt}
  \phantomsubcaption
  \label{fig:3b}
\end{subfigure}%
\begin{subfigure}[b]{0pt}
  \phantomsubcaption
  \label{fig:3c}
\end{subfigure}
\begin{subfigure}[b]{0pt}
  \phantomsubcaption
  \label{fig:3d}
\end{subfigure}%
    \includegraphics[width=0.8\textwidth, trim=0 0 0 0, clip]{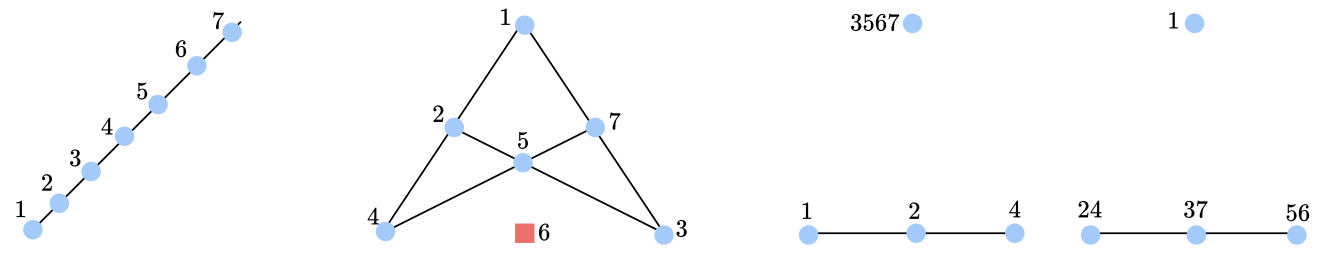}

\caption{(a) Uniform matroid $U_{2,7}$; (b) Matroid $M_{\textup{Fano}}(6)$; (c) $A_{1}'$; (d) $B_{1}'$.}

    \label{figure 4}
\end{figure}
\end{example}

\begin{remark}
This strategy can be significantly improved as follows. On each call, we first reduce $M$ to a simple matroid $M_{\textup{red}}$ by removing loops and identifying double points (see Subsection~\ref{ssec:remove2loops}). The advantage of this is that while the Theorem~\ref{teo ir} cannot be applied to the non-simple matroid $M$ in cases~1 and~2, it can be applied to its simple reduction $M_{\textup{red}}$. Thus, we speed up the termination of the strategy and significantly extend the range of problems that can be addressed. We use this approach in Section~\ref{examples}.

The drawback of this optimization is that the list of matroids given as output does not directly provide the components of the decomposition. However, these can be determined by converting certain simple points back to double points, which is done by adding identical copies of them and also adding back the loops. At the level of irreducible components of circuit varieties, this corresponds to introducing identical copies, up to non-zero scalars, of the variables associated with double points and adding zero vectors for the loops. This operation preserves the irreducibility.
\end{remark}

\subsection{Reduction to labeled hypergraphs}\label{ssec:reduchyper}

Our approach to decomposing the circuit variety of $M$ begins by reducing the problem 
to that of designing an algorithm for computing the maximal matroid degenerations of $M$ 
(see Question~\ref{quest}).
The first step of this algorithm is to degenerate $M$ by declaring a new subset to be dependent. 
Such a declaration, however, typically forces additional unintended dependencies, 
so that the resulting structure may fail to satisfy the matroid axioms.  
The task is therefore to identify the maximal matroids that include all these
induced dependencies. These dependencies admit a natural encoding in terms of a hypergraph, 
i.e. a collection of subsets of the ground set.  
This leads to the problem of determining the maximal matroids whose dependencies contain 
those prescribed by a fixed hypergraph $\Delta$. 

To address this, we refine the problem using labeled hypergraphs, where each subset is assigned a number indicating a bound on its rank. This definition depends on a fixed integer $n$, which we assume to be constant throughout this section.
Using this notion, we reduce Question~\ref{quest} to the  more elementary Questions~\ref{quest:compare} and~\ref{quest 21}.


\begin{definition}\label{def:hyper}
A {\em labeled hypergraph} $\Delta$ on the vertex set $[d]$ is a collection of subsets of $[d]$, called edges, satisfying the following properties:
\begin{enumerate}[label=(\roman*)]
\item\label{def:hyper1} each edge is assigned a label: Type~$i$, for some $0\leq i\leq n-1$;
\item\label{def:hyper2} no pair of edges, $e_{1}, e_{2} \in \Delta$, of the same type satisfy $e_{1} \subsetneq e_{2}$.
\item\label{def:hyper3} if an edge $e$ is of Type~$i$, then $\size{e}\geq i+1$.
\end{enumerate}
The elements of $\Delta$ are called edges,
and $\Delta_{i}$ denotes the collection of edges of Type~$i$. For simplicity, we will typically refer to a labeled hypergraph as a hypergraph.
\end{definition}

This provides a more compact representation of dependencies compared to standard hypergraphs, as illustrated in the following examples.

\begin{example}
To encode that every 3-element subset of $\{1,\ldots,7\}$ is dependent, a hypergraph would require an explicit listing of these $35$ subsets. In contrast, this can be encoded by the labeled hypergraph consisting of the single edge $\{1,\ldots,7\}$ labeled with Type $2$. This simply records that the entire set has rank at most two. 
\end{example}

This refinement significantly reduces the number of candidate matroids to consider and makes the decomposition process more efficient.

\begin{definition}\label{induced hypergraph}
Let $\Delta$ be a collection of subsets of $[d]$ satisfying property~\ref{def:hyper1} of Definition~\ref{def:hyper}. The {\em hypergraph induced} by $\Delta$, denoted $\Delta_{\text{ind}}$, is the labeled hypergraph obtained by removing, for each $0\leq i\leq n-1$, the following sets of edges: 
\[\{e\in \Delta_{i}: \ \text{there exists $e'\in \Delta_{i}$ with $e\subsetneq e'$}\}, \quad \text{and} \quad \{e\in \Delta_{i}:\size{e}\leq i\}.\]
\end{definition}

\begin{definition}\label{def:orderhyper}
Let $\Delta$ be a labeled hypergraph and $N$ a matroid on the same ground set. We write $N\geqhyp \Delta$ if, for each $0\leq i\leq n-1$, 
the following equivalent conditions hold:
\begin{enumerate}[label=(\roman*)]
\item every $(i+1)$-subset of an edge in $\Delta_{i}$ is dependent in $N$;
\item for all $e\in \Delta_{i}$, we have $\rank_{N}(e)\leq i$.
\end{enumerate}
\end{definition}

\begin{example}\label{example: hypergraph delta}
Consider the labeled hypergraph $\Delta = \Delta_{1} \cup \Delta_{2}$ on $[7]$, where
\[
\Delta_{1} = \{\{3,7\}\}, \quad 
\Delta_{2} = \{\{1,2,4\}, \{1,3,7\}, \{1,5,6\}, \{2,3,5\}, \{4,5,7\}, \{3,4,6\}, \{2,6,7\}\}.
\]
The matroids $A_{1}'$ and $B_{1}'$ from Example~\ref{example Fano} (see Figures~\ref{fig:3c} and~\ref{fig:3d}) both satisfy
\[
A_{1}' \geqhyp \Delta \quad \text{and} \quad B_{1}' \geqhyp \Delta.
\]
\end{example}

The following definition shows how to encode a matroid as a labeled hypergraph.

\begin{definition}\label{deltam}
Let $M$ be a matroid of rank $n$ on $[d]$. The labeled hypergraph $\Delta_{M}$ on $[d]$ associated to $M$ is defined such that:
\begin{itemize}
\item for each $0\leq i\leq n-1$, the edges of Type~$i$ of $\Delta_{M}$ are precisely the cyclic flats of $M$ of rank $i$.
\end{itemize}
\end{definition}

\begin{example}
Consider the Fano plane $M_{\textup{Fano}}$ from Example~\ref{example Fano}. The corresponding labeled hypergraph $\Delta_{M_{\textup{Fano}}}$ is
\[
\Delta_{M_{\textup{Fano}}} = \{\{1,2,4\}, \{1,3,7\}, \{1,5,6\}, \{2,3,5\}, \{4,5,7\}, \{2,6,7\}, \{3,4,6\}\},
\]
where each edge of the hypergraph is of Type~2.
\end{example}

The following lemma gives an equivalent characterization for the matroid degenerations of a given matroid $M$, in terms of its labeled hypergraph $\Delta_{M}$.

\begin{lemma}\label{delta}
A matroid $N$ satisfies $N\leq M$ if and only if $N\geqhyp \Delta_{M}$.
\end{lemma}

\begin{proof}
If $N\leq M$, then for any edge $e\in (\Delta_{M})_{i}$ and $0\leq i\leq n-1$, it holds that $\rank_{N}(e)\leq \rank_{M}(e)$. Since $\rank_{M}(e)=i$, it follows that $\rank_{N}(e)\leq i$, which implies $N\geqhyp \Delta_{M}$. 

Conversely, suppose $N\geqhyp \Delta_{M}$, and let $c$ be an arbitrary circuit of $M$. Let $i=\size{c}$. Then, $c$ is contained within a cyclic flat of $M$ of rank $i-1$, so there exists $e\in (\Delta_{M})_{i-1}$ with $c\subset e$. Since $N\geqhyp \Delta_{M}$, we have $\rank_{N}(c)\leq \rank_{N}(e)\leq i-1$, implying that $c$ is dependent in $N$. Since $c$ is an arbitrary circuit, this demonstrates that $N\leq M$.
\end{proof}

\noindent We now exploit this 
characterization to formulate Question~\ref{quest} in term of 
hypergraphs.
\begin{definition}
\noindent The set of all maximal matroid degenerations of a 
hypergraph $\Delta$ is:
\[
    \minabovehyp{\Delta}=\minposet{N:N\geqhyp \Delta \ \text{and $\rank(N)\leq n$}}.
\]
\end{definition}

\begin{example}
Consider the labeled hypergraph $\Delta$ from Example~\ref{example: hypergraph delta}. The matroids $A_{1}'$ and $B_{1}'$ from Example~\ref{example Fano} (see Figures~\ref{fig:3c} and~\ref{fig:3d}) both belong to $\minabovehyp{\Delta}$.
\end{example}

\begin{lemma}\label{lema}
Let $M$ be a matroid on $[d]$. For $e\subset [d]$
denote by $\Delta_{e}$ the hypergraph $\Delta_{M}\cup \{e\}$, where $e$ is assigned Type~$(\size{e}-1)$.
Then, the following holds:
\begin{equation}\label{equality}
\minabove{M}=\minposet{
\bigcup_{e\not \in \mathcal{D}(M)}
\minabovehyp{\Delta_{e}} 
}.\end{equation}
\end{lemma}

\begin{proof}

\smallskip
To prove the inclusion $\subseteq$ in~\eqref{equality}, let $N\in \minabove{M}$. Since $N<M$, there exists a circuit $C$ of $N$ that is independent in $M$. By Lemma~\ref{delta}, we know $N\geqhyp \Delta_{M}$, which implies $N\geqhyp \Delta_{C}$. Furthermore, since $N\in \minabove{M}$, it follows that $N\in \minabovehyp{\Delta_{C}}$. This establishes the inclusion $\subseteq$.

To prove the other inclusion, let $N$ be a matroid belonging to the right-hand side of~\eqref{equality}, say $N\in \minabovehyp{\Delta_{e}}$. 
Since $N\geqhyp \Delta_{M}$, it follows that $N\leq M$. Additionally, since $e\in \mathcal{D}(N)$ and $e\not \in \mathcal{D}(M)$, we have $N<M$. To conclude, we must show that $N\in \minabove{M}$. Assume for the sake of contradiction, that there exists $N\rq\in \minabove{M}$ such that $N<N\rq<M$. Using the previous argument, $N\rq$ belongs to the set on the right-hand side of Equation~\eqref{equality}. However, since both $N$ and $N\rq$ are maximal matroids in this set, this leads to a contradiction. Consequently, $N\in \minabove{M}$, as required. 
\end{proof}
In conclusion, according to Lemma~\ref{lema}, an answer to Question~\ref{quest} can be derived from one of the following two ones. Indeed, combining these allows one to directly compute the right-hand side of  \eqref{equality}.
\begin{subquestion}\label{quest:compare}
Given two matroids $M$ and $M'$, design an algorithm to decide if $M' \leq M$?
\end{subquestion}


\begin{subquestion}\label{quest 21}
Given a labeled hypergraph $\Delta$, design an algorithm to compute $\minabovehyp{\Delta}$?
\end{subquestion}

Solving these two problems will be the focus of Sections~\ref{general rank} and \ref{algorithm}. In the remainder of this section, we present useful reductions that will be extensively used in our algorithms.

\vspace{-2mm}
\subsection{Reduction to simple matroids by removing loops and double points}\label{ssec:remove2loops}

In this section, we describe a reduction to Question~\ref{quest 21} that allows us to assume 
the labeled hypergraph $\Delta$ satisfies $\Delta_{0}=\Delta_{1}=\emptyset$.  
The motivation for this reduction is explained in the following remark.

\begin{remark}
The hypergraph $\Delta$ can be reduced to a simpler labeled hypergraph $\Delta_{\mathrm{red}}$ 
by removing edges of size one and identifying elements of the ground set that lie in an edge of size two.  
This yields a smaller hypergraph, thereby accelerating the termination of the strategy.
\end{remark}

\begin{definition}\label{def:designloop}
For each $k\in [d]$, let $M(k)$ denote the matroid obtained by designating $k$ as a loop. The circuits of this matroid are given by
$\Ccal(M(k))=\Ccal(M\backslash k)\cup \{\{k\}\}$.

Similarly, for a labeled hypergraph $\Delta$ on $[d]$ and $k\in [d]$, we denote by $\Delta \setminus \{k\}$ the labeled hypergraph on $[d]\setminus \{k\}$ obtained by removing $k$, that is whose edges are given by 
\[
    \big\{ e - \{k\} : \; e \in \Delta_i \text{ \; and \;} 
    \left|e - \{k\}\right| \geq i+1 \big\}
\]
that are assigned Type~$i$, for all $0\leq i \leq n-1$.
\end{definition}

\begin{lemma}\label{lem:removeloops}
Let $\Delta$ be a labeled hypergraph and $k\in \Delta_{0}$. There is a bijection, preserving both order and rank, between the sets
\[
 \{N:N\geqhyp \Delta\}, \quad \text{and}
\quad \{N':N'\geqhyp \Delta \setminus \{k\}\}.
\]
\end{lemma}

\begin{proof}
First note that, according to Definition~\ref{def:orderhyper}, the matroids in the second set have ground set $[d]\setminus \{k\}$.
Moreover, any matroid $N$ satisfying $N\geqhyp \Delta$ must have $\{k\}$ as a loop. Then the result follows from the fact that $N\geqhyp \Delta$ if and only if $N(k)\geqhyp \Delta\setminus \{k\}$.
\end{proof}

By applying the previous lemma and after removing all vertices of $\Delta_0$ from the ground set, one can address Question~\ref{quest 21} assuming that $\Delta_{0}=\emptyset$. We now proceed similarly for $\Delta_1$ by introducing the concept of the {\em reduction} of a labeled hypergraph. 

\begin{definition}\label{redu}
Let $\Delta$ be a labeled hypergraph on $[d]$ with $\Delta_{0}=\emptyset$. The reduction $\Delta_{\text{red}}$ of $\Delta$ is defined as follows.\vspace*{-.5em}
\begin{enumerate}[label=\bf\footnotesize\arabic*:]
\item Define an equivalence relation $\sim_{\Delta}$ on $[d]$ as: $i\sim_{\Delta} j$ if $\{i,j\}\subset x$ for some $x\in \Delta_{1}$. Let $\mathcal{Q}$ be the set of minimal representatives in each equivalence class, denoted by $\overline{i}$.
\vs 
\item The reduced hypergraph $\Delta_{\text{red}}$ is constructed on the vertex set $\mathcal{Q}$ by modifying the edges of $\Delta$ as follows. 
For each $2\leq i\leq n$,
\begin{itemize}
\item for\,each $e=\{j_{1},\ldots,j_{k}\}\in \Delta_{i}$, compute its representative $\overline{e}=\{\overline{j_{1}},\ldots,\overline{j_{k}}\}$ in $\mathcal{Q}$;
\item \vspace*{.25em} if $\size{\overline{e}}\geq i+1$, include $\overline{e}$ in $\Delta_{\text{red}}$ and assign to it Type~$i$.
\end{itemize}
\end{enumerate}
Observe that $\Delta_{\text{red}}$ contains no edges of Type~$0$ or Type~$1$.
\end{definition}

\begin{example}
Consider the labeled hypergraph $\Delta$ from Example~\ref{example: hypergraph delta}. The reduced hypergraph $\Delta_{\text{red}}$ on $[6]$, obtained by identifying $3$ with $7$, is
\[
\Delta_{\text{red}} = \{\{1,2,4\}, \{1,5,6\}, \{2,3,5\}, \{3,4,5\}, \{3,4,6\}, \{2,3,6\}\},
\]
where each edge of the hypergraph is of Type~2.
\end{example}

We have the following lemma for the reduction of a labeled hypergraph.

\begin{lemma}\label{lem:remove2points}
Let $\Delta$ be a labeled hypergraph with $\Delta_{0}=\emptyset$. There is a bijection, preserving both order and rank, between the sets  
\begin{equation}\label{sets}
\{N: \Ccal_{1}(N)=\emptyset,\, N\geqhyp \Delta\},\quad  
\text{and} 
\quad \{N\rq: \Ccal_{1}(N\rq)=\emptyset,\, 
N\rq \geqhyp \Delta_{\textup{red}}\}.
\end{equation}
Following Definitions~\ref{def:orderhyper} and \ref{redu}, the matroids in the latter set have ground set $\mathcal{Q}$.
\end{lemma}

\begin{proof}
Let $N$ be a matroid belonging to the set on the left-hand side of~\eqref{sets}. Since $N\geqhyp \Delta$ and $N$ has no loops, it follows that every pair of elements within an edge of $\Delta_{1}$ forms a circuit by  Definition~\ref{def:orderhyper}.(i).
Consequently, $N$ uniquely determines a matroid $N\rq$ with $N\rq \geqhyp \Delta_{\text{red}}$ by identifying double points in $N$. It is straightforward to verify that this assignment is bijective and preserves both order and rank.
\end{proof}

In conclusion to the above two lemmas, after removing all vertices of $\Delta_0$ from the ground set, and then reduction, one can address Question~\ref{quest 21} assuming that $\Delta_{0}=\Delta_1=\emptyset$.


\section{Algorithm for identifying maximal matroid degenerations}\label{general rank}

We introduce an algorithm to answer Question~\ref{quest} and determine the maximal matroid degenerations of a given matroid $M$, on the ground set $[d]$ and of rank $n$.

\vspace{-3mm}\subsection{Comparing matroids}

As outlined in Subsection~\ref{ssec:reduchyper}, the first step in addressing Question~\ref{quest} is to answer Question~\ref{quest:compare}. We therefore present Algorithm~\ref{alg:algocompare5} for comparing matroids.

\begin{algorithm}[h]
\caption{Comparison of matroids}\label{alg:algocompare5}
\begin{algorithmic}[1]
\Require A pair of matroids $M$ and $M'$ on the same ground set, both of rank at most $n$.
\Ensure \texttt{True} if $M' \leq M$, else \texttt{False}.\vs\vs
\State If any of the following test fails, immediately return \texttt{False}, else continue.\vs
\State \emph{Compute} $\Delta_M$ and $\Delta_{M'}$ as defined in Definition~\ref{deltam}, and denote the resulting labeled hypergraphs by $\Delta$ and $\Delta'$, respectively.\vs
\State\label{step:compare:3}\emph{Check} that 
        $\Delta_0 \subset\Delta_0'.$\vs
\State \emph{Reduce} both $\Delta$ and $\Delta'$ with respect to the loops in $\Delta_0'$ (see Definition~\ref{def:designloop}), and denote the resulting labeled hypergraphs on $[d] \setminus \Delta_0'$ by the same symbols.\vs
\State\label{step:compare:5}\emph{Check} that for every $x \in \Delta_1$, there exists $y \in \Delta'_1$ such that $x \subset y$.\vs
\State\label{step:compare:6}\emph{Reduce} both $\Delta$ and $\Delta'$ with respect to the double points in $\Delta_1'$ (see Definition~\ref{redu}), and denote the resulting labeled hypergraphs by $\mathcal{Q}$ and $\mathcal{Q}'$, respectively.\vs

\State\label{step:compare:7}\emph{Check} that for all $2\leq i\leq n-1$ and all
$x \in \Delta_i$,
    $$ 
    \exists A \subset x \ \text{with $\size{A}\leq i$}\ \text{and} \  \exists y \in \Delta_{i-\size{A}}' \ \text{such that} \ (x\setminus A) \subset y.$$
\State If all tests have been passed, return \texttt{True}.
\end{algorithmic}
\end{algorithm}

\begin{example}
In this example, we illustrate Algorithm~\ref{alg:algocompare5} for comparing matroids. 
Consider the matroids $A_{1}'$ and $B_{1}'$ from Example~\ref{example Fano} (see Figures~\ref{fig:3c} and~\ref{fig:3d}). Recall that their corresponding matroid varieties occur as components of $V_{\Ccal(M_{\mathrm{Fano}})}$, as established in Example~\ref{example Fano}. We can easily see that these matroids are not comparable in the weak order. 
Indeed, they agree through tests~1–4, but fail at test~5: the set of double 
points of neither matroid is contained in that of the other. Hence, they 
are not comparable.
\end{example}

The remainder of this subsection is devoted to proving the correctness of Algorithm~\ref{alg:algocompare5}. 
The crucial step of Algorithm~\ref{alg:algocompare5} is step~\ref{step:compare:7}, which determines whether the rank of a certain subset $X$ of the ground set of a matroid $N$ is at most $i$, using only the information encoded in the collection of cyclic flats $\Delta_{N}$. 
As the following lemma shows, this can be decided by checking whether $X$ is contained in a cyclic flat of rank $i-j$ after removing a subset of size $j$, which corresponds precisely to the verification performed in step~\ref{step:compare:7}. 
We make this formal in Lemma~\ref{lemma cyclic flats 2}.


\begin{lemma}\label{lemma cyclic flats 2}
Let $N$ be a matroid of rank at most $n$ on $[d]$, and let $\Delta$ denote the hypergraph $\Delta_{N}$. Then, for any $2 \leq i \leq n-1$ and any subset $X \subseteq [d]$ with $\lvert X \rvert \geq i+1$, the following conditions are equivalent:
\begin{itemize}
\item[{\rm (i)}] $\rank(X)\leq i$.
\item[{\rm (ii)}] 
There exists a subset $A\subset X$ with $\size{A}\leq i$ and an element $y\in \Delta_{i-\size{A}}$ such that ($X\setminus A)\subset y$.
\end{itemize}
\end{lemma}

\begin{proof}
It is clear that (ii) implies (i). For the converse, consider the submatroid $N|X$, and let $A \subset X$ denote the set of coloops in $N|X$. Since $\rank(X) \leq i$, we have $\size{A} \leq i$. Furthermore, removing the coloops yields $\rank(X \setminus A) \leq i - \size{A}$.
Note that $N|(X \setminus A)$ has no coloops, so every element lies in a circuit. Thus, $X \setminus A$ is a union of circuits, and therefore contained in a cyclic flat of rank at most $i - \size{A}$, completing the proof.
\end{proof}

\noindent{\bf Correctness of Algorithm~\ref{alg:algocompare5}.} First, observe that if the tests in Steps~\ref{step:compare:3} and \ref{step:compare:5} fail, then $M$ contains 
a loop and a double point that are not present in $M'$, hence 
$M' \not\leq M$.

Now, suppose that both tests pass, so we are after step~\ref{step:compare:6}. According to Subsection~\ref{ssec:remove2loops}, the resulting hypergraphs correspond to matroids $N$ and $N'$, which are free of loops and double points. By Lemmas~\ref{lem:removeloops} and~\ref{lem:remove2points}, we know that $M' \leq M$ holds if and only if $N' \leq N$.
The latter condition is equivalent to the following:
\begin{itemize}
\item For every $2 \leq i \leq n-1$ and each $x \in \Delta_i$, we have $\rank_{N'}(x) \leq i$.
\end{itemize}

\noindent By applying Lemma~\ref{lemma cyclic flats 2}, the above condition is equivalent to the following:
\begin{itemize}
\item For every $2 \leq i \leq n-1$ and each $x \in \Delta_i$, there exists a subset $A \subset x$ with $\size{A} \leq i$ and an element $y \in \Delta_{i - \size{A}}$ such that $(x \setminus A) \subset y$.
\end{itemize}
This condition corresponds exactly to the check at step~\ref{step:compare:7}. Thus, we conclude that $M' \leq M$ if and only if all tests in the algorithm pass, completing the proof of correctness.\qed

\subsection{Algorithm for identifying $\minabovehyp{\Delta}$}\label{min delta}
In this subsection, we present an algorithm to answer Question~\ref{quest 21}.
All labeled hypergraphs considered here are defined on $[d]$ and constructed with respect to a fixed integer $n$. 

\medskip
We now introduce the valuation $v_{\Delta}$, which provides an upper bound on the rank function of any matroid $N \geqhyp \Delta$.

\begin{definition}
Let $\Delta$ be a labeled hypergraph. We define the {\em valuation} $v_{\Delta}:2^{[d]}\rightarrow \mathbb{Z}$ as:
\[v_{\Delta}(A)=\min\{\size{A},\ n,\ \size{A\setminus e}+i:\ \text{$0\leq i\leq n-1$ and $e\in \Delta_{i}$}\}.\]
\end{definition}

\medskip 

The following lemma plays a key role in the development of 
Algorithm~\ref{algo 21}.

\begin{lemma}\label{def m}
Let $\Delta$ be a labeled hypergraph, and assume that for all $0\leq i,j\leq n-1$ and any $e_{1}\in \Delta_{i},e_{2}\in \Delta_{j}$, the following condition holds:
\begin{equation}\label{condition}i+j\geq v_{\Delta}(e_{1}\cap e_{2})+v_{\Delta}(e_{1}\cup e_{2}).\end{equation}
Then, the set 
\begin{equation}\label{min}\Ccal=\min \big(\bigcup_{0\leq i\leq n-1} \cup_{e\in \Delta_{i}}\textstyle \binom{e}{i+1}\cup \binom{[d]}{n+1}\big)\end{equation}
forms the circuits of a matroid $M_{\Delta}$, where 
$\min$ denotes the inclusion-minimal subsets.
\end{lemma}

\begin{proof}
Observe that, by Definition~\ref{def:hyper}.\ref{def:hyper2}, no element of $\Ccal$ is properly contained in another. Therefore, to verify that $\Ccal$ defines the set of circuits of a matroid, it suffices to check that $\Ccal$ satisfies the circuit elimination axiom:  
for any distinct $C_1, C_2 \in \Ccal$ and any element $x \in C_1 \cap C_2$, there exists a circuit $C_3 \in \Ccal$ such that  
$C_3 \subset (C_1 \cup C_2) \setminus \{x\}.$

Since $C_1, C_2 \in \Ccal$, there exist $e_1, e_2 \in \Delta$ such that $C_1 \subset e_1$ and $C_2 \subset e_2$. Let $0 \leq i, j < n$ such that $e_1 \in \Delta_i$ and $e_2 \in \Delta_j$. Furthermore, let $r = \size{C_1 \cap C_2}$.

\begin{claim}\label{cla:vlowbound}
    $v_{\Delta}(e_1 \cap e_2) \geq r$.
\end{claim}

\begin{proof}
Suppose, by contradiction, that $v_{\Delta}(e_1 \cap e_2) < r$. Since $\size{e_1 \cap e_2} \geq \size{C_1 \cap C_2} = r$, it follows that there exists $0 \leq k < n$ and some $e \in \Delta_k$ such that
\[
    \size{(e_1 \cap e_2) \setminus e} + k < r, \quad \text{so that} \quad \size{C_1 \cap C_2 \setminus e} < r - k.
\]
Consequently, $C_1 \cap C_2$ must contain a $(k+1)$-subset of $e$. Since $k + 1 \leq n$, this contradicts the minimality of $C_1$ and $C_2$ as sets in \eqref{min}.
\end{proof}

Using Claim~\ref{cla:vlowbound} and \eqref{condition}, we have that
    $v_{\Delta}(e_1 \cup e_2) \leq i + j - r$.
On the other hand, we have
    $\size{(C_1 \cup C_2) \setminus \{x\}} = i + j - r + 1$,
which implies
\[
    \size{e_1 \cup e_2} \geq \size{(C_1 \cup C_2) \setminus \{x\}} = i + j - r + 1.
\]
From this, we consider two cases:
\begin{itemize}
    \item \textbf{Case 1:} Assume $i + j - r \geq n$. In this case, $(C_1 \cup C_2) \setminus \{x\}$ contains an $(n+1)$-subset of $[d]$. Consequently, it must include an element of $\Ccal$.
    
    \item \textbf{Case 2:} Assume $i + j - r < n$. Here, there exists $0 \leq k < n$ and $e \in \Delta_k$ such that
        $\size{(e_1 \cup e_2) \setminus e} + k \leq i + j - r$.
    Consequently,
        $\size{((C_1 \cup C_2) \setminus \{x\}) \setminus e} \leq i + j - r - k$, 
    which implies that $(C_1 \cup C_2) \setminus \{x\}$ contains at least $k+1$ elements from $e$. Hence, it includes an element of $\Ccal$.
\end{itemize}
This completes the proof.
\end{proof}

By Lemma~\ref{unico}, Question~\ref{quest 21} becomes straightforward for labeled hypergraphs that satisfy the conditions of Lemma~\ref{def m}.

\begin{lemma}\label{unico}
Let $\Delta$ be a labeled hypergraph as in Lemma~\ref{def m}. Then, $M_{\Delta}$ is the unique maximal matroid degeneration in $\minabovehyp{\Delta}$.
\end{lemma}

\begin{proof}
Let $N \in \minabovehyp{\Delta}$. For each $0 \leq i \leq n-1$, we know that $\rank_{N}(e) \leq i$ for every $e \in \Delta_{i}$. From the description of the circuits of $M_{\Delta}$ in Lemma~\ref{def m}, we deduce that $\mathcal{D}(N) \supseteq \Ccal(M_{\Delta})$, which implies that $N \leq M_{\Delta}$.
Furthermore, since $M_{\Delta} \geqhyp \Delta$, we conclude that $M_{\Delta}$ is indeed the unique maximal matroid degeneration of $\Delta$ in $\minabovehyp{\Delta}$.
\end{proof}

Recall that a {\em stack} is an abstract data type that represents a collection of elements with two primary operations: {\em push}, which adds an element to the collection, and {\em pop}, which removes the most recently added element.
We now present Algorithm~\ref{algo 21}, which provides a solution to Question~\ref{quest 21}.

\begin{algorithm}[h]
\caption{Maximal matroid degenerations of a hypergraph}\label{algo 21}
\begin{algorithmic}[1]
\Require A labeled hypergraph $\Lambda$.\vs
\Ensure The set $Z = \minposet{N : N \geqhyp \Lambda}$.\vs
\State Initialize a stack $L$ and push the hypergraph $\Lambda$, and create an empty list $Y$.
\State \textbf{While} $L$ is not empty \textbf{do}
\begin{enumerate}[label=(\alph*)]
\item Pop the top hypergraph $\Delta$ from $L$.
\item Iterate through all distinct pairs of edges $e_{1},e_{2}\in \Delta$ until one of the following cases first occurs:

\item 
{\bf Case~1:} $e_{1}\in \Delta_{i},e_{2}\in \Delta_{j}$ with $e_{1}\subset e_{2}$ and $i>j$.
Then push onto $L$ the hypergraph induced by $\Delta\cup \{e_{1}\}$, where $e_{1}$ is assigned Type~$j$, see Definition~\ref{induced hypergraph}.\vs

\item 
{\bf Case~2:} $e_{1}\in \Delta_{i},e_{2}\in \Delta_{j}$ with 
$e_{1}\subset e_{2}$ and $j>i+\size{e_{2}\setminus e_{1}}$.
Then push onto $L$ the hypergraph induced by $\Delta\cup \{e_{2}\}$, where $e_{2}$ is assigned Type~$i+\size{e_{2}\setminus e_{1}}$.\vs

\item 
{\bf Case~3:} $e_{1}\in \Delta_{i},e_{2}\in \Delta_{j}$
with $v_{\Delta}(e_{1}\cap e_{2})+v_{\Delta}(e_{1}\cup e_{2})>i+j$.
Then, set  
\[
    s=v_{\Delta}(e_{1}\cap e_{2})+v_{\Delta}(e_{1}\cup e_{2})-i-j,
\]
and push onto the stack $L$ the following two labeled hypergraphs:
\begin{itemize}
    \item The hypergraph induced by $\Delta\cup \{e_{1}\cup e_{2}\}$, where $e_{1}\cup e_{2}$ is assigned Type~$v_{\Delta}(e_{1}\cup e_{2})-\lceil \textstyle\frac{s}{2} \rceil$;
    \item The hypergraph induced by $\Delta\cup \{e_{1}\cap e_{2}\}$, where $e_{1}\cap e_{2}$ is assigned Type~$v_{\Delta}(e_{1}\cap e_{2})-\lceil \textstyle\frac{s}{2} \rceil$.
\end{itemize}\vs
\item\label{step:genhyper:f} If we finish visiting all pairs of distinct edges $e_{1},e_{2}\in \Delta$ and none of these cases occurs, add the matroid $M_{\Delta}$ to $Y$.
\end{enumerate}

\State  Return the set $Z$ of maximal matroids among $Y$ using Algorithm~\ref{alg:algocompare5}.
\end{algorithmic}
\end{algorithm}
\vspace{-5mm}
\paragraph{Termination of Algorithm~\ref{algo 21}.}
Consider the partial order on hypergraphs on $[d]$ defined by $\Delta_{1}\geq \Delta_{2}$ if and only if for every $e\in (\Delta_{2})_{i}$, there exists $e\rq\in (\Delta_{1})_{j}$ with $j\leq i$ and $e\subset e\rq$. Under this ordering, the sequence of hypergraphs in the stack increases strictly at each step. Since the number of hypergraphs is finite, the process must eventually terminate. 
\vspace{-3mm}
\paragraph{Correctness of Algorithm~\ref{algo 21}.} 
We proceed by establishing two successive claims.

\begin{claim}\label{cla:further}
    Suppose that at a certain step, $\Delta$ is the top hypergraph of $L$ and let $N\geqhyp \Delta$. Then, either $N$ satisfies the conditions of Lemma~\ref{def m} or there exists a hypergraph $\widetilde{\Delta}$ that will be visited in the future, with $N\geqhyp \widetilde{\Delta}$. 
\end{claim} 

Let ``$\rank$" denote the rank function of $N$. 
Either $\Delta$ satisfies the conditions of Lemma~\ref{def m} or, by iterating over all distinct pairs of edges $e_{1},e_{2}\in \Delta$, one of the following cases occurs.

\begin{description}
\item[Case 1:] 
Since $N\geqhyp \Delta$, it follows that
$j \geq \rank(e_{2})\geq \rank(e_{1})$. Therefore, $N\geqhyp \Delta \cup \{e_{1}\}$.

\item[Case 2:] 
Since $N\geqhyp \Delta$, we have $\rank(e_{1})\leq i$, which implies 
\[
    \rank(e_{2})\; \leq\; 
    \rank(e_{1})+\size{e_{2}\setminus e_{1}}\;\leq\;
    i+\size{e_{2}\setminus e_{1}}.
\] 
Therefore, $N\geqhyp \Delta \cup \{e_{2}\}$.

\item[Case 3:] 
Since $N\geqhyp \Delta_{1}$, it follows that $\rank(e_{1})\leq i$ and $\rank(e_{2})\leq j$. By the submodularity of the rank function, this implies 
\[\rank(e_{1}\cup e_{2})+\rank(e_{1}\cap e_{2})\leq \rank(e_{1})+\rank(e_{2})\leq i+j=v_{\Delta}(e_{1}\cup e_{2})+v_{\Delta}(e_{1}\cap e_{2})-s.\]
Therefore, either 
\[\rank(e_{1}\cup e_{2})\leq v_{\Delta}(e_{1}\cup e_{2})-\lceil \textstyle\frac{s}{2} \rceil, \quad \text{or} \quad \rank(e_{1}\cap e_{2})\leq v_{\Delta}(e_{1}\cap e_{2})-\lceil \textstyle\frac{s}{2} \rceil.\]
This implies that either $N\geqhyp \Delta \cup \{e_{1}\cup e_{2}\}$ or $N\geqhyp \Delta \cup \{e_{1}\cap e_{2}\}$.
\end{description}
This establishes the Claim~\ref{cla:further}. To conclude the correctness, 
we prove the following claim.

\begin{claim}~\label{cla:final}
Let $N\geqhyp \Lambda$. Then, there exists $N\rq\in Z$ such that $N\leq N\rq$.
\end{claim}
According to the termination criterion of step~\ref{step:genhyper:f}, we eventually visit a hypergraph $\Delta$ satisfying the conditions of Lemma~\ref{def m}. In addition, by Claim~\ref{cla:further}, we have $N\geqhyp \Delta$ and then $N\leq M_{\Delta}$, by Lemma~\ref{unico}. By construction, there exists $N'\in Z$ such that $M_{\Delta} \leq N'$ and the claim follows by transitivity.
\qed

\begin{example}
Consider the labeled hypergraph $\Delta$ from Example~\ref{example: hypergraph delta}. Applying Algorithm~\ref{algo 21} we find that $\minabovehyp{\Delta}$ consists of the four matroids depicted in Figure~\ref{figure 41}.
\end{example}

\begin{figure}[H]
    \centering
    \includegraphics[width=0.8\textwidth, trim=0 0 0 0, clip]{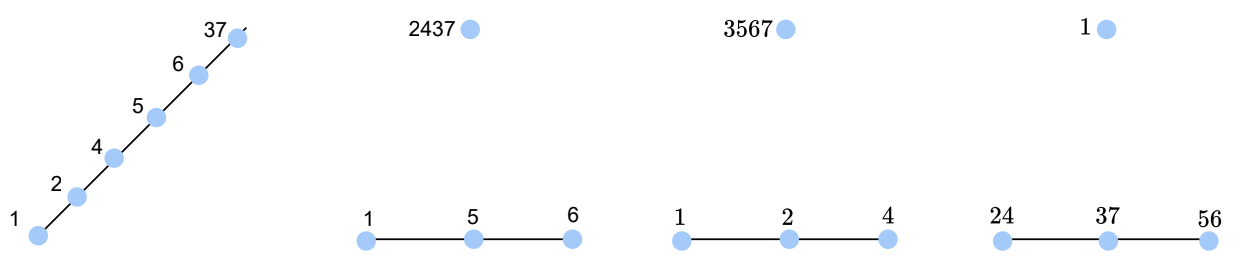}
    \caption{Matroids in $\minabovehyp{\Delta}$ for the hypergraph  $\Delta$ from Example~\ref{example: hypergraph delta}.}
    \label{figure 41}
\end{figure}

\subsection{Algorithm for identifying $\minabove{M}$}
\label{subsec:min_>}
Putting things together, we can now present an algorithm to answer Question~\ref{quest}. Its correctness is a direct consequence of Lemma~\ref{lema}.

\begin{algorithm}
\caption{Maximal matroid degenerations algorithm}\label{211}
\begin{algorithmic}[1]
\Require A matroid $M$;
\Ensure A set $Y = \minabove{M}$.\vs

\State Initialize $L_M$ as an empty set.
\For{$e\not \in \mathcal{D}(M)$}
\State define the labeled hypergraph $\Delta_{e}=\Delta_{M}\cup \{e\}$, where $e$ is assigned Type~$(\size{e}-1)$;

\State\label{step:genrank:update} 
update $L_M = L_M  \cup \minabovehyp{\Delta_{e}}$, using Algorithm~\ref{algo 21}.
\EndFor\vs
\State\label{step:genrank:compare} Using Algorithm~\ref{alg:algocompare5}, compare the matroids in the set $L_M$ to identify the maximal ones. Return the resulting set $Y$.\vs

\end{algorithmic}
\end{algorithm}

\begin{example}
Consider the Fano plane $M_{\textup{Fano}}$ from Example~\ref{example Fano}. Applying Algorithm~\ref{211}, we obtain that $\minabove{M_{\textup{Fano}}}$ consists of the $22$ matroids $U_{2,7}$, $M_{\textup{Fano}}(i)$, $A_{j}'$, and $B_{k}'$ defined in Example~\ref{example Fano}, for $i,j,k\in [7]$; see Figure~\ref{figure 4}.
\end{example}

\begin{remark}\label{limitation}
While Algorithm~\ref{211} is theoretically sound, it may fail to terminate in
practice. The main obstacle is the large amount of redundancy generated during
the search: in Step~\ref{step:genrank:update} many unnecessary candidate
matroids are produced, which forces Step~\ref{step:genrank:compare}, already
the computational bottleneck, to perform far more comparisons than needed.
A substantial part of this redundancy arises from the fact that different
sequences of added dependencies may lead to \emph{isomorphic} degenerations.
Detecting and removing such duplicates requires solving a matroid isomorphism
problem, a task at least as hard as graph isomorphism. This structural
difficulty makes a general termination criterion out of reach.

In the next section, we show how stratifying the search space drastically reduces
this redundancy. Even with these improvements, the complexity grows quickly with
the rank~$n$, which motivates our restriction to the case $n=4$. The following
section provides quantitative evidence supporting this choice.
\end{remark}

\begin{remark}\label{rem:limits}
As $n$ increases, the number of cases examined in Algorithm~\ref{algo 21} grows as $O(n^2)$, leading to an output space of size $O(2^{n^2})$ in the worst case. Moreover, since the algorithm is invoked $\textstyle O\left(\sum_{k=1}^n\binom{d}{k}\right)$ times, the final stages involve searching for maximal elements in a very large poset. If $N$ denotes the size of this poset, the concluding step requires $O(N^2)$ pairwise comparisons using Algorithm~\ref{alg:algocompare5}, each involving roughly $n^2$ tests.
\end{remark}

\section{Optimized algorithm for rank four}\label{algorithm}

To overcome the limitations of Algorithm~\ref{211} outlined in Remark~\ref{limitation}, we now present an optimized variant specifically designed for the rank-four case, along with its implementation.
Throughout, we fix a simple matroid $M$ of rank four on the ground set $[d]$, and denote its set of dependencies by $\mathcal{D}(M)$.
All labeled hypergraphs considered are defined with $n = 4$, and all matroids discussed are assumed to have rank at most four.

\subsection{Decomposition of the problem by stratification}\label{abcd}

The general problem of computing the maximal elements in the set $\{N : N < M\}$ quickly becomes intractable using the approach outlined in Subsection~\ref{subsec:min_>}. To address this, we propose partitioning this set into a stratification, which significantly reduces both the \emph{number of candidate matroids} and the \emph{number of matroid comparisons} required.

More precisely, we define the following subsets of $\{N : N < M\}$ for $i \geq 1$:
\[
    \ABCD{i} = \left\{ N < M : \forall\, 1 \leq j < i,\; \Ccal_j(N)=\Ccal_j(M) \text{\; and\; } \Ccal_i(N) \supsetneq \Ccal_i(M) \right\},
\]
where $\Ccal_i(N)$ denotes the set of circuits (i.e., minimal dependencies) of size $i$ in $N$ (see Definition~\ref{def:matroidcircuit}).
In our setting, where $M$ is a simple matroid of rank 4, these sets are:
\begin{equation*}
\begin{aligned}
&\ABCD{1}=\{N<M: \text{$\Ccal_{1}(N)\neq \emptyset$}\},\\
&\ABCD{2}=\{N<M: \text{$\Ccal_{1}(N)=\emptyset$ and $\Ccal_{2}(N)\neq  \emptyset$}\},\\
&\ABCD{3}=\{N<M: \text{$\Ccal_{1}(N)=\Ccal_{2}(N)= \emptyset$ and $\Ccal_{3}(N)\supsetneq \Ccal_{3}(M)$}\},\\
&\ABCD{4}=\{N<M:\text{$\Ccal_{1}(N)=\Ccal_{2}(N)= \emptyset,\  \Ccal_{3}(N)= \Ccal_{3}(M)$, and $\Ccal_{4}(N)\supsetneq \Ccal_{4}(M)$}\}.
\end{aligned}
\end{equation*}
Note that the decomposition
$\{N:N<M\}=\amalg_{i\geq 1}\, \ABCD{i}$ 
leads to the following equality: 
\begin{equation}\label{unio}
\minabove{M}=\minposet{\amalg_{i\geq 1}\, \minposet{\ABCD{i}}}.
\end{equation}
Therefore, the global problem is decomposed into smaller problems by stratifying $\{N : N < M\}$. In the following subsections, we focus on designing algorithms to solve each of these problems, by efficiently computing each $\minposet{\ABCD{i}}$. Afterward, we compute $\minabove{M}$ through a careful application of Algorithm~\ref{alg:algocompare5}, avoiding unnecessary comparisons.

\subsection{Lemmas from submodularity}
To compute the different strata introduced above, a direct application of Lemma~\ref{def m} is insufficient. Instead, we need to refine the criterion to leverage both the stratification framework and the assumption that the rank is 4. This refinement will allow us to effectively exploit the information specific to the subset $\ABCD{i}$ in question.

Throughout this subsection, we assume that all hypergraphs under consideration are reduced, consisting solely of edges of Type~$2$ and Type~$3$ (see Subsection~\ref{ssec:reduchyper}). We begin by reformulating the submodularity of the rank function in our setting, which is the inequality
\[
\rank(e_1 \cup e_2) + \rank(e_1 \cap e_2) \leq \rank(e_1) + \rank(e_2).
\]

\begin{lemma}\label{submo}
Let $\Delta$ be a labeled hypergraph, $e_{1},e_{2}\in \Delta$, and $N$ a loopless matroid satisfying $N\geqhyp \Delta$. The following properties hold for the rank function of $N$, denoted by $\rank$:
\begin{enumerate}[label=(\roman*)]
\item[{\rm (i)}] if $e_{1},e_{2}\in \Delta_{3}$, then either $\rank(e_{1}\cap e_{2})\leq 2$ or $\rank(e_{1}\cup e_{2})\leq 3$;
\item[{\rm (ii)}] if $e_{1}\in \Delta_{3}$ and $e_{2}\in \Delta_{2}$ then either $\rank(e_{1}\cap e_{2})\leq 1$ or $\rank(e_{1}\cup e_{2})\leq 3$;
\item[{\rm (iii)}] if $e_{1},e_{2}\in \Delta_{2}$, then  either $\rank(e_{1}\cap e_{2})\leq 1$ or $\rank(e_{1}\cup e_{2})\leq 2$;
\item[{\rm (iv)}] if $e_{1},e_{2}\in \Delta_{2}$ and $e_{1}\cap e_{2}\neq \emptyset$, then $\rank(e_{1}\cup e_{2})\leq 3$.
\end{enumerate}
\end{lemma}

Building on the previous lemma, we now characterize the conditions under which a labeled hypergraph defines a matroid, refining Lemma~\ref{def m} in this context.

\begin{lemma}\label{condi matr}
Let $\Delta$ be a labeled hypergraph on $[d]$ and suppose that for any distinct pair of edges $e_{1},e_{2}\in \Delta$ the following conditions hold:
\begin{itemize}
\item[{\rm (i)}] If $e_{1},e_{2}\in \Delta_{3}$ and $\size{e_{1}\cap e_{2}}\geq 3$, then $e_{1}\cap e_{2}\in \Delta_{2}$.
\item[{\rm (ii)}] If $e_{1}\in \Delta_{3}$ and $e_{2}\in \Delta_{2}$ with $\size{e_{1}\cap e_{2}}\geq 2$, then $e_{2}\subset e_{1}$.
\item[{\rm (iii)}] If $e_{1},e_{2}\in \Delta_{2}$, then $\size{e_{1}\cap e_{2}}\leq 1$.
\item[{\rm (iv)}] If $e_{1},e_{2}\in \Delta_{2}$ and $\size{e_{1}\cap e_{2}}=1$, then $e_{1}\cup e_{2}\subset x$ for some $x\in \Delta_{3}$.
\end{itemize}
Then, the following collection of sets forms the circuits of a matroid $M_{\Delta}$ on $[d]$:
\begin{equation}\label{mini}\Ccal=\min (\cup_{e\in \Delta_{2}}\textstyle{\binom{e}{3}}\cup_{e\in \Delta_{3}}\textstyle \binom{e}{4}\cup \binom{[d]}{5}),\end{equation} 
where 
$\min$ denotes the inclusion-minimal subsets.
\end{lemma}

\begin{proof}~
We must verify that $\Ccal$ satisfies the circuit elimination axiom. Specifically, for any distinct $C_{1},C_{2}\in \Ccal$ and $y\in C_{1}\cap C_{2}$, we need to show the existence of $C_{3}\in \Ccal$ with $C_{3}\subset (C_{1}\cup C_{2})\setminus \{y\}$. We consider the following cases:

\medskip
{\bf (1)} Suppose $\size{C_{1}}=\size{C_{2}}=5$. Then, we have $\size{(C_{1}\cup C_{2})\setminus \{y\}}\geq 5$, which implies that $(C_{1}\cup C_{2})\setminus \{y\}$ contains an element of $\Ccal$, as the latter contains $\textstyle \binom{[d]}{5}$.

\smallskip
{\bf (2)} Suppose $\size{C_{1}}=\size{C_{2}}=4$. If $\size{C_{1}\cap C_{2}}\leq 2$, then $\size{(C_{1}\cup C_{2})\setminus \{y\}}\geq 5$, and we conclude as in \textbf{(1)}.
If $\size{C_{1}\cap C_{2}}=3$, condition~(i) implies that $C_{1}\cap C_{2}$ is contained within an edge of $\Delta_{2}$, which contradicts the minimality of $C_{1}$ and $C_{2}$ in~\eqref{mini}.

\smallskip
{\bf (3)} Suppose $\size{C_{1}}=3$ and $\size{C_{2}}=4$. If $\size{C_{1}\cap C_{2}}=1$, then $\size{(C_{1}\cup C_{2})\setminus \{y\}}\geq 5$, and we conclude as in \textbf{(1)}.
Now suppose that $\size{C_{1}\cap C_{2}}= 2$. We know there exist $e_{1}\in \Delta_{3}$ and $e_{2}\in \Delta_{2}$ with $C_{1}\subset e_{2}$ and $C_{2}\subset e_{1}$. By condition~(ii), $e_{2}\subset e_{1}$, which implies $C_{1}\cup C_{2}\subset e_{1}$. Thus, any $4$-subset of $C_{1}\cup C_{2}$ contains an element of $\Ccal$. Since $\size{(C_{1}\cup C_{2})\setminus \{y\}}=4$, the claim follows.

\smallskip
{\bf (4)} Suppose $\size{C_{1}}=\size{C_{2}}=3$. By conditions~(iii) and~(iv), it follows that $C_{1}\cap C_{2}=\{y\}$ and $C_{1}\cup C_{2}\subset x$ for some $x\in \Delta_{3}$. In particular, the 4-subsets in $(C_{1}\cup C_{2})\setminus \{y\}$ are contained in an edge of $\Delta_{3}$, so the claim follows in this case.

\smallskip
This completes the proof.
\end{proof}

Since $M_{\Delta}$ is defined identically as in Lemma~\ref{def m}, the following characterization is a special case of Lemma~\ref{unico}.

\begin{lemma}\label{un min}
Let $\Delta$ be a labeled hypergraph as in Lemma~\ref{condi matr}. Then, $M_{\Delta}$ is the unique maximal matroid degeneration in $\minabovehyp{\Delta}$.
\end{lemma}

\subsection{Computing maximal matroid degenerations of a hypergraph}\label{min above}
Analogous to the general case, we present Algorithm~\ref{alg:computeleaf} to answer Question~\ref{quest 21}. The termination of the algorithm can be established similarly to that of Algorithm~\ref{algo 21}.

\begin{algorithm}[h] 
\caption{Maximal matroid degenerations of a hypergraph}\label{alg:computeleaf}
\begin{algorithmic}[1]
\Require A labeled hypergraph $\Lambda$ and $2 \leq v\leq 4$.
\Ensure The set $Z = \minposet{N \in \ABCD{v} : N \geqhyp \Lambda}$.\vs
\State Initialize a stack $L$ and push the hypergraph $\Lambda$, and create an empty list $Y$.\vs
\State \textbf{While} $L$ is not empty \textbf{do}
\begin{enumerate}[label=(\alph*)]
\item\label{step:rank4hyper:a} Pop the top hypergraph $\Delta$ from $L$.
\item Iterate through all distinct pairs of edges $e_{1},e_{2}\in \Delta$ until one of the following cases first occurs:\vs
\item \label{step:rank4hyper:c}
{\bf Case~1:} $e_{1},e_{2}\in \Delta_{3}$ with $\size{e_{1}\cap e_{2}}\geq 3$ and $e_{1}\cap e_{2}$ not contained within any edge of $\Delta_{2}$. 
Then push onto $L$ the hypergraphs:
\begin{itemize}
    \item The hypergraph induced by $\Delta\cup \{e_{1}\cup e_{2}\}$, where $e_{1}\cup e_{2}$ is assigned Type~3, see Definition~\ref{induced hypergraph};
    \item The hypergraph induced by $\Delta \cup \{e_{1}\cap e_{2}\}$, where $e_{1}\cap e_{2}$ is assigned Type~2, if $v\leq 3$.
\end{itemize}\vs

\item \label{step:rank4hyper:d}
{\bf Case~2:} $e_{1}\in \Delta_{3},e_{2}\in \Delta_{2}$ with $\size{e_{1}\cap e_{2}}\geq 2$ and $e_{2}\not \subset e_{1}$. 
Then push onto $L$ the hypergraphs:
\begin{itemize}
    \item The hypergraph induced by $\Delta\cup \{e_{1}\cup e_{2}\}$, where $e_{1}\cup e_{2}$ is assigned Type~3;
    
    \item $\widetilde{\Delta}_{\text{red}}$, where $\widetilde{\Delta}$ is the hypergraph induced by $\Delta\cup \{e_{1}\cap e_{2}\}$ and $e_{1}\cap e_{2}$ is assigned Type~1 in $\widetilde{\Delta}$, if $v=2$.
\end{itemize}\vs

\item \label{step:rank4hyper:e}
{\bf Case~3:} $e_{1},e_{2}\in \Delta_{2}$ with $\size{e_{1}\cap e_{2}}\geq 2$. Then push onto $L$ the  hypergraphs:
\begin{itemize}
    \item The hypergraph induced by $\Delta \cup \{e_{1}\cup e_{2}\}$, where $e_{1}\cup e_{2}$ is assigned Type~2;
    \item $\widetilde{\Delta}_{\text{red}}$, where $\widetilde{\Delta}$ is the hypergraph induced by $\Delta\cup \{e_{1}\cap e_{2}\}$ and $e_{1}\cap e_{2}$ is assigned Type~1 in $\widetilde{\Delta}$, if $v=2$
\end{itemize}\vs

\item \label{step:rank4hyper:f}
{\bf Case~4:} $e_{1},e_{2}\in \Delta_{2}$ with $\size{e_{1}\cap e_{2}}=1$ and $e_{1}\cup e_{2}$  not contained within any edge of $\Delta_{3}$.
Then push onto $L$ the hypergraph induced by $\Delta\cup \{e_{1}\cup e_{2}\}$, where $e_{1}\cup e_{2}$ is assigned Type~3.\vs

\item\label{step:rank4hyper:g} If we finish visiting all pairs of distinct edges $e_{1},e_{2}\in \Delta$ and none of these cases occurs, add the matroid $M_{\Delta}$ to 
$Y$. 
\end{enumerate}
\State Return the set $Z$ of maximal matroids among $Y$ using Algorithm~\ref{alg:algocompare5}.
\end{algorithmic}
\end{algorithm}

\paragraph{Correctness of Algorithm~\ref{alg:computeleaf}.}
We proceed by establishing the following two claims.
\begin{claim}\label{cla:future}
    Suppose that at step~\ref{step:rank4hyper:a}, $\Delta$ is the top hypergraph of $L$, and let $N\in \ABCD{v}$ be a matroid satisfying $N\geqhyp \Delta$. Then, either $\Delta$ satisfies the conditions of Lemma~\ref{condi matr} or there exists a hypergraph $\widetilde{\Delta}$ that will be visited in the future, with $N\geqhyp \widetilde{\Delta}$.
\end{claim}

Either $\Delta$ satisfies the conditions of Lemma~\ref{condi matr} or, by iterating over all distinct pairs of edges $e_{1},e_{2}\in \Delta$, one of the following cases occurs.

\medskip{\bf Case 1:} By Lemma~\ref{submo}(i), either
    $\rank_{N}(e_{1}\cap e_{2}) \leq 2$ or $\rank_{N}(e_{1}\cup e_{2}) \leq 3$.
Thus, if $v \leq 3$, at least one of the hypergraphs in step~\ref{step:rank4hyper:c} satisfies the claim.  
If $v = 4$, we show that the second inequality holds, so the first hypergraph in step~\ref{step:rank4hyper:c} satisfies the claim. Indeed, in this case, $e_{1} \cap e_{2}$ is not contained in any edge of $\Delta_{2}$, and therefore, not in any edge of $(\Delta_{M})_{2}$. Since $\size{e_{1} \cap e_{2}} \geq 3$, it follows that $\rank_{M}(e_{1} \cap e_{2}) = 3$. But since $N \in \ABCD{4}$, we have $\Ccal_{3}(N) = \Ccal_{3}(M)$, hence $\rank_{N}(e_{1} \cap e_{2}) = 3$. Therefore, the first inequality cannot occur.

\medskip{\bf Case  2:} By Lemma~\ref{submo}(ii), either
    $\rank_{N}(e_{1} \cap e_{2}) \leq 1$ or $\rank_{N}(e_{1} \cup e_{2}) \leq 3$.
If $v = 2$, at least one of the hypergraphs in step~\ref{step:rank4hyper:d} satisfies the claim.  
If $v \geq 3$, the first inequality cannot occur, as $N$ would exhibit a double point, given that $\size{e_{1} \cap e_{2}} \geq 2$. This is impossible since $N \in \ABCD{v}$ with $v \geq 3$. Hence, the first hypergraph in step~\ref{step:rank4hyper:d} satisfies the claim.

\medskip{\bf Case 3:} By Lemma~\ref{submo}(iii), either
    $\rank_{N}(e_{1} \cap e_{2}) \leq 1$ or $\rank_{N}(e_{1} \cup e_{2}) \leq 2$.
Therefore, if $v = 2$, at least one of the hypergraphs in step~\ref{step:rank4hyper:e} satisfies the claim. If $v \geq 3$, we conclude as in Case 2.

\medskip{\bf Case  4:} 
By Lemma~\ref{submo}(iv), $\rank_{N}(e_{1}\cup e_{2})\leq 3$, hence the hypergraph in step~\ref{step:rank4hyper:f} satisfies the claim.

\medskip
To conclude the proof of correctness, it suffices to verify the following claim.
\begin{claim}\label{cla:minimal}
    Let $N\in \ABCD{v}$, with $N\geqhyp \Lambda$. Then, there exists $N\rq \in Z$, such that $N\leq N\rq$.
\end{claim}
 The proof of Claim~\ref{cla:minimal} is identical to the proof of Claim~\ref{cla:final} for the general case of Algorithm~\ref{algo 21}, and is therefore omitted.\qed

\begin{example}\label{ex: hypergraph lambda}
Consider the labeled hypergraph $\Lambda=\Lambda_{2}\cup \Lambda_{3}$ on $[6]$, where
\[
\Lambda_{2}=\{\{1,3,4\}\}, \qquad 
\Lambda_{3}=\{\{1,2,5,6\},\{3,4,5,6\}\}.
\]
Applying Algorithm~\ref{alg:computeleaf} to $\Lambda$, we find that $\minabovehyp{\Lambda}$ consists of the three matroids depicted in Figure~\ref{figure maximals} (from left to right):
\begin{itemize}
\item a rank-three matroid $M_{1}$ in which all points lie on the same hyperplane,
\item a matroid $M_{2}$ in which $2$ is a coloop and both $\{1,3,4\}$ and $\{1,5,6\}$ are dependent,
\item a matroid $M_{3}$ in which $\{3,4\}$ forms a double point.
\end{itemize}
\end{example}

 \begin{figure}[H]
    \centering
    \begin{subfigure}[b]{0pt}
  \phantomsubcaption
  \label{fig:5a}
\end{subfigure}%
\begin{subfigure}[b]{0pt}
  \phantomsubcaption
  \label{fig:5b}
\end{subfigure}%
\begin{subfigure}[b]{0pt}
  \phantomsubcaption
  \label{fig:5c}
\end{subfigure}
    \includegraphics[width=0.7\textwidth, trim=0 0 0 0, clip]{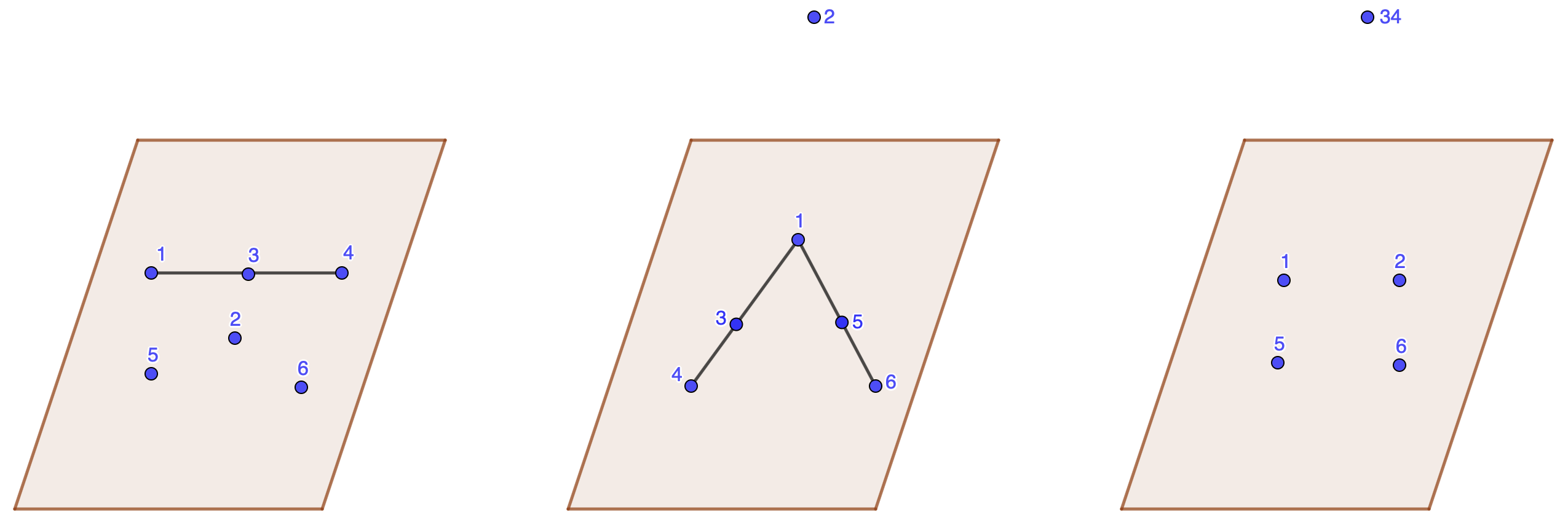}
    \caption{Matroids in $\minabovehyp{\Lambda}$ for the hypergraph $\Lambda$ from Example~\ref{ex: hypergraph lambda}: 
(a) $M_1$ (Left); 
(b) $M_2$ (Center); 
(c) $M_3$ (Right).}
    \label{figure maximals}
\end{figure}

\begin{remark}
For $v = 4$, each hypergraph $\Delta \in L$ preserves the same Type~2 edges as those in $\Delta_{M}$. As a result, Cases~3 and 4 do not arise. Thus, in this case, the stack $L$ consists of a single hypergraph throughout the entire process, and $Z$ contains a single matroid.
\end{remark}

\subsection{Algorithm for identifying $\minabove{M}$}
We present an algorithm to determine $\minabove{M}$. From Equation~\eqref{unio}, we have:
\[
    \minabove{M} = \minposet{\amalg_{1 \leq i \leq 4} \, \minposet{\ABCD{i}}}.
\]
The algorithm computes $\minposet{\ABCD{i}}$ for $1 \leq i \leq 4$ using Algorithm~\ref{alg:computeleaf}, then identifies the maximal matroids among these with Algorithm~\ref{alg:algocompare5}. A Python implementation is available at
\url{https://github.com/rprebet/maximal_matroids}.
We assume the input is a simple matroid of rank four, which is valid since any matroid $M$ can be reduced to a simple matroid $M_{\textup{red}}$ by removing loops and identifying double points (see Subsection~\ref{ssec:remove2loops}). As shown in Lemmas~\ref{lem:removeloops} and~\ref{lem:remove2points}, there is a correspondence between the maximal matroid degenerations of $M$ and $M_{\textup{red}}$, simplifying the problem.

\begin{algorithm}
\caption{Optimized maximal matroid degenerations}\label{algo 2}
\begin{algorithmic}[1]
\Require A simple matroid $M$ of rank 4.
\Ensure The set $\minabove{M}$\vs.
\State\label{step:algmin4:1} Compute $\minposet{\ABCD{1}}=\{M(i):i\in [d]\}$, using the notation of  Definition~\ref{def:designloop}.\vs

\State\label{step:algmin4:2} Compute $\minposet{\ABCD{v}}$ for each $v=2,3,4$ as follows:
\begin{enumerate}[label={(\alph*)}]
\item Construct the set $\Pcal_v$, of all $x\subset [d]$ of size $v$ such that $x\not \in \mathcal{D}(M)$.

\item For each $x\in \Pcal_v$, do the following:
\begin{enumerate}[label=(\roman*)]
\item Construct the labeled hypergraph $\Delta_{x}=(\Delta_{M}\cup \{x\})_{\text{red}}$ where $x$ is assigned Type~$v-1$.

\item Compute the set $W_x$ of maximal matroid degenerations of $\Delta_x$ in $\ABCD{v}$, by calling Algorithm~\ref{alg:computeleaf} on input $\Delta_x$ and $v$.
\end{enumerate}

\item Finally, using Algorithm~\ref{alg:algocompare5}, return the maximal matroids in the set 
$\cup_{x \in \Pcal_v} W_x$\vs
\end{enumerate}

\State\label{step:algmin4:3} Using Algorithm~\ref{alg:algocompare5} compute successively:
\begin{enumerate}[label=(\alph*)]
\item  $L_{3} = \left\{ N \in \minposet{\ABCD{3}}: 
\nexists N' \in \minposet{\ABCD{4}}
\hphantom{ \cup L_{3} \cup L_{2}}\,\,
,\; N\leq N\rq \right\}$;\vs

\item  $L_{2} = \left\{ N \in \minposet{\ABCD{2}}:
\nexists N' \in \minposet{\ABCD{4}}\cup L_{3}
\hphantom{\cup L_{2}}\;\,
,\; N\leq N\rq \right\}$;\vs

\item  $L_{1} = \left\{ N \in \minposet{\ABCD{1}}:
\nexists N' \in \minposet{\ABCD{4}} \cup L_{3} \cup L_{2}
,\; N\leq N\rq \right\}$.\vs

\end{enumerate}
\State Return $L_{1} \cup L_{2} \cup L_{3} \cup \minposet{\ABCD{4}}$
\end{algorithmic}
\end{algorithm}

\medskip
\noindent {\bf Correctness of Algorithm~\ref{algo 2}.} 
First, we show that each $\minposet{\ABCD{v}}$ is correctly computed at steps~\ref{step:algmin4:1} and~\ref{step:algmin4:2}.
The case for $\minposet{\ABCD{1}}$ is straightforward, as adding more than one loop to $M$ results in non-maximal elements, and the matroids obtained at step~\ref{step:algmin4:1} are not pairwise comparable.
For $v \geq 2$, we aim to establish the following equality:
\begin{equation}\label{sz2}
\textup{max} \{\bigcup_{x\in\Pcal_v}W_{x}\} = \minposet{\ABCD{v}}
\end{equation}
where 
$\Pcal_v$ is the set of all $x\subset [d]$ of size $v$ such that $x\not \in \mathcal{D}(M)$.

\medskip

To prove the inclusion $\supset$, let $N \in \minposet{\ABCD{v}}$. Since $N < M$, there exists $x \in \mathcal{D}(N) \setminus \mathcal{D}(M)$ with $\size{x} = v$. Then, $x \in \Pcal_v$ and $N \geqhyp \Delta_{x}$. Therefore, we have:
\[
N\in \minposet{\ABCD{v}}\cap \{N\rq:N\rq\geqhyp \Delta_{x}\}
\subset \minposet{N\rq\in \ABCD{v}:N\rq \geqhyp \Delta_{x}}=W_{x}.
\]
Moreover, since $W_x \subset \ABCD{v}$, it follows that $N\in \ABCD{v}$ belongs to $\textup{max} \{\cup_{x\in\Pcal_v} \,W_{x}\}$.

\smallskip

To establish the inclusion $\subset$, 
let $N \in \textup{max} \{\cup_{x\in\Pcal_v} \,W_{x}\}$. By contradiction, suppose that $N \not \in \minposet{\ABCD{v}}$. Then, there exists $N' \in \ABCD{v}$ satisfying $N < N' < M$.
Since $N\rq<M$, there exists $x\in \mathcal{D}(N\rq)\setminus \mathcal{D}(M)$ with $\size{x}=v$. Then, $x\in \Pcal_v$ and $N< N\rq \geqhyp \Delta_{x}$.
Hence, 
\[
    N\in W_{x}=\minposet{N^{\ast}\in \ABCD{v}:N^{\ast} \geqhyp \Delta_{x}},
\]
implying that $N\rq \leq N$, which contradicts $N < N'$. This shows that $N\in \minposet{\ABCD{v}}$.

\medskip
It remains to show that, after step~\ref{step:algmin4:3}, we have
   $ \minabove{M} = L_1 \cup L_2 \cup L_3 \cup \textup{max} \{\ABCD{4}\}$.
This follows directly from the definition of $\ABCD{v}$, as no matroid in $\ABCD{v}$ can be smaller than or equal to a matroid in $\ABCD{i}$ for any $i < v$. Specifically, we have $\Ccal_{j}(N) = \Ccal_{i}(M)$ for all $1 \leq j \leq i < v$. \qed

\begin{example}
Consider the simple rank-four matroid $M$ on $[6]$ whose set of circuits is
\[
\{\{1,2,3,4\},\{3,4,5,6\},\{1,2,5,6\}\}.
\]
Applying Algorithm~\ref{algo 2} to $M$, we obtain that $\minabove{M}$ consists of the following $10$ matroids:
\begin{itemize}
\item the uniform matroid of rank three $U_{3,6}$,
\item the matroid in which $2$ is a coloop and both $\{1,3,4\}$ and $\{1,5,6\}$ are dependent (see Figure~\ref{fig:5b}), together with the five analogous matroids obtained by declaring any of the points in $\{1,3,4,5,6\}$ to be a coloop,
\item the three matroids in which the pairs $\{3,4\}$, $\{1,2\}$, or $\{5,6\}$ form a double point (see Figure~\ref{fig:5c}), for the case in which $\{3,4\}$ is a double point).
\end{itemize}
\end{example}

\begin{remark}
In principle, this optimized algorithm could be extended from rank $n=4$ to arbitrary $n$. The main theoretical challenge lies in extending the approach used in Lemma~\ref{condi matr} to refine Lemma~\ref{def m} for arbitrary ranks. However, in the general case, the relevant cases to consider do not appear to be immediately clear. 
Furthermore, as discussed in Remark~\ref{rem:limits}, the proposed optimization does not appear to provide a substantial improvement in the algorithm’s efficiency for practical applications. This observation is supported by experimental results, which indicate similar performance even for rank 5.
\end{remark}

\section{Examples}\label{examples}

In this section, we apply our algorithm to several classical rank-four
configurations to illustrate its effectiveness in computing maximal matroid
degenerations and, consequently, the irreducible decompositions of their circuit
varieties.  These computations lie beyond the capabilities of current symbolic
and numerical algebra systems.  The proofs of the technical lemmas used here are
provided in Section~\ref{appen}.

Each example highlights a distinct structural phenomenon.  
The Vámos matroid illustrates extreme non-realizability; the Steiner system
$S(3,4,8)$ shows how high symmetry drastically reduces the number of maximal
degenerations; the dual of the Fano matroid connects our approach with
projective geometries; and the dual of $K_{3,3}$ marks the computational
boundary of the method.  
These examples also illustrate the practical efficiency of Algorithm~\ref{alg:computeleaf}: in all
cases, including the largest one, the dual of $K_{3,3}$, the algorithm terminates
quickly and produces the complete list of maximal degenerations.


\medskip
We begin by recalling the notion of the automorphism group of a matroid, which will be used in the examples of this section to group the maximal matroid degenerations into symmetry classes.

\begin{definition}
Let $M$ be a matroid on the ground set $[d]$. The {\em automorphism group} of $M$, denoted $\text{Aut}(M)$, is the subgroup of all permutations $\sigma\in \mathbb{S}_{d}$ that preserve dependent sets of $M$, meaning that $X\in \mathcal{D}(M)$ if and only if $\sigma(X)\in \mathcal{D}(M)$.
\end{definition}

\subsection{V\'{a}mos matroid}\label{Vamos}

Consider the {\bf V\'{a}mos matroid} $M_{\textup{V\'{a}mos}}$ depicted in Figure~\ref{fig:6a}. This matroid is a paving matroid of rank four that is not representable over any field. Its collection of dependent hyperplanes is given by:
\[\{\{1,2,3,4\},\{3,4,5,6\},\{5,6,7,8\},\{7,8,1,2\},\{3,4,7,8\}\}.\]
The non-realizability of $M_{\textup{V\'{a}mos}}$ arises from the absence of the dependency $\{1,2,5,6\}$. By incorporating this missing dependent hyperplane, we obtain a realizable matroid, denoted by $A$, as shown in Figure~\ref{fig:6b}.

\begin{figure}[H]
    \centering
    \begin{subfigure}[b]{0pt}
  \phantomsubcaption
  \label{fig:6a}
\end{subfigure}%
\begin{subfigure}[b]{0pt}
  \phantomsubcaption
  \label{fig:6b}
\end{subfigure}%
\begin{subfigure}[b]{0pt}
  \phantomsubcaption
  \label{fig:6c}
\end{subfigure}
    \includegraphics[width=0.7\textwidth, trim=0 0 0 0, clip]{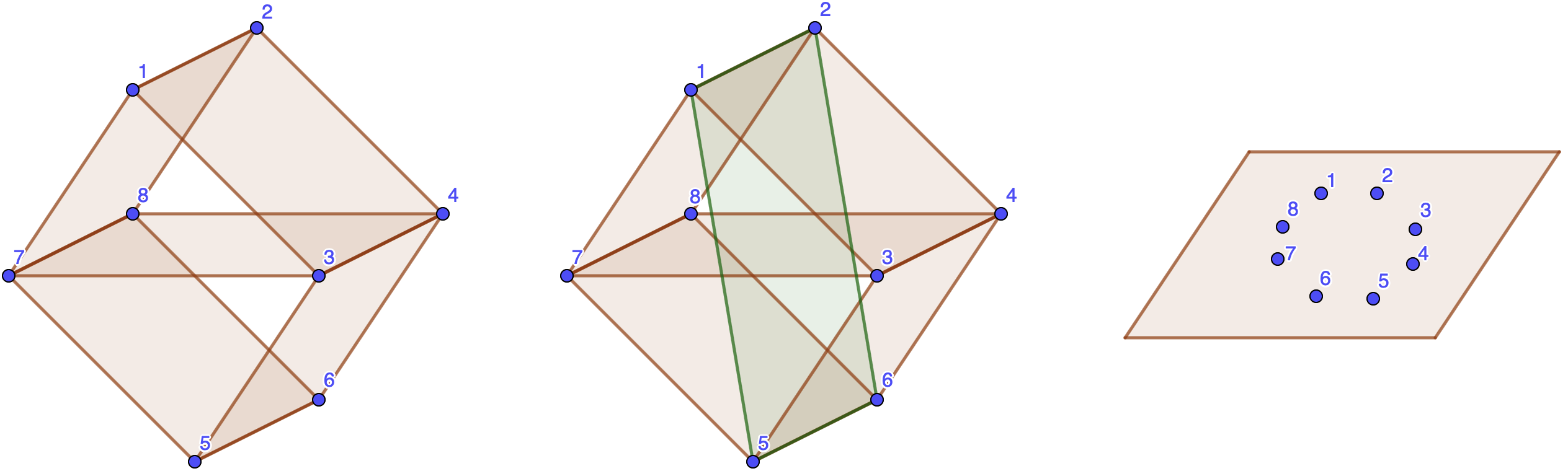}
    \caption{(a) V\'{a}mos matroid $M_{\textup{V\'{a}mos}}$ (Left); (b) Realizable degeneration $A$ of V\'{a}mos matroid (Center); (c) The uniform matroid $U_{3,8}$ (Right).}
    \label{figure 1}
\end{figure}

Next, we introduce several matroids that are smaller than $M_{\textup{V\'{a}mos}}$ in the 
weak order poset: 

\begin{itemize}
\item[{\rm (i)}] Let $B_{1}$ and $B_{2}$ denote the matroids in which the points $\{1,2,5,6,a,b\}$ lie on the same hyperplane, where $\{a,b\}=\{3,4\}$ for $B_{1}$ and $\{a,b\}=\{7,8\}$ for $B_{2}$; see Figure~\ref{fig:7a}.
\item[{\rm (ii)}] Let $C_{1}$ and $C_{2}$ denote the matroids in which the points $\{3,4,7,8,a,b\}$ lie on the same hyperplane, where $\{a,b\}=\{5,6\}$ for $C_{1}$ and $\{a,b\}=\{1,2\}$ for $C_{2}$; see Figure~\ref{fig:7b}.
\item[{\rm (iii)}] Let $D_{1}$ and $D_{2}$ denote the matroids obtained from $M_{\textup{V\'{a}mos}}$ by identifying the points $\{3,4\}$ and  $\{7,8\}$, respectively; see Figure~\ref{fig:7c}.
\item[{\rm (iv)}] Let $E_{1}$ and $E_{2}$ denote the matroids obtained from $M_{\textup{V\'{a}mos}}$ by identifying the points $\{1,2\}$ and  $\{5,6\}$, respectively; see Figure~\ref{fig:7d}.
\item[{\rm (v)}] Consider the matroid of rank four shown in Figure~\ref{figure 2} (Right), with the set of circuits of size three $\{\{1,3,4\},\{1,5,6\},\{1,7,8\}\}$ and set of circuits of size four $\{\{5,6,7,8\},\{3,4,5,6\},\{3,4,7,8\}\}$. We denote by $F_{i}$, with $i\in [8]$, the matroids obtained by applying an automorphism of $M_{\textup{V\'{a}mos}}$ to this matroid; see Figure~\ref{fig:7e}.

\end{itemize}

\begin{figure}[H]
    \centering
    \begin{subfigure}[b]{0pt}
  \phantomsubcaption
  \label{fig:7a}
\end{subfigure}%
\begin{subfigure}[b]{0pt}
  \phantomsubcaption
  \label{fig:7b}
\end{subfigure}%
\begin{subfigure}[b]{0pt}
  \phantomsubcaption
  \label{fig:7c}
\end{subfigure}
\begin{subfigure}[b]{0pt}
  \phantomsubcaption
  \label{fig:7d}
\end{subfigure}
\begin{subfigure}[b]{0pt}
  \phantomsubcaption
  \label{fig:7e}
\end{subfigure}
    \includegraphics[width=0.8\textwidth, trim=0 0 0 0, clip]{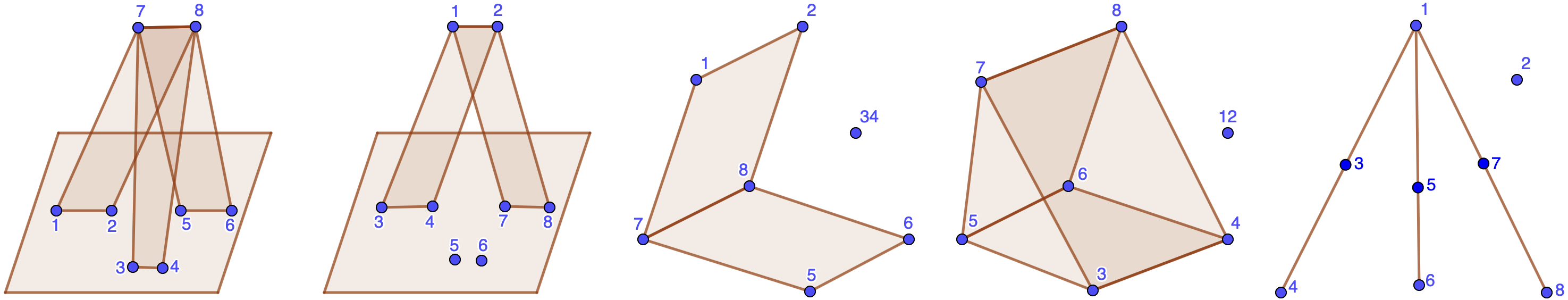}
    \caption{(a) Matroid $B_{1}$ (b) Matroid $C_{1}$ (c) Matroid $D_{1}$ (e) Matroid $E_{1}$ (f) Matroid $F_{1}$.}
    \label{figure 2}
\end{figure}

Observe that the matroids $C_{1}$ and $C_{2}$ satisfy the conditions of Theorem~\ref{teo ir} (ii), which implies $V_{\Ccal(C_{i})}=V_{C_{i}}\cup V_{U_{3,8}}$,
for $i=1,2$, where $U_{3,8}$ denotes the uniform matroid of rank three on $[8]$, see Figure~\ref{fig:6c}. On the other hand, the matroids $D_{1}$ and $D_{2}$ satisfy the conditions of Theorem~\ref{teo ir} (i), which implies that their matroid varieties and their circuit varieties coincide. We will also use the following lemma:

\begin{lemma}\label{lema red}
Let $A$ be the matroid in Figure~\ref{fig:6b}. The following statements hold:
\begin{itemize}
\item[{\rm (i)}] For $i=1,2$, $V_{\Ccal(E_{i})}\subset V_{A}\cup V_{\Ccal(C_{i})}$.
\item[{\rm (ii)}] We have 
\vspace{-5mm}\begin{equation}\label{fi}
\bigcup_{i=1}^{8}V_{\Ccal(F_{i})}\subset V_{A}\bigcup_{j=1}^{2}V_{\Ccal(D_{j})}\bigcup_{k=1}^{2}V_{\Ccal(E_{k})}.
\end{equation}
\item[{\rm (iii)}] We have 
\vspace{-5mm}\begin{equation}\label{bi}
\bigcup_{i=1}^{2}V_{\Ccal(B_{i})}\subset V_{A}\bigcup_{j=1}^{2}V_{\Ccal(C_{j})}\bigcup_{k=1}^{2}V_{\Ccal(D_{k})}\bigcup_{l=1}^{2}V_{\Ccal(E_{l})}\bigcup_{r=1}^{8}V_{\Ccal(F_{r})}.
\end{equation}
\end{itemize}
\end{lemma}

Let $\ABCD{i}$ for $i\in [4]$ denote the collections of matroids defined as in Subsection~\ref{abcd}. We will need the following lemma:

\begin{lemma}\label{prop}
Let $M$ be a simple matroid of rank four. Then
\[V_{\Ccal(M)}=V_{M}\bigcup_{N\in \minposet{\cup_{i=1}^{3}\ABCD{i}}}V_{\Ccal(N)}\bigcup_{N\rq \in \ABCD{4}}V_{N\rq}.\]
\end{lemma}

By applying Lemma~\ref{prop} to $M_{\textup{V\'{a}mos}}$, we obtain
\begin{equation}\label{union}
V_{\Ccal(M_{\textup{V\'{a}mos}})}=\bigcup_{N\in \minposet{\cup_{i=1}^{3}\mathcal{S}_{i}(M_{\text{V\'{a}mos}})}}V_{\Ccal(N)}\bigcup_{N\rq \in \mathcal{S}_{4}(M_{\text{V\'{a}mos}})}V_{N\rq},
\end{equation}
Note that in this expression we are using the non-realizability of $M_{\textup{V\'{a}mos}}$. 
Furthermore, by applying Algorithm~\ref{algo 2}, we deduce that all the maximal matroids in $\cup_{i=1}^{3}\mathcal{S}_{i}(M_{\text{V\'{a}mos}})$ are smaller than or equal to some matroid from the set $\{B_{i},C_{i},D_{i},E_{i},F_{i}\}$. Using this, along with Equation~\eqref{union} and Lemma~\ref{lema red}, we obtain that
\begin{equation}\label{sec eq}
V_{\Ccal(M_{\textup{V\'{a}mos}})}=V_{A}\cup V_{U_{3,8}}\bigcup_{i=1}^{2}V_{C_{i}}\bigcup_{j=1}^{2}V_{D_{j}}\bigcup_{N \in \mathcal{S}_{4}(M_{\text{V\'{a}mos}})}V_{N}.
\end{equation}
Moreover, we have the following lemma:
\begin{lemma}\label{sec lem}
The matroid varieties $V_{N}$ for $N \in \mathcal{S}_{4}(M_{\textup{V\'{a}mos}})$ are redundant in~\eqref{sec eq}.
\end{lemma}

\begin{proposition}
The irreducible decomposition of the circuit variety of $M_{\textup{V\'{a}mos}}$ is 
\[V_{\Ccal(M_{\textup{V\'{a}mos}})}=V_{A}\cup V_{U_{3,8}}\bigcup_{i=1}^{2}V_{C_{i}}\bigcup_{j=1}^{2}V_{D_{j}}.\]
\end{proposition}

\begin{proof}
Using Lemma~\ref{sec lem}, Equation~\eqref{sec eq} directly leads to the above decomposition.
We know that $V_{A}$ is irreducible, as indicated in \textup{\cite[Table~5.1]{corey2023singular}}. Since all matroids in this decomposition, except for $A$, are inductively connected, Theorem~\ref{teo ir} implies that their associated varieties are all irreducible. Moreover, it is easy to verify the non-redundancy of this decomposition. Thus, we have the irreducible decomposition of $V_{M_{\textup{V\'{a}mos}}}$. 
\end{proof}

\subsection{The unique $S(3,4,8)$}\label{uniq}

To introduce the matroid $M_{\textup{Steiner}}$, the focus of this section, we begin by recalling the notion of Steiner systems.

\begin{definition}\label{steiner}
A {\em Steiner system} with parameter dependencies $n\leq k\leq d$, denoted $S(n,k,d)$, consists of a collection of $k$-elements subsets of $[d]$, called {\em blocks}, such that every $n$-element subset of $[d]$ is contained in exactly one block. Each Steiner system $S(n-1,k,d)$ defines an $n$-paving matroid on $[d]$, where the blocks correspond to the dependent hyperplanes. For further details, see \cite{van2018number}.
\end{definition}

\begin{example}
The~collection~$\{1,2,4\},\{1,3,7\},\{1,5,6\},\{2,3,5\},\{4,5,7\},\{2,6,7\},\{3,4,\\ 6\}$ of subsets of $[7]$
constitutes the blocks of an $S(2,3,7)$ Steiner system. This system defines a point-line configuration, where each block corresponds to a line. The associated matroid is the \emph{Fano plane}, see Figure~\ref{fano}.
\end{example}

Consider the unique {\bf Steiner quadruple system} $S(3,4,8)$, which defines a paving matroid of rank four, denoted $M_{\textup{Steiner}}$, with the following dependent hyperplanes: 
\begin{equation*}
\begin{aligned}
\{&\{1,2,4,8\},\{2,3,5,8\},\{3,4,6,8\},\{4,5,7,8\},\{1,5,6,8\},\{2,6,7,8\},\{1,3,7,8\},\\
& \{3,5,6,7\},\{1,4,6,7\},\{1,2,5,7\},\{1,2,3,6\},\{2,3,4,7\},\{1,3,4,5\},\{2,4,5,6\}\},
\end{aligned}
\end{equation*}
This matroid can also be viewed as the set of points in the three-dimensional affine plane over $\mathbb{F}_{2}$.
Using Algorithm~\ref{algo 2}, we establish that the set $\minabove{M_{\textup{Steiner}}}$ consists of the following matroids: 

\begin{itemize}
\item[{\rm (i)}] The matroids obtained from $M_{\textup{Steiner}}$ by identifying the four points that lie in a dependent hyperplane; see Figure~\ref{fig:8a}.
\item[{\rm (ii)}] A matroid where $\rank(\{1,2,3,4,5,6,7\})=3$, and the circuits of size three are
\[\{1,2,4\},\{1,3,7\},\{1,5,6\},\{2,3,5\},\{4,5,7\},\{3,4,6\},\{2,6,7\},\]
along with all the matroids obtained from this by applying an automorphism; see Figure~\ref{fig:8b}. 
\item[{\rm (iii)}] The uniform matroid $U_{3,8}$; see Figure~\ref{fig:8c}.
\item[{\rm (iv)}] The matroids $M_{\textup{Steiner}}(i)$ for $i\in [8]$.
\end{itemize}

There are $14$ matroids of type  (i) and $8$ matroids of   type (ii). We label these as $A_{i}$ and $B_{j}$, respectively, where $i\in [14]$ and $j\in [8]$.

\begin{figure}[H]
    \centering
    \begin{subfigure}[b]{0pt}
  \phantomsubcaption
  \label{fig:8a}
\end{subfigure}%
\begin{subfigure}[b]{0pt}
  \phantomsubcaption
  \label{fig:8b}
\end{subfigure}%
\begin{subfigure}[b]{0pt}
  \phantomsubcaption
  \label{fig:8c}
\end{subfigure}
    \includegraphics[width=0.8\textwidth, trim=0 0 0 0, clip]{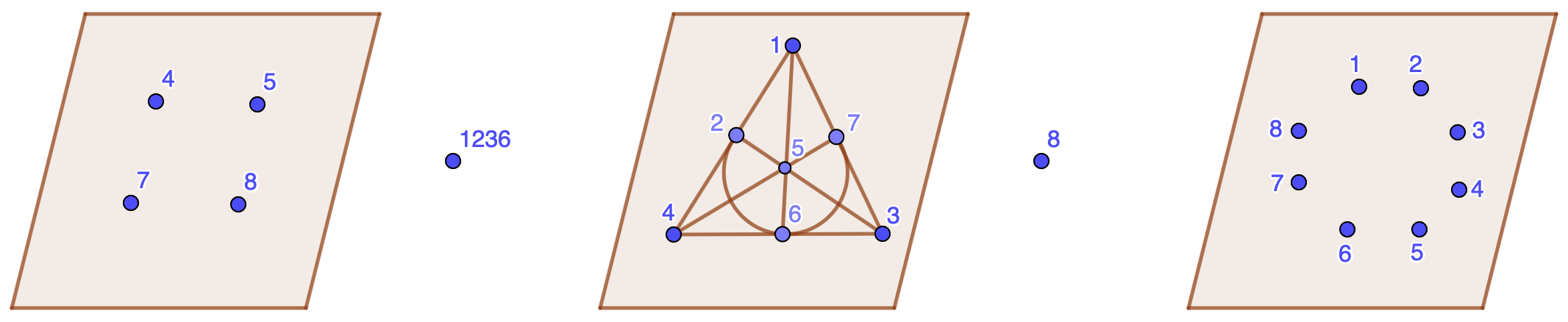}
    \caption{Maximal matroid degenerations 
    of $S(3,4,8)$ in items (i), (ii), and (iii) above: (a) $A_{1}$ (Left); (b) $B_{1}$ (Center); (c) $C_{1}$ (Right).  
    }
    \label{figure 3}
\end{figure}


\begin{lemma}\label{not real}
The matroids $M_{\text{Steiner}}(i)$ for $i\in [8]$ are not realizable.
\end{lemma}

Observe that Lemma~\ref{not real} implies that the matroid $M_{\textup{Steiner}}$ is not realizable. Consequently, combining the above collection of maximal matroids together with Proposition~\ref{deco circ}, gives the following decomposition:

\begin{equation}\label{dec ste}
V_{\Ccal(M_{\textup{Steiner}})}=V_{\Ccal(U_{3,8})}\bigcup_{i=1}^{8}V_{\Ccal(M_{\textup{Steiner}}(i))}\bigcup_{j=1}^{14}V_{\Ccal(A_{j})}\bigcup_{k=1}^{8}V_{\Ccal(B_{k})}.
\end{equation}

The matroids $U_{3,8}$ and $A_{j}$ are nilpotent. Consequently, by Theorem~\ref{teo ir}, their matroid varieties coincide with their circuit varieties. On the other hand, each matroid $B_{k}$ is the direct sum of the trivial rank-one matroid on a single element and the Fano plane. Recall that the Fano plane, which we denote $M_{\textup{Fano}}$ and is depicted in Figure~\ref{fano}, is the point-line configuration from Example~\ref{example Fano}. 
In \cite{liwski2025minimal}, the irreducible decomposition of $V_{\Ccal(M_{\textup{Fano}})}$ was determined as follows:
\begin{equation}\label{dec fano}
V_{\Ccal(M_{\text{Fano}})}=V_{U_{2,7}}\bigcup_{i=1}^{7}V_{M_{\text{Fano}}(i)}\bigcup_{j=1}^{7}V_{A_{j}\rq}\bigcup_{k=1}^{7}V_{B_{k}\rq},\end{equation}
where the matroids $A_{j}\rq$ and $B_{k}\rq$ are defined in Example~\ref{example Fano} 
(see also 
Figure~\ref{figure 4}).

Equation~\eqref{dec fano} gives the irreducible decomposition of each $V_{\Ccal(B_{k})}$. Substituting this into Equation~\eqref{dec ste}, we have:
\begin{equation}\label{dec ste 2}
V_{\Ccal(M_{\textup{Steiner}})}=V_{U_{3,8}}\bigcup_{i=1}^{8}V_{\Ccal(M_{\textup{Steiner}}(i))}\bigcup_{j=1}^{14}V_{A_{j}}\bigcup_{k=1}^{56}V_{C_{k}}\bigcup_{l=1}^{28}V_{D_{l}}\bigcup_{r=1}^{8}V_{E_{r}}\bigcup_{s=1}^{56}V_{F_{s}},
\end{equation}
where the matroids $C_{k},D_{l},E_{r}$ and $F_{s}$ are defined as follows:

\begin{itemize}
\item {$C_{k}$ for $k\in [56]$:} These are the matroids obtained from $M_{\textup{Steiner}}$ by making one of its points a loop and another a coloop; see Figure~\ref{fig:9a}.
\item {$D_{l}$ for $l\in [28]$:} These are the matroids obtained from $M_{\textup{Steiner}}$ by making two of its points coloops; see Figure~\ref{fig:9b}.
\item {$E_{r}$ for $r\in [8]$:} These are the rank-three matroids obtained from $M_{\textup{Steiner}}$ by making one of its points a coloop, with the remaining points forming the uniform matroid $U_{2,7}$.
\item {$F_{s}$ for $s\in [56]$:} Consider the matroid obtained from $M_{\textup{Steiner}}$ by making the point $8$ a coloop and identifying the points $\{3,5,6,7\}$. The matroids $F_{s}$ are then those obtained by applying automorphisms of  $M_{\textup{Steiner}}$ to this matroid.
\end{itemize}

\begin{figure}[H]
    \centering
    \begin{subfigure}[b]{0pt}
  \phantomsubcaption
  \label{fig:9a}
\end{subfigure}%
\begin{subfigure}[b]{0pt}
  \phantomsubcaption
  \label{fig:9b}
\end{subfigure}%
\begin{subfigure}[b]{0pt}
  \phantomsubcaption
  \label{fig:9c}
\end{subfigure}
    \includegraphics[width=0.8\textwidth, trim=0 0 0 0, clip]{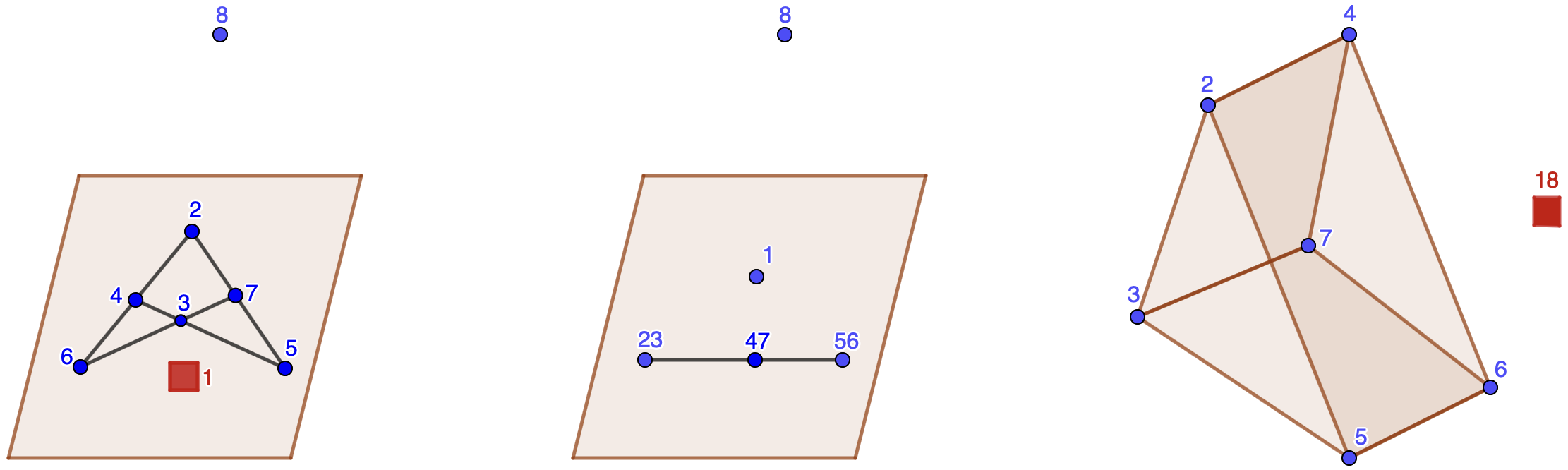}
    \caption{(a) Matroid $C_{1}$ (Left); (b) Matroid $D_{1}$ (Center); (c) Matroid $G_{1}$ (Right).}
    \label{figure 5}
\end{figure}

Observe that $E_{r}\leq U_{3,8}$ for each $r\in [8]$, and for each $s\in [56]$, there exists $j\in [14]$ with $F_{s}\leq A_{j}$. Hence, the varieties $V_{E_{r}}$ and $V_{F_{s}}$ are redundant in~\eqref{dec ste 2}, leading to:
\begin{equation}\label{dec ste 3}
V_{\Ccal(M_{\textup{Steiner}})}=V_{U_{3,8}}\bigcup_{i=1}^{8}V_{\Ccal(M_{\textup{Steiner}}(i))}\bigcup_{j=1}^{14}V_{A_{j}}\bigcup_{k=1}^{56}V_{C_{k}}\bigcup_{l=1}^{28}V_{D_{l}}.
\end{equation}
We also denote by $G_{p}$ for $p\in [28]$ the matroids obtained from $M_{\textup{Steiner}}$ by making two of its points into loops; see Figure~\ref{fig:9c}.  Under the notations above, we have:

\begin{proposition}
   The irreducible decomposition of the circuit variety of $M_{\textup{Steiner}}$ is 
\[V_{\Ccal(M_{\textup{Steiner}})}=V_{U_{3,8}}\bigcup_{j=1}^{14}V_{A_{j}}\bigcup_{k=1}^{56}V_{C_{k}}\bigcup_{l=1}^{28}V_{D_{l}}\bigcup_{p=1}^{28}V_{G_{p}}.\]
\end{proposition}

\begin{proof}
Each of the matroids $G_{p}$ satisfies the conditions of Theorem~\ref{teo ir} (ii), implying that
$V_{\Ccal(G_{p})}\subset V_{G_{p}}\cup V_{U_{3,8}}$.
Applying Algorithm~\ref{algo 2} to each matroid $M_{\textup{Steiner}}(i)$, we find that each of their maximal matroid degenerations is smaller than or equal to some matroid in $\{U_{3,8},A_{j},C_{k},G_{p}\}$. Using this, along with the non-realizability of each matroid $M_{\textup{Steiner}}(i)$, Equation~\eqref{dec ste 3} becomes:
\begin{equation}\label{dec ste 4}
V_{\Ccal(M_{\textup{Steiner}})}=V_{U_{3,8}}\bigcup_{j=1}^{14}V_{A_{j}}\bigcup_{k=1}^{56}V_{C_{k}}\bigcup_{l=1}^{28}V_{D_{l}}\bigcup_{p=1}^{28}V_{G_{p}}.
\end{equation}
Since all matroids in this decomposition are inductively connected, by Theorem~\ref{teo ir}, their corresponding varieties are irreducible. Furthermore, since no two of these matroids are comparable with respect to the weak order, 
the decomposition is non-redundant. Hence, this is the irreducible decomposition of $V_{\Ccal(M_{\textup{Steiner}})}$.
\end{proof}

\subsection{Fano dual}

To introduce the matroid $M_{\text{Fano}}^{\ast}$, which is the focus of this section, we first recall the notion of the dual of a matroid:

\begin{definition}
Let $M$ be a matroid on $[d]$. The \emph{dual matroid} $M^{\ast}$ is the matroid on $[d]$ whose independent sets are precisely those subsets of $[d]$ that are disjoint from some basis of $M$.
\end{definition}

Consider the dual of the {\bf Fano plane}, denoted $M_{\text{Fano}}^{\ast}$. This is the paving matroid of rank four on $[7]$, with the following set of dependent hyperplanes:
\[\{4,5,6,7\},\{2,3,5,6\},\{2,3,4,7\},\{1,3,5,7\},\{1,3,4,6\},\{1,2,4,5\},\{1,2,6,7\}.\]
Applying Algorithm~\ref{algo 2}, we obtain that the set 
$\minabove{\Dual}$ consists of the following matroids; see Figure~\ref{figure 6}, from left to right:
\begin{itemize}
\item[{\rm (i)}] The matroids $M_{\text{Fano}}^{\ast}(i)$ for $i\in [7]$.
\item[{\rm (ii)}] A matroid obtained by identifying 3 points outside a dependent hyperplane of $M_{\text{Fano}}^{\ast}$. 
\item[{\rm (iv)}] A matroid that is the direct sum of a quadrilateral set and the trivial rank-one matroid on a single element. 
\item[{\rm (i)}] The uniform matroid $U_{3,7}$.
\end{itemize}

\begin{figure}[H]
    \centering
    \begin{subfigure}[b]{0pt}
  \phantomsubcaption
  \label{fig:10a}
\end{subfigure}%
\begin{subfigure}[b]{0pt}
  \phantomsubcaption
  \label{fig:10b}
\end{subfigure}%
\begin{subfigure}[b]{0pt}
  \phantomsubcaption
  \label{fig:10c}
\end{subfigure}
\begin{subfigure}[b]{0pt}
  \phantomsubcaption
  \label{fig:10d}
\end{subfigure}
    \includegraphics[width=0.8\textwidth, trim=0 0 0 0, clip]{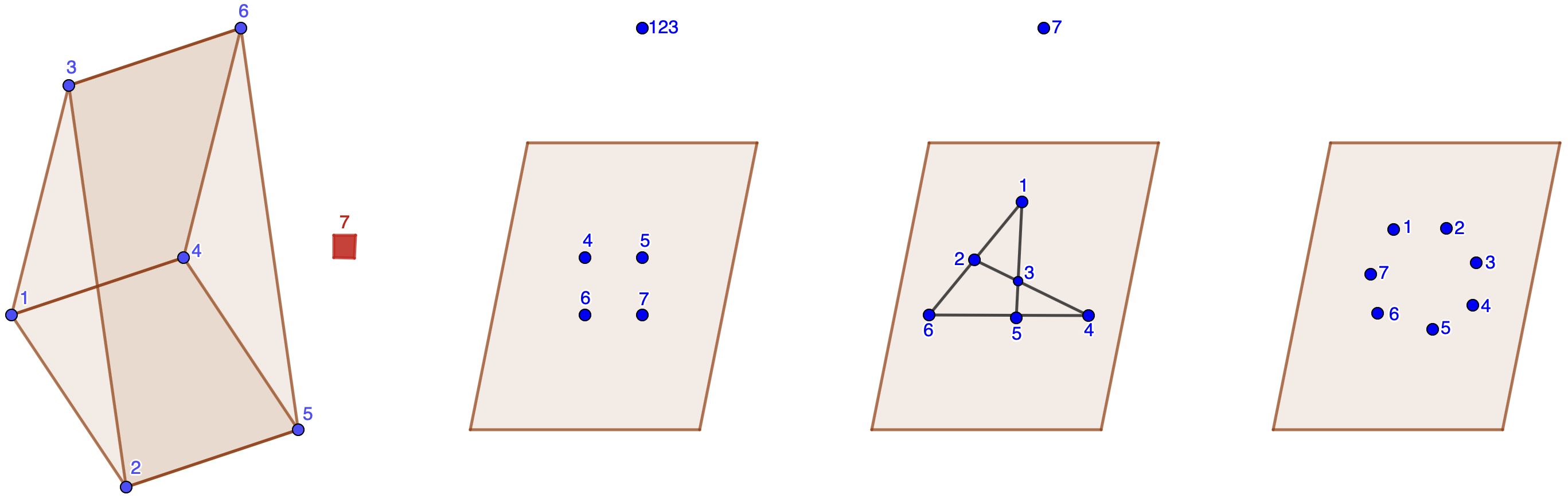}
    \caption{Maximal matroid degenerations  $\Dual(7),A_{1},B_{1},U_{3,7}$ of $\Dual$.}
    \label{figure 6}
\end{figure}

There are seven matroids of type (ii), each corresponding to a dependent hyperplane of $\Dual$, and seven matroids of third (iii), each determined by the choice of a coloop. We denote these matroids by
$A_{j}$ and $B_{k}$, 
for $j,k\in [7]$.  Under the notations above, we have:

\begin{proposition}
   The irreducible decomposition of the circuit variety of $\Dual$ is 
\[V_{\Ccal(\Dual)}=V_{U_{3,7}}\bigcup_{i=1}^{7}V_{\Dual(i)}\bigcup_{j=1}^{7}V_{A_{j}}\bigcup_{k=1}^{7}V_{B_{k}}.\]
\end{proposition}

\begin{proof}
Using the above collection of maximal matroids together with Proposition~\ref{deco circ} and using the non-realizability of $\Dual$, we obtain the following decomposition:
\begin{equation}\label{equ du}
V_{\Ccal(\Dual)}=V_{\Ccal(U_{3,7})}\bigcup_{i=1}^{7}V_{\Ccal(\Dual(i))}\bigcup_{j=1}^{7}V_{\Ccal(A_{j})}\bigcup_{k=1}^{7}V_{\Ccal(B_{k})}.
\end{equation}
Observe that the matroids $U_{3,7}$ and $A_{j}$ are nilpotent. Consequently, by Theorem~\ref{teo ir} (i), their matroid varieties coincide with their circuit varieties. On the other hand, each matroid $\Dual(i)$ satisfies the conditions of Theorem~\ref{teo ir} (ii), implying that
$V_{\Ccal(\Dual(i))}\subset V_{\Dual(i)}\cup V_{U_{3,7}}$.
Furthermore, each matroid $B_{k}$ is the direct sum of the trivial rank-one matroid on a single element and the quadrilateral set $\text{QS}$ from Example~\ref{ej quad}. 
By Example~\ref{ej quad}, we have $V_{\Ccal(\text{QS})}=V_{\text{QS}}\cup V_{U_{2,7}}$,
which implies that
$V_{\Ccal(B_{k})}\subset V_{B_{k}}\cup V_{U_{3,7}}$,
for each $k\in [7]$. 
Using this, Equation~\eqref{equ du} becomes:
\begin{equation*}
V_{\Ccal(\Dual)}=V_{U_{3,7}}\bigcup_{i=1}^{7}V_{\Dual(i)}\bigcup_{j=1}^{7}V_{A_{j}}\bigcup_{k=1}^{7}V_{B_{k}}.
\end{equation*}
All matroids in this decomposition are inductively connected, and thus, by Theorem~\ref{teo ir}, their corresponding varieties are irreducible. Moreover, since no two matroids are comparable with respect to the weak order, 
the decomposition is non-redundant. Therefore, we conclude that this is the irreducible decomposition of $V_{\Ccal(\Dual)}$.
\end{proof}

 \subsection{Dual of $M(K_{3,3})$}\label{k33}

To introduce the matroid $M_{3,3}$, the focus of this section, we begin by recalling the notions of \emph{truncation} and \emph{erection}:

\begin{definition}
Let $M$ ba a matroid of rank $n$ on $[d]$.
\begin{itemize}
 \item The {\em truncation} of $M$ is the matroid of rank $n-1$ on $[d]$, whose independent sets are those of $M$ with size at most $n-1$.\vs
\item A matroid $N$ is called an {\em erection} of $M$, if $M$ is the truncation of $N$. Among all erections of $M$ there exists a unique matroid with the fewest dependent sets, known as the \emph{free erection} of $M$. For further details, see \cite{crapo1970erecting}.
\end{itemize}
\end{definition}

 Consider the graphic matroid $M(K_{3,3})$ associated with the bipartite graph $K_{3,3}$, and let $\K$ denote its dual. This matroid has rank four and contains the following  $3$-circuits:
 \begin{equation}\label{lines}
 \{1,2,3\},\{4,5,6\},\{7,8,9\},\{1,4,7\},\{2,5,8\},\{3,6,9\},
 \end{equation}
 as well as the following collection of $4$-circuits:
 \[\{2,3,4,7\},\{1,3,5,8\},\{1,2,6,9\},\{1,5,6,7\},\{2,4,6,8\},\{3,4,5,9\},\{1,4,8,9\},\{2,5,7,9\},\{3,6,7,8\}.\]
Alternatively, this matroid can be described as the free erection of the $3\times 3$ grid. The $3\times 3$ grid is the point-line configuration on $[9]$ with the set of lines as given in~\eqref{lines}. 
Applying Algorithm~\ref{algo 2}, we obtain that the set $\minabove{\K}$ consists of the following matroids:

\begin{itemize}
\item[{\rm (i)}] The truncation of $\K$, referred to as the $3\times 3$ grid, and denoted by $A$; see Figure~\ref{fig:11a}.
\item[{\rm (ii)}] The matroid obtained from $\K$ by identifying the points $\{5,6,8,9\}$ (see Figure~\ref{fig:11b}), along with all the matroids obtained from this by applying an automorphism of $\K$.
\item[{\rm (iii)}] The matroid obtained from $\K$ by identifying the pairs of points $\{1,4\},\{2,5\}$ and $\{3,6\}$ (see Figure~\ref{fig:11c}), and all the matroids obtained from this by applying an automorphism of~$\K$.
\item[{\rm (iv)}] The matroid obtained from $\K$ by identifying the pairs of points $\{2,3\}$ and $\{4,7\}$, where the points $\{2,3,4,5,6,7,8,9\}$ form a hyperplane (see Figure~\ref{fig:11d}), along with all matroids obtained from this construction through automorphisms of $\K$.
\item[{\rm (v)}] The matroids $\K(i)$ for $i\in [9]$.
\end{itemize}

\begin{figure}[H]
    \centering
       \begin{subfigure}[b]{0pt}
  \phantomsubcaption
  \label{fig:11a}
\end{subfigure}%
\begin{subfigure}[b]{0pt}
  \phantomsubcaption
  \label{fig:11b}
\end{subfigure}%
\begin{subfigure}[b]{0pt}
  \phantomsubcaption
  \label{fig:11c}
\end{subfigure}
\begin{subfigure}[b]{0pt}
  \phantomsubcaption
  \label{fig:11d}
\end{subfigure}
    \includegraphics[width=0.9\textwidth, trim=0 0 0 0, clip]{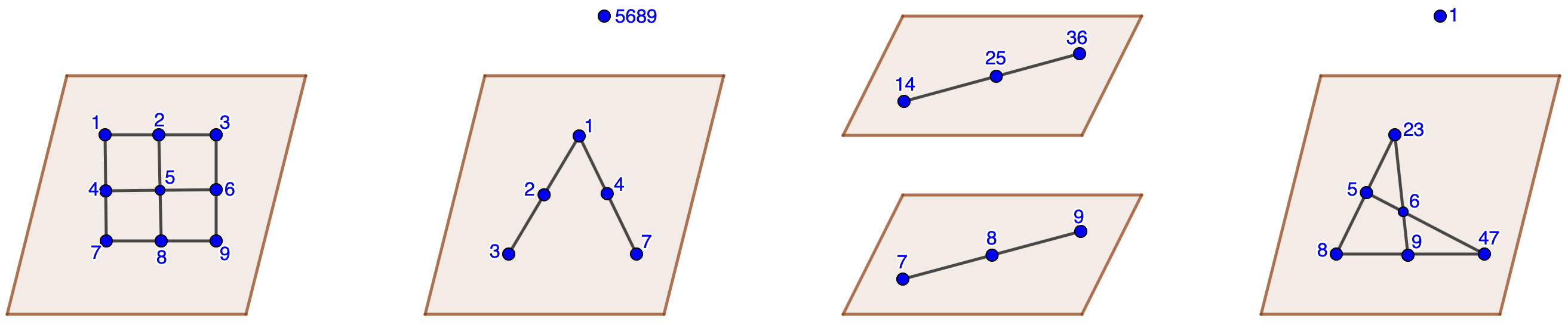}
    \caption{(a) Matroid $A$; (b) Matroid $B_{1}$; (c) Matroid $C_{1}$; (d) Matroid $D_{1}$.}  
    \label{figure 7}
\end{figure}

We denote the matroid of the first type as $A$.
There are $9$ matroids of the second type, $6$ matroids of the third type, and $9$ matroids of the fourth type. We label these matroids as $B_{i},C_{j}$ and $D_{k}$ for $j\in [6]$ and $i,k\in [9]$. 
In the following lemma, we will show how the matroid variety and the circuit variety of the above matroids are related.

\begin{lemma}\label{lema k}
The following statements hold:
\begin{itemize}
\item[{\rm (i)}] $V_{\Ccal(A)}=V_{A}$.
\item[{\rm (ii)}] For $i\in [9]$, $V_{\Ccal(B_{i})}=V_{B_{i}}\subset V_{\K}$.
\item[{\rm (iii)}] For $j\in [6]$, $V_{\Ccal(C_{j})}=V_{C_{j}}\subset V_{\K}$.
\item[{\rm (iv)}] For $k\in [9]$, $V_{\Ccal(D_{k})}\subset V_{\K}\cup V_{A}$.
\item[{\rm (v)}] For $l\in [9]$, the varieties $V_{\Ccal(\K(l))}$ are redundant in~\eqref{deco k}.
\end{itemize}
\end{lemma}

\begin{proposition}
    The irreducible decomposition of the circuit variety of $\K$ is $V_{\Ccal(\K)}=V_{\K}\cup V_{A}$, where $A$ is the $3\times 3$ grid in Figure~\ref{fig:11a}.
\end{proposition}

\begin{proof}
Using the above collection of maximal matroids together with Proposition~\ref{deco circ} and Lemma~\ref{lema k} (i), we obtain the following decomposition:
\begin{equation}\label{deco k}
V_{\Ccal(\K)}=V_{\K}\cup V_{A}\bigcup_{i=1}^{9}V_{\Ccal(B_{i})}\bigcup_{j=1}^{6}V_{\Ccal(C_{j})}\bigcup_{k=1}^{9}V_{\Ccal(D_{k})}\bigcup_{l=1}^{9}V_{\Ccal(\K(l)).}
\end{equation}
Using Lemma~\ref{lema k}, Equation~\eqref{deco k} becomes:
$V_{\Ccal(\K)}=V_{\K}\cup V_{A}$.
Both matroids in this decomposition are inductively connected. Thus, by Theorem~\ref{teo ir}, both matroid varieties are irreducible. Additionally, it is easy to see that the decomposition is non-redundant. Therefore, this is the irreducible decomposition of $V_{\Ccal(\K)}$.
\end{proof}

\section{Maximal matroid degenerations of Steiner systems}\label{furt}
In this section, we focus on matroids arising from Steiner systems and propose a conjecture concerning the structure of their maximal matroid degenerations. Our computations suggest that these degenerations follow a specific and regular pattern. Recall from Definition~\ref{steiner} that every Steiner system $S(n{-}1, k, d)$ gives rise to an $n$-paving matroid on the ground set $[d]$, with the blocks of the system corresponding to the dependent hyperplanes of the matroid.

\medskip

To formulate our conjecture, we begin with the following definition. 

\begin{definition}
Let $M$ be an $n$-paving matroid on $[d]$ associated with a Steiner system $S(n{-}1,k,d)$, denoted by $S$, and let $B$ denote its set of blocks. We define a family of matroids that are (strictly) less than $M$ in the 
weak order poset as follows.
\begin{itemize}
\item For each block $b \in B$, define $M_b$ to be the matroid of rank $n$ on $[d]$ obtained by collapsing (or identifying) all elements outside of $b$. More precisely, $M_b$ is the matroid whose set of circuits of size at most $n$ is given by 
$\textstyle{\binom{b}{n} \cup \binom{[d] \setminus b}{2}}$.\vs

\item For each element $i \in [d]$, let $B_i$ denote the set of blocks in $B$ that contain $i$. Then $\{b \setminus \{i\} : b \in B_i\}$ forms a Steiner system of type $S(n{-}2, k{-}1, d{-}1)$ on the ground set $[d] \setminus \{i\}$. Let $\widetilde{M_i}$ denote the $(n{-}1)$-paving matroid associated with this Steiner system. We define $M_i$ to be the direct sum
\[
M_i := \{i\} \oplus \widetilde{M_i},
\]
where $\{i\}$ is the rank-one matroid on the singleton $\{i\}$.
\end{itemize}
\end{definition}

\begin{example}
Consider the Steiner system $S(3,4,8)$ from \S\ref{uniq}. The matroids $M_b$ and $M_i$ for $b \in B$, $i \in [8]$ are shown in Figures~\ref{fig:8a} and~\ref{fig:8b}, respectively. 
\end{example}
\begin{example}\label{ex steiner}
We present two classical families of Steiner systems:
\begin{itemize}
\item A finite \emph{projective plane} $\mathbb{P}G(2,q)$ of order $q$ corresponds to a Steiner system of type $S(2, q+1, q^{2} + q + 1)$, where the blocks are the lines of the plane.
\item A finite \emph{affine plane} $\mathbb{A}G(2,q)$ of order $q$ corresponds to a Steiner system of type $S(2, q, q^{2})$. An affine plane of order $q$ can be obtained from a projective plane of the same order by deleting a single line along with all the points incident to it.
\end{itemize}
\end{example}

Applying our algorithm for computing maximal matroid degenerations to the family of matroids associated with projective and affine planes yields the following result.

\begin{theorem}
Consider the following spaces:
\begin{itemize}
\item The projective planes $\mathbb{P}G(2,2)$ and $\mathbb{P}G(2,3)$ of orders two and three. 
\item The affine planes $\mathbb{A}G(2,3)$ and $\mathbb{A}G(2,4)$ of orders three and four.
\end{itemize}
Let $S$ denote any of the associated Steiner systems, and let $B$ represent its set of blocks. Furthermore, let $M$ denote the matroid constructed from $S$ on the ground set $[d]$. Then, the set $\minabove{M}$ consists of the following collection of matroids:
\begin{itemize}
\item The matroids $M_{b}$ for $b \in B$.
\item The matroids $M_{i}$ for $i \in [d]$.
\item The matroids $M(i)$ for $i \in [d]$.
\item The uniform matroid $U_{2,d}$.
\end{itemize}
\end{theorem}

From this result, we propose the following conjecture, which  provides insights into the decomposition of circuit varieties of the affine or projective planes.

\begin{conjecture}
Let $M$ be a matroid constructed from an affine or projective plane. Then, the set $\minabove{M}$ consists precisely of the collection of matroids $\{M_{b}, M_{i}, M(i), U_{2,d}\}$.
\end{conjecture}

\section{Appendix}\label{appen}

In this section, we develop methods for verifying redundancy among matroid varieties and provide proofs for the lemmas that were stated without proof in Section~\ref{examples}.

\subsection{Main tool for verifying redundancy}

In this subsection, we present the main tool for investigating the remaining  Question~\ref{question} that we recall below. 

\begin{questione}[\ref{question}]
Given two realizable matroids $M$ and $N$ on the same ground set, under what conditions does the inclusion $V_N \subseteq V_M$ hold between their associated varieties?
\end{questione}

A necessary condition for the inclusion $V_N \subseteq V_M$ to hold is that $N \leq M$, 
as established in~\textup{\cite[Lemma~7.2]{liwski2025minimal}}. To introduce the main tool we use in analyzing this question, we begin with the following definition.

\begin{definition}\normalfont
Let $M$ be a matroid of rank $n$ on $[d]$. 
The  projective realization space $\mathcal{R}(M)$ of $M$ is the set of all the tuples of points $\gamma=(\gamma_{1},\ldots,\gamma_{d})$ in $\mathbb{P}^{n-1}$ that satisfy:
\begin{equation*}
\{\gamma_{i_{1}},\ldots,\gamma_{i_{k}}\}\ \text{are linearly dependent} \Longleftrightarrow \{i_{1},\ldots,i_{k}\} \ \text{is dependent in $M$}.
\end{equation*}
The {\em moduli space} $\mathcal{M}(M)$ of $M$ is defined as the quotient of $\mathcal{R}(M)$ by the action of the projective general linear group $\text{PGL}_{n}(\CC)=\text{GL}_{n}(\CC)/\CC^{\ast}$. The {\em matroid stratum} $\text{Gr}(M,\CC)$ is defined as the quotient of $\Gamma_{M}$ by the action of the general linear group $\text{GL}_{n}(\CC)$.
\end{definition}

Suppose that $M$ contains a circuit of size $n+1$, which we may assume without loss of generality to be $\{1, \ldots, n+1\}$. Then, each isomorphism class in $\mathcal{M}(M)$ admits a unique representative $\gamma \in \mathcal{R}(M)$ satisfying the following condition:
\begin{equation}\label{ecua matr}(\gamma_{1},\gamma_{2},\ldots,\gamma_{n},\gamma_{n+1})=
(e_{1},e_{2},\ldots,e_{n},e_{1}+\ldots+e_{n}),
\end{equation}
where $\{e_{1},\ldots,e_{n}\}$ denotes the canonical basis of $\CC^{n}$. Consequently,  
$\mathcal{M}(M)$ can be characterized as the set of all tuples of points  $\gamma=(\gamma_{1},\ldots,\gamma_{d})$ in $\mathbb{P}^{n-1}$ such that:
\begin{center}
 Equation~\eqref{ecua matr} holds for $(\gamma_{1},\ldots,\gamma_{n+1})$,
\quad \text{and}\quad $\gamma \in \mathcal{R}(M)$.
\end{center}

To describe $\text{Gr}(M,\mathbb{C})$, we fix a reference basis $\lambda \in \mathcal{B}(M)$, which we assume, without loss of generality, to be $\{1, \ldots, n\}$. For each isomorphism class in $\text{Gr}(M,\mathbb{C})$, there exists a unique representative $\gamma \in \Gamma_M$ such that
$(\gamma_1, \ldots, \gamma_n) = (e_1, \ldots, e_n)$.
Thus, the matroid stratum $\text{Gr}(M,\mathbb{C})$ can be characterized as the set of all tuples of vectors $\gamma =(\gamma_1, \ldots, \gamma_d)$ in $\mathbb{C}^n$ satisfying the following conditions:
\begin{center} $(\gamma_1, \ldots, \gamma_n)= (e_1, \ldots, e_n)$ \quad\text{and}\quad
$\gamma \in \Gamma_M$.
\end{center}

We recall the following result from \cite{Fatemeh6}, which will be used in addressing Question~\ref{question}.

\begin{theorem}[\textup{\cite[Theorem~4.15]{Fatemeh6}}]\label{poli}
Let $M$ be an inductively connected matroid. Then, there exists $d\in \mathbb{N}$, and an open subset $U\subset \CC^{d}$ such that $\Gamma_{M}\cong U$.
\end{theorem}

Moreover, by applying Procedure~1 from \cite{Fatemeh6}, we can explicitly determine the polynomials defining the open set $U$ in Theorem~\ref{poli}. This procedure can also be adapted to describe $\mathcal{M}(M)$ and $\text{Gr}(M, \CC)$.

\begin{example}\label{exa}
Consider the paving matroid $M$ of rank four on $[8]$ with the set of dependent hyperplanes:
$\{\{1,2,3,4\},\{3,4,5,6\},\{5,6,7,8\},\{1,2,7,8\}\}$.
The permutation $w=(1,3,5,7,2,4,6,8)$ satisfies the conditions of Definition~\ref{induc}, verifying that $M$ is inductively connected. 
Following \textup{\cite[Procedure~(1)]{Fatemeh6}}, we find that the moduli space $\mathcal{M}(M)$ is composed by all matrices of the form:
\begin{equation}\label{matri}
\begin{pmatrix}
1 & 0 & 0 & 0 & 1 & x & x & xz^{-1}\\
0 & 1 & 0 & 0 & 1 & y & z & 1\\
0 & 0 & 1 & 0 & 1 & 1 & w & 1\\
0 & 0 & 0 & 1 & 1 & 1 & 1 & s
\end{pmatrix},
\end{equation}
where the columns correspond to the points $(1,3,5,7,2,4,6,8)$, respectively, and all minors associated with bases are nonzero.
\end{example}

\begin{remark}\label{rem}
Theorem~\ref{poli} can be utilized to describe the moduli spaces of more general classes of matroids. For instance, consider the matroid $A$ from Subsection~\ref{Vamos}, which is a paving matroid of rank four on $[8]$ with set of dependent hyperplanes:
\[\{\{1,2,3,4\},\{3,4,5,6\},\{5,6,7,8\},\{7,8,1,2\},\{3,4,7,8\},\{1,2,5,6\}\}.\]
This is the matroid obtained from the matroid $M$ in Example~\ref{exa} by adding the dependencies $\{3,4,7,8\}$ and $\{1,2,5,6\}$. To describe $\mathcal{M}(A)$, we note that these additional dependencies impose conditions on the minors of the matrix in~\eqref{matri}, specifically requiring them to vanish for the specified quadruples. Imposing this vanishing condition is equivalent to setting $z=1$. As a result, the moduli space $\mathcal{M}(A)$ consists of all matrices of the form:
\begin{equation*}
\begin{pmatrix}
1 & 0 & 0 & 0 & 1 & x & x & x\\
0 & 1 & 0 & 0 & 1 & y & 1 & 1\\
0 & 0 & 1 & 0 & 1 & 1 & w & 1\\
0 & 0 & 0 & 1 & 1 & 1 & 1 & s
\end{pmatrix},
\end{equation*}
where the columns correspond to the points $(1,3,5,7,2,4,6,8)$, respectively, and all minors associated with bases are nonzero.
\end{remark}

In the next section, we will implicitly apply the same approach as in this remark to derive explicit descriptions of similar moduli spaces.

\subsection{Proofs of lemmas from Section~\ref{examples}}

This subsection presents the proofs of the lemmas from Section~\ref{examples}. In all these proofs we will use the notion of an {\em infinitesimal motion}, which we will now define.

\begin{definition} 
An {\em infinitesimal motion} refers to a perturbation that can be made arbitrarily small. Typically, we aim to show that a given element $x$ lies in the closure of a set $S$. Rather than explicitly stating that, for every $\epsilon>0$, there exists a perturbation of $x$ of distance at most $\epsilon$ landing in $S$, we will simply say that an infinitesimal motion (or infinitesimal perturbation) can be applied to $x$ to obtain an element of $S$.
\end{definition}

\begin{proof}[{\bf Proof of Lemma~\ref{lema red}.}]
By Remark~\ref{rem}, the set 
$\mathcal{M}(A)$ consists of all matrices of the form:
\begin{equation}\label{matrix 5}
\begin{pmatrix}
1 & 0 & 0 & 0 & 1 & x & x & x\\
0 & 1 & 0 & 0 & 1 & y & 1 & 1\\
0 & 0 & 1 & 0 & 1 & 1 & w & 1\\
0 & 0 & 0 & 1 & 1 & 1 & 1 & s
\end{pmatrix},
\end{equation}
where the columns correspond to the points $(1,3,5,7,2,4,6,8)$, respectively, and all minors associated with bases are nonzero.

(i) To prove the claim, we show that $V_{\Ccal(E_{2})}\subset V_{A}\cup V_{\Ccal(C_{2})}$. The argument for the matroid $E_{1}$ follows analogously. Consider the submatroid $E_{2}\rq$ of $E_{2}$ on $\{1,2,3,4,7,8\}$, with the dependent hyperplanes
$\{1,2,3,4\}$, $\{3,4,7,8\}$ and $\{1,2,7,8\}$.
Since $E_{2}\rq$ satisfies the conditions of Theorem~\ref{teo ir} (ii), we have $V_{\Ccal(E_{2}\rq)}=V_{E_{2}\rq}\cup V_{U_{3,6}}$, and so  $V_{\Ccal(E_{2})}\subset V_{E_{2}}\cup V_{\Ccal(C_{2})}$.

To prove the claim, it suffices to show that $V_{E_{2}}\subset V_{A}$. We will show that for any $\gamma \in V_{E_{2}}$, its vectors can be infinitesimally perturbed to obtain $\widetilde{\gamma}\in V_{A}$. By applying Theorem~\ref{poli}, we find that $\mathcal{M}(E_{2})$ consists of all matrices of the form:

\begin{equation*}\label{matrix 4}
\begin{pmatrix}
1 & 0 & 0 & 0 & 1 & x & 0 & x\\
0 & 1 & 0 & 0 & 1 & y & 0 & 1\\
0 & 0 & 1 & 0 & 1 & 1 & 1 & 1\\
0 & 0 & 0 & 1 & 1 & 1 & 0 & s
\end{pmatrix},
\end{equation*}
where the columns correspond to the points $(1,3,5,7,2,4,6,8)$, respectively, and all minors associated with bases are nonzero. Thus, by applying a projective transformation, we may assume that $\gamma$ is of this form. By choosing $\epsilon$ infinitesimally close to $0$ and perturbing the vector $\gamma_{6}=(0,0,1,0)$ to $(\epsilon x, \epsilon, 1, \epsilon)$, we obtain a tuple of vectors $\widetilde{\gamma}$ as in~\eqref{matrix 5}, which represents a realization of $A$. 

\smallskip
(ii) To show the inclusion in~\eqref{fi}, we will show that
\begin{equation}\label{f1}
V_{\Ccal(F_{1})}\subset V_{A}\cup V_{\Ccal(E_{1})},
\end{equation} 
where $F_{1}$ is the rank-four matroid in Figure~\ref{figure 2} (Right), characterized by the following circuits of sizes three and four:
\[\Ccal_{3}(F_{1})=\{\{1,3,4\},\{1,5,6\},\{1,7,8\}\},\quad \text{and} \quad \Ccal_{4}(F_{1})=\{\{5,6,7,8\},\{3,4,5,6\},\{3,4,7,8\}\}.\] It is sufficient to establish the claim for $F_{1}$ since the argument extends analogously to the remaining matroids $F_{i}$. To prove the inclusion in~\eqref{f1}, we show that any $\gamma \in V_{\Ccal(F_{1})}$ lies in the union on the right-hand side of the equation. We consider the following cases:

\smallskip
{\bf Case~1.} Suppose $\gamma_{1}=0$. Then $\{\gamma_{1},\gamma_{2}\}$ is dependent, implying $\gamma \in V_{\Ccal(E_{1})}$.

\smallskip
{\bf Case~2.} Suppose $\gamma_{1}\neq 0$. In this case, we observe that the vectors in $\gamma$ can be infinitesimally perturbed such that the set  $\{\gamma_{1},\gamma_{3},\gamma_{5},\gamma_{7},\gamma_{2}\}$ forms a frame in $\CC^{4}$ while preserving the tuple of vectors in $V_{\Ccal(F_{1})}$. Therefore, we may assume without loss of generality that this set of vectors constitutes a frame in $\CC^{4}$. By applying a suitable projective transformation, we can assume that $\gamma$ is of the form:

\begin{equation*}\label{matris}
\begin{pmatrix}
1 & 0 & 0 & 0 & 1 & x_{1} & x_{3} & x_{5}\\
0 & 1 & 0 & 0 & 1 & x_{2} & 0 & 0\\
0 & 0 & 1 & 0 & 1 & 0 & x_{4} & 0\\
0 & 0 & 0 & 1 & 1 & 0 & 0 & x_{6}
\end{pmatrix},
\end{equation*}
where the columns correspond to the points $(1,3,5,7,2,4,6,8)$, respectively. Moreover, by applying a small perturbation to the values $x_{i}$ we can assume that $\gamma$ realizes $F_{1}$. Consider values $\epsilon_{1},\epsilon_{2},\epsilon_{3}$ infinitesimally close to zero, chosen such that
\[\frac{\epsilon_{1}}{x_{1}}=\frac{\epsilon_{2}}{x_{3}}=\frac{\epsilon_{3}}{x_{5}},\]
and denote this common value by $\epsilon$. Using these parameters, we perturb the vectors of $\gamma$ to obtain a new tuple of vectors $\widetilde{\gamma}$, represented by the following matrix:
\begin{equation}\label{matris 2}
\begin{pmatrix}
1 & 0 & 0 & 0 & 1 & x_{1} & x_{3} & x_{5}\\
0 & 1 & 0 & 0 & 1 & x_{2} & \epsilon_{2} & \epsilon_{3}\\
0 & 0 & 1 & 0 & 1 & \epsilon_{1} & x_{4} & \epsilon_{3}\\
0 & 0 & 0 & 1 & 1 & \epsilon_{1} & \epsilon_{2} & x_{6}
\end{pmatrix},
\end{equation}
where the columns correspond to the points $(1,3,5,7,2,4,6,8)$. This collection represents an infinitesimal perturbation of $\gamma$. Rescaling the last three columns of the matrix in~\eqref{matris 2} by the scalars $\epsilon_{1}^{-1},\epsilon_{2}^{-1},\epsilon_{3}^{-1}$, we obtain a matrix that realizes the same matroid as $\widetilde{\gamma}$:
\begin{equation*}\label{matris 3}
\begin{pmatrix}
1 & 0 & 0 & 0 & 1 & \epsilon^{-1} & \epsilon^{-1} & \epsilon^{-1}\\
0 & 1 & 0 & 0 & 1 & x_{2}\epsilon_{1}^{-1} & 1 & 1\\
0 & 0 & 1 & 0 & 1 & 1 & x_{4}\epsilon_{2}^{-1} & 1\\
0 & 0 & 0 & 1 & 1 & 1 & 1 & x_{6}\epsilon_{3}^{-1}
\end{pmatrix},
\end{equation*}
where the columns correspond to the points $(1,3,5,7,2,4,6,8)$, respectively. This matrix has the same structure as the one in \eqref{matrix 5}, hence representing a realization of $A$. Thus, 
 $\widetilde{\gamma}\in \Gamma_{A}$. Since $\widetilde{\gamma}$ represents an infinitesimal motion of $\gamma$, it follows that $\gamma\in V_{A}$.

\smallskip
(iii) To prove the inclusion in~\eqref{bi}, it is sufficient to prove that
\begin{equation}\label{bi 2}
V_{\Ccal(B_{2})}\subset V_{A}\bigcup_{j=1}^{2}V_{\Ccal(C_{j})}\bigcup_{k=1}^{2}V_{\Ccal(D_{k})}\bigcup_{l=1}^{2}V_{\Ccal(E_{l})}\bigcup_{r=1}^{8}V_{\Ccal(F_{r})}.
\end{equation}
The argument for $B_{1}$ follows analogously, so proving the claim for $B_{2}$ will suffice. To establish the inclusion in~\eqref{bi 2}, we show that any $\gamma\in V_{\Ccal(B_{2})}$ lies in the union on the right-hand side of this equation. The proof is divided into the following cases:

\smallskip
{\bf Case~1.} Suppose that among the pairs of vectors $\{\gamma_{1},\gamma_{2}\},\{\gamma_{5},\gamma_{6}\},\{\gamma_{3},\gamma_{4}\},\{\gamma_{7},\gamma_{8}\}$, at least one is dependent. In such cases, we have
$\textstyle{\gamma\in \bigcup_{i=1}^{2}V_{\Ccal(D_{i})}\bigcup_{j=1}^{2}V_{\Ccal(E_{j})}}$,
as desired. 

\smallskip
{\bf Case~2.} Suppose one of the following conditions holds:
\[\rank\{\gamma_{1},\gamma_{2},\gamma_{5},\gamma_{6},\gamma_{7},\gamma_{8},\gamma_{3}\}\leq 3, \quad \text{or} \quad \rank\{\gamma_{1},\gamma_{2},\gamma_{5},\gamma_{6},\gamma_{7},\gamma_{8},\gamma_{4}\}\leq 3.\]
We may assume without loss of generality that the first condition is satisfied.

\smallskip
{\bf Case~2.1.} Suppose the following triples of vectors have rank at most two: 
\begin{equation}\label{triples}
\{\gamma_{3},\gamma_{1},\gamma_{2}\},\{\gamma_{3},\gamma_{5},\gamma_{6}\}, \{\gamma_{3},\gamma_{7},\gamma_{8}\}.
\end{equation}
 In this case, the matroid associated with $\gamma$ is smaller than or equal to $F$, where $F$ is the rank-four matroid defined by the following circuits of sizes three and four:
\[\Ccal_{3}(F)=\{\{3,1,2\},\{3,5,6\},\{3,7,8\}\},\quad \text{and} \quad \Ccal_{4}(F)=\{\{5,6,7,8\},\{1,2,5,6\},\{1,2,7,8\}\}.\]
Since $\gamma\in V_{\Ccal(F)}$ and $F$ is one of the matroids $F_{i}$, $\gamma$ belongs to 
 the right-hand side of~\eqref{bi 2}.

\smallskip
{\bf Case~2.2.} Suppose at least one of the triples of vectors in~\eqref{triples} has rank three. Without loss of generality, assume this triple is $\{\gamma_{1},\gamma_{2},\gamma_{3}\}$. Since $\rank\{\gamma_{1},\gamma_{2},\gamma_{3}\}=3$ and $\{\gamma_{1},\gamma_{2},\gamma_{3},\gamma_{4}\}$ is dependent, it follows that $\rank \{\gamma_{i}:i\in [8]\}\leq 3$. Consequently, 
$\gamma\in V_{U_{3,8}}$, implying that $\gamma$ belongs to the union on the right-hand side of~\eqref{bi 2}.

\smallskip
{\bf Case~3.} If 
neither of the previous cases applies, then the vectors $\{\gamma_{1},\gamma_{2},\gamma_{5},\gamma_{6},\gamma_{7},\gamma_{8}\}$ span a hyperplane $H\subset \CC^{3}$ with $\gamma_{3},\gamma_{4}\not \in H$. Additionally, the quadruples of vectors \[\{\gamma_{3},\gamma_{4},\gamma_{1},\gamma_{2}\},\quad \{\gamma_{3},\gamma_{4},\gamma_{5},\gamma_{6}\},\quad \{\gamma_{3},\gamma_{4},\gamma_{7},\gamma_{8}\}\] span hyperplanes
$H_{1},H_{2},H_{3}\neq H$ in $\CC^{4}$, respectively. To prove that $\gamma$ lies in the right-hand side of~\eqref{bi 2}, we will show that $\gamma \in V_{A}$. To do so, we will see that $\gamma$ can be infinitesimally perturbed to produce a tuple of vectors that realizes $A$.

\smallskip
{\bf Claim.} The vectors of $\gamma$ can be infinitesimally perturbed to produce $\widetilde{\gamma}\in V_{\Ccal(B_{2})}$, where $\{\widetilde{\gamma}_{3},\widetilde{\gamma}_{4},\widetilde{\gamma}_{1},\widetilde{\gamma}_{5},\widetilde{\gamma}_{7}\}$ forms a frame of $\CC^{4}$.

\smallskip
To prove the claim, note that since $\gamma_{3}, \gamma_{4} \notin H$, we may infinitesimally perturb the vectors $\gamma_{1}, \gamma_{5}, \gamma_{7}$ within $H$ to obtain  $\widetilde{\gamma}_{1},\widetilde{\gamma}_{5},\widetilde{\gamma}_{7}$ such that $\{\gamma_{3},\gamma_{4},\widetilde{\gamma}_{1},\widetilde{\gamma}_{5},\widetilde{\gamma}_{7}\}$ forms a frame of $\CC^{4}$. Let $H_{1}\rq,H_{2}\rq,H_{3}\rq $ denote the hyperplanes spanned by the perturbed triples
\[\{\gamma_{3},\gamma_{4},\widetilde{\gamma}_{1}\},\{\gamma_{3},\gamma_{4},\widetilde{\gamma_{5}}\},\{\gamma_{3},\gamma_{4},\widetilde{\gamma_{7}}\},\]
which represent an infinitesimal motion of $H_{1},H_{2},H_{3}$, respectively. The vectors of $\widetilde{\gamma}$ are completed as follows:
\begin{itemize}
\item $(\widetilde{\gamma}_{3},\widetilde{\gamma}_{4})=(\gamma_{3},\gamma_{4})$.
\item The vectors $\widetilde{\gamma}_{2},\widetilde{\gamma}_{6},\widetilde{\gamma}_{8}$ are chosen to lie in the subspaces $H\cap H_{1},H\cap H_{2},H\cap H_{3}$, respectively, and are selected to be infinitesimally close to $\gamma_{2},\gamma_{6},\gamma_{8}$.
\end{itemize}
It is easy to verify that $\widetilde{\gamma}$ satisfies the required conditions.

\smallskip
Since $\widetilde{\gamma}$ is an infinitesimal motion of $\gamma$, it suffices to show that $\widetilde{\gamma}\in V_{A}$. Given that $\{\widetilde{\gamma}_{3},\widetilde{\gamma}_{4},\widetilde{\gamma}_{1},\widetilde{\gamma}_{5},\widetilde{\gamma}_{7}\}$ forms a frame of $\CC^{4}$, by applying a suitable projective transformation, we may assume that $\widetilde{\gamma}$ is represented by the following matrix:
\begin{equation*}\label{matris 31}
\begin{pmatrix}
1 & 0 & 0 & 0 & 1 & x_{1} & x_{3} & x_{5}\\
0 & 1 & 0 & 0 & 1 & 0 & 0 & 0\\
0 & 0 & 1 & 0 & 1 & x_{2} & x_{4} & x_{5}\\
0 & 0 & 0 & 1 & 1 & x_{2} & x_{3} & x_{6}
\end{pmatrix},
\end{equation*}
where the columns correspond to the points $(1,3,5,7,4,2,6,8)$. Since the pairs of vectors $\{\widetilde{\gamma}_{1},\widetilde{\gamma}_{2}\},\{\widetilde{\gamma}_{5},\widetilde{\gamma}_{6}\}$ and $\{\widetilde{\gamma}_{7},\widetilde{\gamma}_{8}\}$ are linearly independent, it follows that $x_{2},x_{3},x_{5}\neq 0$. By rescaling the last three columns of the matrix by  $x_{1}^{-1},x_{2}^{-1},x_{3}^{-1}$, we obtain a matrix realizing the same matroid as $\widetilde{\gamma}$:
\begin{equation*}\label{matris 32}
\begin{pmatrix}
1 & 0 & 0 & 0 & 1 & x_{1}x_{2}^{-1} & 1 & 1\\
0 & 1 & 0 & 0 & 1 & 0 & 0 & 0\\
0 & 0 & 1 & 0 & 1 & 1 & x_{4}x_{3}^{-1} & 1\\
0 & 0 & 0 & 1 & 1 & 1 & 1 & x_{6}x_{5}^{-1}
\end{pmatrix},
\end{equation*}
where the columns correspond to the points $(1,3,5,7,4,2,6,8)$. We select $\epsilon$ infinitesimally close to $0$ and perturb $\widetilde{\gamma}$ infinitesimally to obtain $\gamma\rq$, represented by the following matrix:
\begin{equation*}\label{matris 33}
\begin{pmatrix}
1 & 0 & 0 & 0 & 1 & x_{1}x_{2}^{-1} & 1 & 1\\
0 & 1 & 0 & 0 & 1 & \epsilon & \epsilon & \epsilon\\
0 & 0 & 1 & 0 & 1 & 1 & x_{4}x_{3}^{-1} & 1\\
0 & 0 & 0 & 1 & 1 & 1 & 1 & x_{6}x_{5}^{-1}
\end{pmatrix},
\end{equation*}
where the columns correspond to the points $(1,3,5,7,4,2,6,8)$.
This matrix has the same structure as the one in \eqref{matrix 5}, up to exchanging the roles of the pairs $\{1,2\}$ and $\{3,4\}$. Hence, it represents a realization of $A$. Thus, 
 $\gamma \rq \in \Gamma_{A}$. Since $\gamma \rq$ represents an infinitesimal motion of $\gamma$, it follows that $\gamma\in V_{A}$. 
\end{proof}

\begin{proof}[{\bf Proof of Lemma~\ref{prop}.}] The inclusion $\supset$ is clear. To prove the reverse inclusion, let $\gamma$ be a tuple of vectors in $V_{C(M)}$. This tuple defines a matroid  $N(\gamma)\leq M$, for which $\gamma \in \Gamma_{N(\gamma)}$. We proceed by considering three cases:

\begin{itemize}
\item If $N(\gamma)=M$, then $\gamma \in V_{M}$.

\item If $N(\gamma)<M$ and $\Ccal_{i}(N)\supsetneq \Ccal_{i}(M)$ for some $i\in \{1,2,3\}$, then there exists $N\in \minposet{\cup_{i=1}^{3}\ABCD{i}}$ such that $N(\gamma)< N$, implying $\gamma \in V_{\Ccal(N)}$.

\item If $N(\gamma)< M$ and $\Ccal_{i}(N)= \Ccal_{i}(M)$ for each $i\in \{1,2,3\}$, then there exists $N\in \ABCD{4}$ such that $N(\gamma)=N$, implying $\gamma \in V_{N}$.
\end{itemize}
Therefore, the reverse inclusion is proven, completing the proof.
\end{proof}

\begin{proof}[{\bf Proof of Lemma~\ref{sec lem}.}] Let $N\in \mathcal{S}_{4}(M_{\text{V\'{a}mos}})$. By definition, $N$ is a paving matroid that satisfies $N<M_{\textup{V\'{a}mos}}$. To prove the lemma, we will prove the following inclusion: 
\begin{equation}\label{incl}
V_{N}\subset V_{A}\bigcup_{i=1}^{2}V_{\Ccal(B_{i})}\bigcup_{j=1}^{2}V_{\Ccal(C_{j})}.
\end{equation}
We consider the following cases:

\smallskip
{\bf Case~1.} Suppose two dependent hyperplanes in $M_{\text{V\'{a}mos}}$ are contained in the same dependent hyperplane of $N$. In this case, one of the following occurs:
\begin{itemize}
\item $\rank_{N}\{1,2,5,6,3,4\}\leq 3$.
\item $\rank_{N}\{1,2,5,6,7,8\}\leq 3$.
\item $\rank_{N}\{3,4,7,8,5,6\}\leq 3$.
\item $\rank_{N}\{3,4,7,8,1,2\}\leq 3$.
\end{itemize}
These cases imply that $N\leq B_{1}$, $N\leq B_{2}$, $N\leq C_{1}$ or $N\leq C_{2}$, respectively. Consequently, for the associated circuit varieties, we have $\textstyle{V_{\Ccal(N)}\subset \bigcup_{i=1}^{2}V_{\Ccal(B_{i})}\bigcup_{j=1}^{2}V_{\Ccal(C_{j})}},$
which implies the desired inclusion in \eqref{incl}.

\smallskip
{\bf Case~2.} Suppose there are no pairs of dependent hyperplanes in $M_{\text{V\'{a}mos}}$ that collapse into the same dependent hyperplane in  $N$. If $\{1,2,5,6\}$ is not a dependent set in $N$, then $N$ is not realizable, and the inclusion holds trivially. Now suppose $\{1,2,5,6\}$ is dependent in $N$, implying $N\leq A$.
By applying Theorem~\ref{poli} to $\text{Gr}(A,\CC)$, we observe that this space consists of all matrices of the form:

\begin{equation}\label{matrix}
\begin{pmatrix}
1 & 0 & 0 & 0 & p & 1 & 1 & w\\
0 & 1 & 0 & 0 & q & x & x & xw-xp+q\\
0 & 0 & 1 & 0 & r & y & 0 & r\\
0 & 0 & 0 & 1 & 1 & 0 & z & 1
\end{pmatrix},
\end{equation}
where the columns correspond to the points $(1,2,3,5,7,4,6,8)$, respectively, and all minors associated with the bases of $A$ are non-zero. Since $N$ is obtained from $A$ by adding or enlarging dependent hyperplanes, the space $\text{Gr}(N,\CC)$ is characterized by matrices of the same form, with the additional condition that the minors corresponding to the non-bases of $N$ vanish, while those corresponding to the bases of $N$ remain non-zero. To prove the inclusion in~\eqref{incl}, we will show that $V_{N}\subset V_{A}$. To see this, we will demonstrate that any $\gamma\in \Gamma_{N}$ can be infinitesimally perturbed to produce a tuple of vectors $\widetilde{\gamma}\in \Gamma_{A}$. Consider $\gamma\in \Gamma_{N}$. By applying a suitable linear transformation, we can assume that $\gamma$ takes the form of~\eqref{matrix}, with specific values
$\widetilde{p},\widetilde{q},\widetilde{r},\widetilde{x},\widetilde{y},\widetilde{z},\widetilde{w} \in \CC$.

Now, examine all minors of~\eqref{matrix} that correspond to the non-bases of $A$. These minors are polynomials in the variables  $p,q,r,x,y,z,w$, and they are non-zero. By infinitesimally perturbing $\widetilde{p},\widetilde{q},\widetilde{r},\widetilde{x},\widetilde{y},\widetilde{z},\widetilde{w}$ to ensure that these minors do not vanish, we obtain a tuple $\widetilde{\gamma}\in \Gamma_{A}$. This completes the proof.
\end{proof}

\begin{proof}[{\bf Proof of Lemma~\ref{not real}.}] We will show that the matroid $M_{\text{Steiner}}(5)$ is not realizable; the other cases follow by analogous arguments. Let $N$ denote this matroid, defined on the ground set $\{1,2,3,4,6,7,8\}$ with the following collection of dependent hyperplanes:
\[\{1,2,4,8\},\{3,4,6,8\},\{2,6,7,8\},\{1,3,7,8\},\{1,4,6,7\},\{1,2,3,6\},\{2,3,4,7\}.\]
Suppose by contrary that there exists $\gamma \in \Gamma_{N}$. Applying a suitable linear transformation, we may assume that $\gamma$ is represented by the matrix:
\begin{equation}\label{matrix 2}
\begin{pmatrix}
1 & 0 & 0 & 0 & x_{1} & x_{5} & x_{9} \\
0 & 1 & 0 & 0 & x_{2} & x_{6} & x_{10} \\
0 & 0 & 1 & 0 & x_{3} & x_{7} & x_{11} \\
0 & 0 & 0 & 1 & x_{4} & x_{8} & x_{12} 
\end{pmatrix},
\end{equation}
where the columns correspond to the points $(1,3,4,7,2,6,8)$, respectively. Since the sets $\{1,3,7,8\},\{1,4,7,6\},\{2,3,4,7\}$ are dependent in $N$, it follows that $x_{1}=x_{6}=x_{11}=0$. Furthermore, as the sets $\{1,3,4,2\},\{1,3,4,6\}$ and $\{1,3,5,8\}$ are independent in $N$, the corresponding minors are nonzero, which implies $x_{4},x_{8},x_{12}\neq 0$. By rescaling the last three columns, we may assume $x_{4}=x_{8}=x_{12}=1$. Additionally, since $\{1,2,3,6\}$ is dependent in $N$, the vanishing of the corresponding minor implies $x_{3}=x_{7}$. Thus, $\gamma$ takes the form:
\begin{equation}\label{matrix 21}
\begin{pmatrix}
1 & 0 & 0 & 0 & 0 & x & y \\
0 & 1 & 0 & 0 & q & 0 & z \\
0 & 0 & 1 & 0 & r & r & 0 \\
0 & 0 & 0 & 1 & 1 & 1 & 1 
\end{pmatrix},
\end{equation}
where the columns correspond to the points $(1,3,4,7,2,6,8)$.
The dependencies $\{1,2,4,8\},\{3,4,6,8\},\{2,6,7,8\}$ in $N$ yield the following vanishing conditions:
\[\det \begin{pmatrix}
q & z\\
1 & 1
\end{pmatrix}=
\det \begin{pmatrix}
x & y\\
1 & 1
\end{pmatrix}=
\det \begin{pmatrix}
0 & x & y\\
q & 0 & z\\
r & r & 0
\end{pmatrix}=0.
\]
The first two equalities give $z=q$ and $y=x$, which, when substituted into the third, yields $xqr=0$. Thus $x=0$, $q=0$ or $r=0$. 

\smallskip  
\noindent  
\textbf{Case $x=0$:} The vanishing of the minor on $\{3,4,6,7\}$ contradicts its independence. 

\smallskip  
\noindent  
\textbf{Case $q=0$:} The minor on $\{1,2,4,7\}$ vanishes, violating its independence in $N$.

\smallskip  
\noindent  
\textbf{Case $r=0$:} The vanishing of the minor on $\{1,3,6,7\}$ yields a similar contradiction.

\smallskip  
\noindent  
In all cases, we reach a contradiction, implying that $N$ is not realizable.
\end{proof}

\begin{proof}[{\bf Proof of Lemma~\ref{lema k}.}]

(i) The matroid $A$ corresponds to the point-line configuration defined by the set of lines in~\eqref{lines}. This configuration is commonly referred to as the $3 \times 3$ grid. As shown in~\cite{Fatemeh3}, the matroid and circuit varieties associated to $A$ coincide.

\smallskip
(ii) Consider the matroid $N$ obtained from $\K$ by identifying the points $\{5,6,8,9\}$, one of the matroids $B_i$. As the argument is analogous for all $B_i$, it suffices to show that 
\begin{equation}\label{vn2}
V_{\Ccal(N)}=V_{N}\subset V_{\K}.
\end{equation}
Since $N$ satisfies the conditions of Theorem~\ref{teo ir} (i), 
$V_{\Ccal(N)}=V_{N}$. To prove that $V_{N}\subset V_{\K}$, we will see that any $\gamma\in \Gamma_{N}$ can be perturbed infinitesimally to obtain a tuple of vectors realizing $\K$. Following the same procedure as in Theorem~\ref{poli}, we find that any realization of $N$ is projectively equivalent to the realization given by the matrix:

\begin{equation*}\label{matrix 2111}
\begin{pmatrix}
1 & 0 & 0 & 0 & 0 & 1 & 1 & 0 & 0 \\
0 & 1 & 0 & 0 & 0 & 1 & 0 & 0 & 0\\
0 & 0 & 1 & 0 & 0 & 0 & 1 & 0 & 0\\
0 & 0 & 0 & 1 & 1 & 0 & 0 & 1 & 1 
\end{pmatrix},
\end{equation*}
where the columns correspond to 
$(1,2,4,5,9,3,7,6,8)$. We select $\epsilon$ infinitesimally close to zero, and perturb $\gamma$ infinitesimally to obtain $\widetilde{\gamma}$, represented by the matrix:
\begin{equation*}\label{matrix 2112}
\begin{pmatrix}
1 & 0 & 0 & 0 & \epsilon & 1 & 1 & 0 & 0 \\
0 & 1 & 0 & 0 & \epsilon & 1 & 0 & 0 & \epsilon\\
0 & 0 & 1 & 0 & \epsilon & 0 & 1 & \epsilon & 0\\
0 & 0 & 0 & 1 & 1 & 0 & 0 & 1 & 1 
\end{pmatrix},
\end{equation*}
where the columns correspond to the points $(1,2,4,5,9,3,7,6,8)$. It is easy to verify that $\widetilde{\gamma}$ represents a realization of $\K$. Hence, it follows that $\gamma \in V_{\K}$.

\smallskip
(iii) Consider the matroid $N$ obtained from $\K$ by identifying the points lying within each pair $\{4,7\},\{5,8\}$ and $\{6,9\}$, which is one of the matroids $C_{i}$. Since the argument is analogous for all matroids $C_{i}$ it suffices to show that 
\begin{equation}\label{vn}
V_{\Ccal(N)}=V_{N}\subset V_{\K}.
\end{equation}
Since $N$ satisfies the conditions of Theorem~\ref{teo ir} (i), $V_{\Ccal(N)}=V_{N}$. To prove that $V_{N}\subset V_{\K}$, we will see that any $\gamma\in \Gamma_{N}$ can be perturbed infinitesimally to obtain a tuple of vectors realizing $\K$. Following the same procedure as in Theorem~\ref{poli}, we find that any realization of $N$ is projectively equivalent to the realization given by the matrix:

\begin{equation*}\label{matrix 2113}
\begin{pmatrix}
1 & 0 & 1 & 0 & 1 & 1 & 1 & 1 & 0 \\
0 & 1 & 0 & 1 & 1 & 1 & 0 & 1 & 1\\
0 & 0 & 1 & 0 & 1 & 0 & 1 & 1 & 0\\
0 & 0 & 0 & 1 & 1 & 0 & 0 & 1 & 1 
\end{pmatrix},
\end{equation*}
where the columns correspond to 
$(1,2,4,5,9,3,7,6,8)$. We select $\epsilon$ infinitesimally close to zero, and perturb $\gamma$ infinitesimally to obtain $\widetilde{\gamma}$, represented by the matrix:
\begin{equation*}\label{matrix 2114}
\begin{pmatrix}
1 & 0 & \epsilon & 0 & 1 & 1 & 1 & \epsilon & 0 \\
0 & 1 & 0 & \epsilon & 1 & 1 & 0 & \epsilon & 1\\
0 & 0 & 1 & 0 & 1 & 0 & 1 & 1 & 0\\
0 & 0 & 0 & 1 & 1 & 0 & 0 & 1 & 1 
\end{pmatrix},
\end{equation*}
where the columns correspond to the points $(1,2,4,5,9,3,7,6,8)$. It is easy to verify that $\widetilde{\gamma}$ represents a realization of $\K$. Hence, it follows that $\gamma \in V_{\K}$.

\smallskip
(iv) Let $N$ denote the matroid obtained from $\K$ by identifying the points within the pairs $\{2,3\}$ and $\{4,7\}$, where the points $\{2,3,4,5,6,7,8,9\}$ form a hyperplane. This matroid corresponds to one of the $D_{i}$, and is depicted in Figure~\ref{fig:11d}. Since the argument is analogous for all $D_{i}$, it suffices to show that 
\begin{equation}\label{vn 4}
V_{\Ccal(D_{i})}\subset V_{\K}\cup V_{A}.
\end{equation}
Consider the submatroid $N\rq$ of $N$ induced on the points $\{2,4,5,6,8,9\}$. Observe that $N\rq$ is isomorphic to the matroid $\text{QS}$ from Example~\ref{ej quad}. From this example, we know that 
$V_{\Ccal(N\rq)}=V_{N\rq}\cup V_{U_{2,6}}$. 
Thus,
$V_{\Ccal(N)}\subset V_{N}\cup V_{\Ccal(A)}=V_{N}\cup V_{A}$. 
To complete the proof, we show that $V_{N}\subset V_{\K}$. Specifically, we show that any realization $\gamma\in \Gamma_{N}$ can be infinitesimally perturbed to obtain a realization of $\K$. Following the procedure outlined in Theorem~\ref{poli}, any realization of $N$ is projectively equivalent to the realization given by the matrix:

\begin{equation*}\label{matrix 21131}
\begin{pmatrix}
0 & 0 & 0 & 0 & 0 & 0 & 0 & 0 & 1 \\
0 & 1 & 0 & 1 & 1 & 1 & 0 & 0 & 0\\
0 & 1 & 1 & 0 & 0 & 0 & 1 & 1 & 0\\
1 & 1 & 1 & 1 & 0 & 0 & 0 & 0 & 0 
\end{pmatrix},
\end{equation*}
where the columns correspond to the points $(5,9,6,8,2,3,4,7,1)$. To construct a realization of $\K$, we choose $z_{1},z_{2},z_{3},z_{4}\in \CC$ infinitesimally close to zero, and perturb $\gamma$ infinitesimally to obtain $\widetilde{\gamma}$, represented by the matrix
\begin{equation*}\label{matrix 211312}
\begin{pmatrix}
z_1 & z_2 & z_3 & z_4 & z_4-z_1 & z_2-z_3 & z_3-z_1 & z_2-z_4 & 1 \\
0 & 1 & 0 & 1 & 1 & 1 & 0 & 0 & 0\\
0 & 1 & 1 & 0 & 0 & 0 & 1 & 1 & 0\\
1 & 1 & 1 & 1 & 0 & 0 & 0 & 0 & 0 
\end{pmatrix},
\end{equation*}
where the columns correspond to the points $(5,9,6,8,2,3,4,7,1)$. It is easy to verify that $\widetilde{\gamma}$ represents a realization of $\K$. Hence, it follows that $\gamma \in V_{\K}$.

\smallskip
(v) The proof follows by applying the same arguments as in the previous lemmas.
\end{proof}

\vspace{-3mm}

\bibliographystyle{abbrv}

\bibliography{Citation}

\end{document}